\documentclass[3p,sort&compress]{elsarticle}
\journal{Journal of Computational Physics (Elsevier)}

\usepackage{framed,multirow}
\usepackage{latexsym}
\usepackage{amsmath}
\usepackage{amssymb}
\usepackage{amsthm}
\usepackage{mathtools}
\usepackage{pifont}
\usepackage{enumitem}
\usepackage{booktabs}
\usepackage{subcaption}
\usepackage[dvipsnames]{xcolor}
\usepackage{yfonts}
\usepackage[scr=boondoxo,scrscaled=1.05]{mathalfa}
\usepackage{appendix}
\usepackage{enumitem}
\usepackage{url}
\usepackage{ifthen}
\usepackage{tensor}
\usepackage[normalem]{ulem}

\theoremstyle{plain}
\newtheorem{theorem}{Theorem}

\theoremstyle{definition}
\newtheorem{definition}{Definition}
\newtheorem{example}{Example}

\theoremstyle{remark}
\newtheorem{remark}{Remark}

\makeatletter
\def\ps@pprintTitle{%
     \let\@oddhead\@empty
     \let\@evenhead\@empty
     \def\@oddfoot
       {\hbox to \textwidth%
        {\ifnopreprintline\relax\else
        \@myfooterfont%
         \ifx\@elsarticlemyfooteralign\@elsarticlemyfooteraligncenter%
           \hfil\@elsarticlemyfooter\hfil%
         \else%
         \ifx\@elsarticlemyfooteralign\@elsarticlemyfooteralignleft%
           \@elsarticlemyfooter\hfill{}%
         \else%
         \ifx\@elsarticlemyfooteralign\@elsarticlemyfooteralignright%
           {}\hfill\@elsarticlemyfooter%
         \else%
               {\bfseries Updated version} of preprint submitted to \ifx\@journal\@empty%
                 Elsevier%
            \else\@journal\fi\hfill\@date\fi%
         \fi%
         \fi%
         \fi%
         }
       }%
     \let\@evenfoot\@oddfoot}
\makeatother

\begin{document}

\begin{frontmatter}
\title{Flow-driven spectral chaos (FSC) method for simulating long-time dynamics of arbitrary-order non-linear stochastic dynamical systems}

\author[1]{Hugo Esquivel}
\ead{hesquive@purdue.edu}

\author[1]{Arun Prakash}
\ead{aprakas@purdue.edu}

\author[2]{Guang Lin\fnref{fn2}}
\ead{guanglin@purdue.edu}

\address[1]{Lyles School of Civil Engineering, Purdue University, 550 W Stadium Ave, West Lafayette IN 47907, USA}
\address[2]{Department of Mathematics, School of Mechanical Engineering, Purdue University, 150 N University St, West Lafayette IN 47907, USA}

\fntext[fn2]{Department of Statistics, Department of Earth, Atmospheric, and Planetary Sciences (by courtesy), Purdue University.}

\begin{abstract}
Uncertainty quantification techniques such as the time-dependent generalized polynomial chaos (TD-gPC) use an adaptive orthogonal basis to better represent the stochastic part of the solution space (aka random function space) in time.
However, because the random function space is constructed using tensor products, TD-gPC-based methods are known to suffer from the curse of dimensionality.
In this paper, we introduce a new numerical method called the \emph{flow-driven spectral chaos} (FSC) which overcomes this curse of dimensionality at the random-function-space level.
The proposed method is not only computationally more efficient than existing TD-gPC-based methods but is also far more accurate.
The FSC method uses the concept of \emph{enriched stochastic flow maps} to track the evolution of a finite-dimensional random function space efficiently in time.
To transfer the probability information from one random function space to another, two approaches are developed and studied herein.
In the first approach, the probability information is transferred in the mean-square sense, whereas in the second approach the transfer is done \emph{exactly} using a new theorem that was developed for this purpose.
The FSC method can quantify uncertainties with high fidelity, especially for the long-time response of stochastic dynamical systems governed by ODEs of arbitrary order.
Six representative numerical examples, including a nonlinear problem (the Van-der-Pol oscillator), are presented to demonstrate the performance of the FSC method and corroborate the claims of its superior numerical properties.
Finally, a parametric, high-dimensional stochastic problem is used to demonstrate that when the FSC method is used in conjunction with Monte Carlo integration, the curse of dimensionality can be overcome altogether.
\end{abstract}


\end{frontmatter}

\section*{Note}

This is an updated version of the journal article.
Errata can be found on the last page of this document.

\section*{Highlights}

\begin{itemize}\setlength\itemsep{0em}
\item FSC can propagate and quantify with high fidelity the long-time response of SODEs.
\item FSC is partially insensitive to the curse of dimensionality.
\item FSC is computationally more efficient than TD-gPC for the same accuracy level.
\item FSC can transfer the probability information exactly at every instant of time.
\item FSC can be used in the resolution of high-dimensional stochastic dynamical systems.
\end{itemize}

\section{Introduction}

The need for quantifying uncertainties for real-world applications arises in different fields, such as in physics, engineering, economics, sociology, etc.
In structural engineering, for example, the source of random variability can arise from: material properties, imperfections in geometry, loading scenarios, boundary conditions, etc.
Once this random variability is identified, it can be characterized mathematically using random variables, stochastic processes or, more generally, random fields in space and time.
Various methods for solving stochastic differential equations have been proposed to date, among which we mention: Monte Carlo-based methods \cite{kroese2013handbook,fishman2013monte,rubinstein2016simulation}, collocation-based methods \cite{babuvska2007stochastic,nobile2008sparse,foo2008multi,foo2010multi,teckentrup2015multilevel}, perturbation-based methods \cite{collins1969eigenvalue,liu1986probabilistic,liu1986random,kleiber1992stochastic}, operator-based methods \cite{shinozuka1988response,yamazaki1988neumann,deodatis1991weightedI,deodatis1991weightedII}, and spectral-based methods \cite{xiu2010numerical,le2010spectral,sullivan2015introduction}.
This work is based on the spectral approach, for which we give a short historical overview below.

The \emph{polynomial chaos} (PC), as originally introduced by Wiener in 1938 \cite{wiener1938homogeneous} and then further extended in \cite{wiener1943discrete,wiener1958nonlinear}, is a spectral-based method to model stochastic processes with (independent) Gaussian random variables.
Roughly speaking, the method uses Hermite polynomials as the underlying basis to expand a stochastic process in the space of random functions, and it is considered to be an extension of the theory of nonlinear functionals developed by Volterra in 1913 \cite{volterra1913leccons} for stochastic systems.
Such an expansion is known to be convergent in the mean-square sense for stochastic processes with a finite second moment, thanks to the Cameron-Martin theorem \cite{cameron1947orthogonal}.
Therefore, these processes are also termed \emph{second-order stochastic processes} in the literature.
Even though the PC method was applied to solve different stochastic problems at the time, it was later recognized that it suffered from non-uniform convergence for systems with non-Gaussian random variables.
Lucor et al.~\cite{lucor2001spectral} demonstrated that, under Wiener's framework of Hermite functionals, the convergence rate of systems subjected to Gaussian input is exponential but substantially slower otherwise \cite{xiu2002wiener,field2003new}.
Nonetheless, steady progress was made between the 1950s and 1980s towards generalizing Wiener's ideas for systems with non-Gaussian inputs (e.g.~\cite{ito1951multiple,ogura1972orthogonal,segall1976orthogonal,schetzen1981nonlinear,engel1982multiple}).

In the early 1990s, Ghanem and Spanos \cite{ghanem1990polynomial,ghanem1991stochastic} developed a method in the context of stochastic finite elements.
The method essentially uses Wiener's theory on polynomial chaos to decompose a second-order stochastic process into deterministic and non-deterministic parts.
The non-deterministic part of the process can then be treated as an element of a Hilbert space, and thus, be approximated by its Galerkin projection onto a subspace spanned by a finite number of Hermite polynomials.
Because the subspace still needs to be spanned by Hermite polynomials, this method is only capable of achieving exponential convergence for stochastic systems involving Gaussian random variables.
Yet, the method was successfully applied by several researchers in the branch of continuum mechanics, including solid and fluid mechanics, in problems displaying random variability in their definition (e.g.~\cite{ghanem1993stochastic,ghanem1999stochastic,knio2001stochastic,le2002stochastic}).

In 2002, Xiu and Karniadakis \cite{xiu2002wiener} introduced the \emph{generalized polynomial chaos} (gPC) method to overcome the issue of convergence rate of the PC method.
By employing an orthogonal basis from the Askey family---but concordant with the measure defined in the probability space---, they showed that a process expanded with such a basis leads to exponential convergence to the solution.
Thus, in the years that followed, the gPC method was demonstrated to be capable of solving a wider number of stochastic problems found in practice (e.g.~\cite{xiu2003modeling,asokan2005using,knio2006uncertainty,najm2009uncertainty}).
The method, however, was later found not to be suitable for problems that feature strong nonlinear dependencies over the probability space as time progresses.
For example, for long-time integration of stochastic dynamical systems, the gPC method fails to capture the probability moments accurately because the probability distribution of the solution changes significantly with time.
In 2005, Wan and Karniadakis \cite{wan2005adaptive} developed the \emph{multi-element generalized polynomial chaos} (ME-gPC) method to account for these nonlinear dependencies in time, such as the ability to handle stochastic discontinuities and long-time response of stochastic dynamical systems on-the-fly.
The key idea of ME-gPC is to adaptively decompose the random space into elements until a pre-specified threshold for the relative error in variance is reached.
Then, a stochastic spectral expansion is used on each random element to push the system's state forward in time.
This process is repeated every time the threshold is exceeded during the simulation.
The ME-gPC method and its variants (e.g.~\cite{wan2006beyond,wan2006multi,foo2008multi,ma2009adaptive,foo2010multi,jakeman2013minimal}) have been proved to be capable of solving numerous problems in engineering and sciences (e.g.~\cite{wan2006long,agarwal2009domain,kewlani2009multi,prempraneerach2010uncertainty,oladyshkin2011concept,kewlani2012polynomial}).

The dynamically orthogonal PC (DO-PC) is another approach used for uncertainty quantification.
It was formulated by Sapsis and Lermusiaux \cite{sapsis2009dynamically} in 2009 to study the response of continuous stochastic dynamical systems more effectively.
In this approach, the time rate of change of the spatio-temporal function space is ensured to be kept orthogonal to itself as the simulation proceeds.
This condition, called the dynamically orthogonal (DO) condition, is enforced at every time step to derive an exact, closed set of evolution equations in time.
With additional restrictions on the form of the solution representation, the DO-PC approach can recover both the POD (Proper Orthogonal Decomposition) method \cite{papoulis1965random,holmes2012turbulence} and the gPC method.
Since its inception, the DO-PC has undergone further modifications and extensions to broaden its range of applications (e.g.~\cite{choi2013convergence,ueckermann2013numerical,cheng2013dynamicallyI,cheng2013dynamicallyII}).
An error analysis for the DO-PC method can be found in \cite{musharbash2015error}.

In 2010, the time-dependent gPC (TD-gPC) method was proposed by Gerritsma et al.~\cite{gerritsma2010time} to address the issue of long-time integration in the gPC method.
This was motivated by the fact that the probability distribution of the solution changes with time, which in turn requires that the random basis (of the solution space) is frequently updated during the simulation to ensure that the mean-square error is kept orthogonal to the discretized random function space.
To keep the computational cost low, the random basis is adaptively updated whenever a preset threshold value is met during the simulation.
Whenever this threshold value is met, a new set of orthogonal polynomials is generated from the monomials of the system’s state for use in subsequent time steps of the simulation.
Heuveline and Schick \cite{heuveline2014hybrid} modified the TD-gPC method (mTD-gPC) to account for stochastic dynamical systems governed by second-order ODEs, and in doing so they also improved the accuracy of the method.
In mTD-gPC, the stochastic part of the solution space is spanned (at the \emph{reset times}) by performing a full tensor product between an evolving random function space (that depends upon the evolution of the system’s state) and the original random function space (which is spanned according to the gPC method).
However, since both TD-gPC and mTD-gPC use tensor products to construct a suitable random basis, they both suffer from the curse of dimensionality because the number of basis vectors in both these approaches grows considerably fast with the dimensionality of the probability space (and sometimes this growth may be exponential if not addressed well). 
Heuveline and Schick \cite{heuveline2014hybrid} also developed a multi-element version of the mTD-gPC method called the \emph{hybrid generalized polynomial chaos} as a means to keep the dimensionality of the random function space relatively low on each random element.

More recently, Luchtenburg et al.~\cite{luchtenburg2014long} developed a method for long-time uncertainty propagation in dynamical systems.
The method consists of approximating the intermediate short-time flow maps by spectral polynomial bases, so that the system's long-time flow map is constructed via a flow map composition.
These short-time flow maps are represented by low-degree polynomial bases to account for the stretching and folding effect caused by the evolution of the system’s state in phase space.
Ozen and Bal \cite{ozen2016dynamical} introduced the dynamical gPC (DgPC) method to quantify uncertainties in the long-time response of stochastic dynamical systems.
The method uses a generalization of the PCE (Polynomial Chaos Expansion) framework to construct a set of orthogonal polynomials from measures that evolve dynamically in time.
They demonstrated that results obtained with DgPC compare well with other standard methods such as Monte Carlo.
However, the method has limited applicability for large stochastic dynamical systems.

In this paper, a novel method called the \emph{flow-driven spectral chaos} (FSC) is proposed to capture the long-time response of stochastic dynamical systems.
The FSC method uses the concept of \emph{enriched} stochastic flow maps to track the evolution of a finite-dimensional random function space efficiently in time.
In this approach, the enriched stochastic flow map of the system is by definition a flow map that pushes forward the first few time derivatives of the solution (including the solution itself) in an augmented random phase space.
Unlike mTD-gPC (or gPC), the number of basis vectors needed to construct the orthogonal bases in FSC does not grow with the dimensionality of the probability space.
Therefore, the FSC method does not suffer from the curse of dimensionality at the random-function-space level.
However, as with all spectral-based methods, it does suffer from the curse of dimensionality at the random-space level, because the number of quadrature points needed to compute the inner products accurately can grow exponentially with the dimensionality of the probability space.
Nevertheless, as we show in Section \ref{sec2StoPro5Dim}, when the FSC method is used in conjunction with Monte Carlo integration to compute the inner products, the curse of dimensionality can be eliminated altogether.
Thus, the FSC method presents a major advance over gPC-based methods since for the same level of accuracy in the solution it is computationally far more efficient.

This paper is organized as follows.
Section \ref{sec2SetNot} introduces the setting and notation used in this manuscript, and then a quick overview of the standard gPC method is provided.
Section \ref{sec2GraSch} discusses the Gram-Schmidt process for random function spaces and also outlines a \emph{new theorem} that has been developed for orthogonalizing a sequence of independent random functions.
This theorem is then utilized in the FSC scheme (Section \ref{sec2FSCmet}) to transfer the probability information of the system's state \emph{exactly} at the current time of the simulation.
Section \ref{sec2StoFloMap} reviews the concept of stochastic flow map, followed by the definition of \emph{enriched} stochastic flow map in Section \ref{sec2EnrStoFloMap}.
In Section \ref{sec2FSCmet} we describe the proposed FSC method in detail using two different approaches (FSC-1 and FSC-2) for the transfer of the probability information.
Six numerical examples are then presented in Section \ref{sec2NumExa}, followed by a discussion of the numerical results in Section \ref{sec2DisNumRes}.
In Section \ref{sec2StoPro5Dim} we solve a parametric, high-dimensional stochastic problem to demonstrate (from a numerical standpoint) that using the FSC method, in conjunction with Monte Carlo integration to compute the inner products, it is possible to overcome the curse of dimensionality at both the random-function-space level and the random-space level---thus eliminating it altogether.
In Appendices \ref{appsec2DerVanderPolOsc} and \ref{appsec2StoPro5Dim} we present in detail the discretization of the two random function spaces needed to simulate the stochasticity of a Van-der-Pol oscillator and the system described in Section \ref{sec2StoPro5Dim} using the spectral approach.
Finally, Appendix \ref{appsec2TimComAnaThe1} presents a comparison between the time-complexity analyses of our new theorem and the traditional Gram-Schmidt process in order to assess the computational cost of both approaches algebraically.

\section{Setting and notation}\label{sec2SetNot}

\paragraph{Spaces} The spaces that we use in this work are defined below.

\begin{definition}[Temporal space]\label{sec2SetNotDef1}
Let the topological space $(\mathfrak{T},\mathcal{O})$ be called \emph{temporal space}, where $\mathfrak{T}=[0,T]$ is a closed interval representing the temporal domain of the system, $T$ is a positive real number symbolizing the duration of the simulation, and $\mathcal{O}=\mathcal{O}_\mathbb{R}\cap\mathfrak{T}$ is the topology on $\mathfrak{T}$ with $\mathcal{O}_\mathbb{R}$ denoting the standard topology over $\mathbb{R}$.
\end{definition}

\begin{remark}
Although this temporal space can be specialized further to be a Hilbert space, in this manuscript we only need the topological structure of it to assist Definition \ref{sec2SetNotDef3} in regard to continuity of functions in time.
In simple terms, this temporal space defines the time interval of interest for running the stochastic simulations.
\end{remark}

\begin{definition}[Random space]\label{sec2SetNotDef2}
Let $(\Omega,\boldsymbol{\Omega},\lambda)$ be a (complete) \emph{probability space}, where $\Omega$ is the sample space, $\boldsymbol{\Omega}\subset 2^\Omega$ is the $\sigma$-algebra on $\Omega$ (aka the collection of events), and $\lambda:\boldsymbol{\Omega}\to[0,1]$ is the probability measure on $\boldsymbol{\Omega}$.
Let $\xi:(\Omega,\boldsymbol{\Omega})\to(\mathbb{R}^d,\mathcal{B}_{\mathbb{R}^d})$ be a measurable function (aka random variable) given by $\xi=\xi(\omega)$, with $\mathcal{B}_{\mathbb{R}^d}$ denoting the Borel $\sigma$-algebra over $\mathbb{R}^d$.
Furthermore, let the measure space $(\Xi,\boldsymbol{\Xi},\mu)$ be called \emph{random space}, where $\Xi=\xi(\Omega)\subset\mathbb{R}^d$ is a set representing the random domain of the system, $\boldsymbol{\Xi}=\mathcal{B}_{\mathbb{R}^d}\cap\Xi$ is the $\sigma$-algebra on $\Xi$, and $\mu:\boldsymbol{\Xi}\to[0,1]$ is the probability measure on $\boldsymbol{\Xi}$ defined by the pushforward of $\lambda$ by $\xi$, that is $\mu=\xi_*(\lambda)$.
Here $d$ symbolizes the dimensionality of the random space.
\end{definition}

\begin{remark}
In addition to the standard definition of a `probability space', we define a `random space' in Definition \ref{sec2SetNotDef2} to address cases where the probability space may be abstract.
The random variable $\xi$ relates these two spaces and aids in computation.
\end{remark}

From these two definitions it is clear that more structure can be added to these spaces; for example, a metric, a norm, an inner product, etc.
However, we opt not to do so herein to keep the above definitions as simple as possible, and more importantly, because they are not needed in this manuscript.

\begin{definition}[Temporal function space]\label{sec2SetNotDef3}
Let $\mathscr{T}(n)=C^n(\mathfrak{T},\mathcal{O};\mathbb{R})$ be a continuous $n$-differentiable function space.
This \emph{temporal function space} is the space of all functions $f:(\mathfrak{T},\mathcal{O})\to(\mathbb{R},\mathcal{O}_\mathbb{R})$ that have continuous first $n$ derivatives on $(\mathfrak{T},\mathcal{O})$.
\end{definition}

\begin{remark}
The temporal function space is thus defined to indicate the level of differentiability that some temporal functions need to possess in Sections \ref{sec2StoFloMap} and \ref{sec2EnrStoFloMap}---especially those concerning with the state of the dynamical system under consideration.
\end{remark}

\begin{definition}[Random function space]\label{sec2SetNotDef4}
Let $\mathscr{Z}=(L^2(\Xi,\boldsymbol{\Xi},\mu;\mathbb{R}),\langle\,\cdot\,,\cdot\,\rangle)$ be a Lebesgue square-integrable space equipped with its standard inner product $\langle\,\cdot\,,\cdot\,\rangle:L^2(\Xi,\boldsymbol{\Xi},\mu;\mathbb{R})\times L^2(\Xi,\boldsymbol{\Xi},\mu;\mathbb{R})\to\mathbb{R}$ given by $\langle f,g\rangle=\int fg\,\mathrm{d}\mu$.
This \emph{random function space} (aka RFS in this manuscript) is the space of all (equivalence classes of) measurable functions $f:(\Xi,\boldsymbol{\Xi})\to(\mathbb{R},\mathcal{B}_\mathbb{R})$ that are square-integrable with respect to $\mu$.
This space is known to form a Hilbert space because it is complete under the metric induced by the inner product.
In addition, let $\{\Psi_j:(\Xi,\boldsymbol{\Xi})\to(\mathbb{R},\mathcal{B}_\mathbb{R})\}_{j=0}^\infty$ be a complete orthogonal basis in $\mathscr{Z}$, such that $\Psi_0(\xi)=1$ for all $\xi\in\Xi$.
\end{definition}

\begin{remark}
In the literature, the `random function space' just defined is also called `random space' to simplify the terminology of the space.
However, in this work, a distinction between the two spaces is needed.
We claim that the FSC method \emph{does not} suffer from the curse of dimensionality at the \emph{random-function-space} level, because the number of basis vectors that we use to span $\mathscr{Z}$ does not depend upon the dimensionality of the random space---in contrast to other spectral methods such as gPC, TD-gPC, etc.~which use tensor products to construct $\mathscr{Z}$.
However, as with all spectral methods the FSC method \emph{does} suffer from the curse of dimensionality at the \emph{random-space} level, because we still have the issue that the higher the dimensionality of the random space is, the more difficult is to compute the inner products accurately.
We emphasize, however, that this is still an open area of research and that there are several numerical techniques available in the literature that deal with this issue, e.g.~\cite{davis2007methods,novak1996high,bungartz2004sparse,leobacher2014introduction}.
In Section \ref{sec2StoPro5Dim}, for example, we show that Monte Carlo integration can be used to address the curse of dimensionality at the random-space level, and that together with the FSC method, it can eliminate the curse of dimensionality of the proposed spectral approach at both random levels.
\end{remark}

From Definition \ref{sec2SetNotDef4} it follows that any function $f\in\mathscr{Z}$ can be represented in a Fourier series of the form:
\begin{equation*}
f=\sum_{j=0}^\infty f^j\Psi_j,
\end{equation*}
where $f^j$ denotes the $j$-th coefficient of the series with the superscript not denoting an exponentiation.

Moreover, let $\Upsilon_{ij}=\langle\Psi_i,\Psi_j\rangle$ be the $(i,j)$-th component of the inner-product tensor associated with the chosen orthogonal basis in $\mathscr{Z}$.
Then, because of the orthogonality property of the basis and the selection of the first basis vector to be identically equal to one ($\Psi_0\equiv 1$), one obtains:
\begin{equation*}
\Upsilon_{ij}=\langle\Psi_i,\Psi_i\rangle\,\delta_{ij}=
\begin{cases}
1 & \text{for $i=j=0$}\\
\langle\Psi_i,\Psi_i\rangle & \text{for $i=j$ with $i,j>0$}\\
0 & \text{otherwise.}
\end{cases}
\end{equation*}
Following the notation and conventions of multilinear and tensor algebra, we note that $\boldsymbol{\Upsilon}=\Upsilon_{ij}\,\Psi^i\otimes\Psi^j:\mathscr{Z}^2\to\mathbb{R}$ is a symmetric tensor of type $(0,2)$ given by $\boldsymbol{\Upsilon}[f,g]=\Upsilon_{ij}f^ig^j$, where $\Psi^i:\mathscr{Z}\to\mathbb{R}$ is the $i$-th dual basis vector in $\mathscr{Z}'$ defined by
\begin{equation*}
\Psi^i[h]:=[\Psi^i, h]=\frac{\langle\Psi_i,h\rangle}{\langle\Psi_i,\Psi_i\rangle}\equiv h^i.
\end{equation*}
Here $[\,\cdot\,,\cdot\,]:\mathscr{Z}'\times\mathscr{Z}\to\mathbb{R}$ represents the dual pairing between $\mathscr{Z}$ and $\mathscr{Z}'$ satisfying the property: $[\Psi^i,\Psi_j]=\delta\indices{^i_j}$ with $\delta\indices{^i_j}$ denoting the Kronecker delta.
The second equality follows from the Riesz representation theorem \cite{rudin1987real}, which means that the map $\Psi^i\mapsto\Psi_i/\langle\Psi_i,\Psi_i\rangle$ is an isometric isomorphism between $\mathscr{Z}'$ and $\mathscr{Z}$.

\begin{definition}[Solution space and root space]\label{sec2SetNotDef5}
Let $\mathscr{U}=\mathscr{T}(n)\otimes\mathscr{Z}$ and $\mathscr{V}=\mathscr{T}(0)\otimes\mathscr{Z}$ be the \emph{solution space} and the \emph{root space} of the system, respectively.
Then, as a result of these definitions, we have: $\mathscr{Z}\subset\mathscr{U}\subset\mathscr{V}$.
\end{definition}

\begin{remark}
These two definitions are used below to relate the solution space and the root space via the partial differential operator $\mathcal{L}$.
They are also used in the manuscript to simplify the notation of these spaces.
\end{remark}

Throughout this paper, we assume that the components of the $d$-tuple random variable $\xi=(\xi^1,\ldots,\xi^d)$ are mutually independent and that the random domain $\Xi$ is a hypercube of $d$ dimensions obtained by performing a $d$-fold Cartesian product of intervals $\bar{\Xi}_i:=\xi^i(\Omega)$.
Letting $\mu^i(\mathrm{d}\xi^i)=:\mathrm{d}\mu^i$ denote the probability measure of $\mathrm{d}\xi^i$ around $\xi^i\in\bar{\Xi}_i$, one can then define the measure in $\mathscr{Z}$ by
\begin{equation*}
\mu=\bigotimes_{i=1}^d\mu^i,\quad\text{or equivalently,}\quad\mathrm{d}\mu\equiv\mu(\mathrm{d}\xi)=\prod_{i=1}^d\mu^i(\mathrm{d}\xi^i)\equiv\mathrm{d}\mu^1\cdots\mathrm{d}\mu^d.
\end{equation*}

\paragraph{Problem statement}

In this work, we consider the following stochastic problem (assumed well-posed).

Find the real-valued stochastic process $u:\mathfrak{T}\times\Xi\to\mathbb{R}$ in $\mathscr{U}$, such that ($\mu$-a.e.):
\begin{subequations}\label{eq2SetNot1000}
\begin{align}
\mathcal{L}[u]=f&\qquad\text{on $\mathfrak{T}\times\Xi$}\label{eq2SetNot1000a}\\
\big\{\mathcal{B}_k[u](0,\cdot\,)=b_k\big\}_{k=1}^n&\qquad\text{on $\{0\}\times\Xi$},\label{eq2SetNot1000b}
\end{align}
\end{subequations}
where $\mathcal{L}:\mathscr{U}\to\mathscr{V}$ is a partial differential operator of order $(n,0)$, $\mathcal{B}_k[\,\cdot\,](0,\cdot\,):\mathscr{U}\to\mathscr{Z}$ is a partial differential operator of order $(n-1,0)$ that upon differentiation evaluates the resulting function at $t=0$, $f:\mathfrak{T}\times\Xi\to\mathbb{R}$ is a function in $\mathscr{V}$ given by $f=f(t,\xi)$, and $b_k:\Xi\to\mathbb{R}$ is a function in $\mathscr{Z}$ given by $b_k=b_k(\xi)$.

The operators $\mathcal{L}$ and $\mathcal{B}_k$ take differentiations only in time and can be, in general, nonlinear.
For the case when $(n,d)=(2,3)$ and $\mathcal{L}$ and $\mathcal{B}_k$ are linear operators, we get: $\xi=(\xi^1,\xi^2,\xi^3)$ and
\begin{gather*}
\mathcal{L}[u](t,\xi)=a_2(t,\xi)\,\ddot{u}(t,\xi)+a_1(t,\xi)\,\dot{u}(t,\xi)+a_0(t,\xi)\,u(t,\xi)\\
\mathcal{B}_1[u](0,\xi)=b_{11}(\xi)\,\dot{u}(0,\xi)+b_{10}(\xi)\,u(0,\xi)\,\,\\
\mathcal{B}_2[u](0,\xi)=b_{21}(\xi)\,\dot{u}(0,\xi)+b_{20}(\xi)\,u(0,\xi),
\end{gather*}
where $a_0,a_1,a_2\in\mathscr{V}$ with $a_2\neq0$, and $b_{10},b_{11},b_{20},b_{21}\in\mathscr{Z}$ such that $b_{10}b_{21}-b_{11}b_{20}\neq 0$.
Observe that in these expressions, $\dot{u}:=\partial_t u$ and $\ddot{u}:=\partial^2_t u$ denote the first and second partial derivatives of $u$ with respect to time.
This example also shows that, for the more general case, the stochasticity of the system can enter via the operators $\mathcal{L}$ and $\mathcal{B}_k$, and the source functions $f$ and $b_k$.

Because $u$ is already assumed to be an element of $\mathscr{U}$ in \eqref{eq2SetNot1000}, the stochastic systems that we are interested in are those whose underlying process is of second-order only.
A stochastic process $u$ is said to be of second-order if its second moment is finite, or equivalently, if $u(t,\cdot\,)\in\mathscr{Z}$ for all $t\in\mathfrak{T}$.

In this sense, since $u\in\mathscr{U}$, it can be represented by the Fourier series:
\begin{equation}\label{eq2SetNot1020}
u(t,\xi)=\sum_{j=0}^\infty u^j(t)\,\Psi_j(\xi),
\end{equation}
where $u^j$ is a temporal function in $\mathscr{T}(n)$ denoting the $j$-th random mode of $u$.
This series, usually referred to as \emph{stochastic spectral expansion} in the literature \cite{le2010spectral,sullivan2015introduction}, will be used herein as the \emph{solution representation} of the underlying process to seek.

It is worth mentioning that if we demand $u$ to be sufficiently smooth in the solution space, especially in $\mathscr{Z}$, the expansion given by \eqref{eq2SetNot1020} will lead to exponential convergence to the solution, since $\{\Psi_j\}_{j=0}^\infty$ is an orthogonal basis with respect to the probability measure $\mu$ in $\mathscr{Z}$.
This particular selection of the basis for the underlying process is known as the \emph{optimal basis}, and it can be obtained by using any orthogonalization technique such as the Gram-Schmidt process \cite{cheney2010linear}.

A system governed by \eqref{eq2SetNot1000} can also be expressed in modeling notation as
\begin{equation}\label{eq2SetNot1000star}
y=\boldsymbol{\mathcal{M}}[u][x]\quad\text{subject to initial condition}\quad \boldsymbol{\mathcal{I}}[u],\tag{\ref{eq2SetNot1000}*}\\
\end{equation}
where $\boldsymbol{\mathcal{M}}[u]:\mathscr{V}^r\to\mathscr{V}^s$ represents the mathematical model of the system defined by \eqref{eq2SetNot1000a}, $x=(x_1,\ldots,x_r):\mathfrak{T}\times\Xi\to\mathbb{R}^r$ is the $r$-tuple input of $\boldsymbol{\mathcal{M}}[u]$, and $y=(y_1,\ldots,y_s):\mathfrak{T}\times\Xi\to\mathbb{R}^s$ is the $s$-tuple output of $\boldsymbol{\mathcal{M}}[u]$ (aka the $s$-tuple observable in physics or the $s$-tuple response in engineering).
In addition, $\boldsymbol{\mathcal{I}}[u]$ represents the initial condition for $\boldsymbol{\mathcal{M}}[u]$ which is given by \eqref{eq2SetNot1000b}.
The objective of this mathematical model is to propagate and quantify the effects of input uncertainty $x$ on system's output $y$.
Note that here $x$ is to be understood as the model's input and not necessarily as the system's input.
Therefore, the components of $x$ might not only include the source function $f$ in \eqref{eq2SetNot1000} but also the coefficients of operator $\mathcal{L}$.

\paragraph{Discretization of random function space (standard gPC method)}

For this work, let us simply consider a $p$-discretization of the random function space $\mathscr{Z}$ as follows.
Let $\mathscr{Z}^{[P]}=\text{span}\{\Psi_j\}_{j=0}^P$ be a finite subspace of $\mathscr{Z}$ with $P+1\in\mathbb{N}_1$ denoting the dimensionality of the subspace, and let $u^{[P]}(t,\cdot\,)$ be an element of $\mathscr{Z}^{[P]}$.

Then, from \eqref{eq2SetNot1020} it follows that:
\begin{equation}\label{eq2SetNot1030}
u(t,\xi)\approx u^{[P]}(t,\xi)=\sum_{j=0}^P u^j(t)\,\Psi_j(\xi),
\end{equation}
provided that $\{\Psi_j\}_{j=0}^\infty$ is well-graded to carry out the approximation of $u$ this way.

If $d$ denotes the dimensionality of the random space, and $p$ is the maximal order polynomial in $\{\Psi_j\}_{j=0}^P$, then the total number of terms that we obtain after expanding \eqref{eq2SetNot1030} can be determined as
\begin{equation}\label{eq2SetNot1030A}
P+1=\binom{d+p}{p}=\frac{(d+p)!}{d!p!}.
\end{equation}
This expression shows that the total number of terms used in \eqref{eq2SetNot1030} grows combinatorially fast as a function of $d$ and $p$, and thus, it suffers to some extent from the \emph{curse of dimensionality}.
In practice, the usefulness of representing the solution with such a construction (i.e.~by means of a \emph{total-order tensor product}) is limited for problems where $d$ and $p$ are less than 10 or so \cite{sullivan2015introduction}.
For higher-dimensional spaces, more general \emph{sparse tensor products} can be utilized to help alleviate better the curse of dimensionality, e.g.~by means of \emph{Smolyak-based tensor products}.
However, for low dimensional spaces, \emph{full tensor products} can still be used whenever $d$ is 2 or 3.
In full tensor products, the total number of terms increases exponentially fast as a function of $d$ and $p$.
That is, $P+1=(p+1)^d$.

\begin{remark}\label{rmk2paper100}
An orthogonal basis in $\mathscr{Z}^{[P]}$ can be constructed as products of univariate orthogonal polynomials in the following way.
Let $\{\Psi_j^{(i)}:\bar{\Xi}_i\to\mathbb{R}\}_{j=0}^\infty$ be an orthogonal basis with respect to $\mu^i$, where $i\in\{1,2,\ldots,d\}$.
These bases are usually chosen to be univariate polynomials along the $i$-th dimension satisfying the condition that $\Psi_0^{(i)}(\xi^i)=1$ for all $\xi^i\in\bar{\Xi}_i$.
Then, one defines:
\begin{equation}\label{eq2SetNot1032}
\Psi_j\equiv\Psi_{\pi(k)}:=\bigotimes_{i=1}^d \Psi_{k_i}^{(i)}=\Psi_{k_1}^{(1)}\otimes\cdots\otimes\Psi_{k_d}^{(d)},
\end{equation}
where $k=(k_1,\ldots,k_d)\in\mathbb{N}_0^d$ is a multi-index with $|k|=k_1+\cdots+k_d$, and $\Psi_{\pi(k)}$ is a function given by
\begin{equation*}
\Psi_{\pi(k)}(\xi)=\prod_{i=1}^d\Psi_{k_i}^{(i)}(\xi^i)=\Psi_{k_1}^{(1)}(\xi^1)\,\cdots\,\Psi_{k_d}^{(d)}(\xi^d).
\end{equation*}
In these expressions, $\pi:\mathbb{N}_0^d\to\mathbb{N}_0$ is an ordering map that sorts the elements in ascending order based on the multi-index degree $|k|$, followed by a reverse-lexicographic ordering for those elements that share the same multi-index degree.
In case of resorting to a \emph{total-order tensor product}, the condition $|k|\leq p$ is enforced in \eqref{eq2SetNot1032} to make $p$ (which is always taken less than $\max |k|$) be the maximal order polynomial in $\{\Psi_j\}_{j=0}^P$.
\end{remark}

\begin{remark}\label{note2paper200}
In Section \ref{sec2FSCmet} we will see that in our FSC method the basis $\{\Psi_j\}_{j=0}^P$ is not constructed by performing a tensor product like in Remark \ref{rmk2paper100}, but by recurring instead to the time derivatives of the solution itself.
This is in contrast to the standard TD-gPC method \cite{gerritsma2010time}, which uses tensor products to construct a basis based on the monomials of the solution, and for which it is known suffers from the curse of dimensionality.
\end{remark}

For notational convenience, expansion \eqref{eq2SetNot1030} will simply be written hereafter as
\begin{equation}\label{eq2SetNot1031}
u(t,\xi)=u^j(t)\,\Psi_j(\xi),
\end{equation}
where a summation sign is implied over the repeated index $j$, and $j\in\{0,1,\ldots,P\}$ unless indicated otherwise.
Observe that the superscript $^{[P]}$ in $u$ was dropped to avoid unnecessary complexity in notation.

Substituting \eqref{eq2SetNot1031} into \eqref{eq2SetNot1000} gives
\begin{subequations}\label{eq2SetNot1040}
\begin{align}
\mathcal{L}[u^j\Psi_j]=f&\qquad\text{on $\mathfrak{T}\times\Xi$}\label{eq2SetNot1040a}\\
\big\{\mathcal{B}_k[u^j\Psi_j](0,\cdot\,)=b_k\big\}_{k=1}^n&\qquad\text{on $\{0\}\times\Xi$}.\label{eq2SetNot1040b}
\end{align}
\end{subequations}

Projecting \eqref{eq2SetNot1040} onto $\mathscr{Z}^{[P]}$ yields a system of $P+1$ ordinary differential equations of order $n$ in the variable $t$, where the unknowns are the random modes $u^j=u^j(t)$ and their first $n-1$ time derivatives:
\begin{subequations}\label{eq2SetNot1050}
\begin{align}
\Psi^i\big[\mathcal{L}[u^j\Psi_j]\big]=\Psi^i[f]&\qquad\text{on $\mathfrak{T}$}\label{eq2SetNot1050a}\\
\big\{\Psi^i\big[\mathcal{B}_k[u^j\Psi_j](0,\cdot\,)\big]=\Psi^i[b_k]\big\}_{k=1}^n&\qquad\text{on $\{0\}$} \label{eq2SetNot1050b}
\end{align}
\end{subequations}
with $i,j\in\{0,1,\ldots,P\}$.
This is the so-called \emph{orthogonal projection} of \eqref{eq2SetNot1040} onto $\mathscr{Z}^{[P]}$, and it ensures that the mean-square error resulting from the finite representation of $u$ using \eqref{eq2SetNot1030} is orthogonal to $\mathscr{Z}^{[P]}$ \cite{ghanem1991stochastic}.

A closer look at \eqref{eq2SetNot1050} indicates that the system of equations that we are dealing with at this point is no longer `stochastic' but `deterministic' since the randomness of the stochastic system has effectively been absorbed by the application of the dual vectors $\{\Psi^i\in\mathscr{Z}'\}_{i=0}^P$.
In other words, system \eqref{eq2SetNot1050} does not depend on the tuple $(t,\xi)$ but only on $t$ at this stage of the analysis.

\paragraph{Discretization of temporal function space}

Because \eqref{eq2SetNot1050} is a system of ordinary differential equations with initial conditions, any suitable time integration method can be used to find its solution at discrete times; giving therefore rise to an $(h,p)$-discretization for $\mathscr{T}(n)$ in general.

\paragraph{Probability moments}

In probability theory, the real-valued expectation, $\mathbf{E}:\mathscr{Z}\to\mathbb{R}$, is a linear map that outputs the expected value of a real-valued random variable, and it is given by:
\begin{equation*}
\mathbf{E}[f]=\int f\,\mathrm{d}\mu.
\end{equation*}
In contrast, the real-valued covariance, $\mathrm{Cov}:\mathscr{Z}^2\to\mathbb{R}$, is a symmetric, bilinear map that measures the joint variability of two real-valued random variables.
It is defined by:
\begin{equation*}
\mathrm{Cov}[f,g]=\mathbf{E}\big[\big(f-\mathbf{E}[f]\big)\big(g-\mathbf{E}[g]\big)\big].
\end{equation*}
These two maps can be used as the building block to construct other maps, such as the variance of $f$ which is defined as $\mathrm{Var}[f]=\mathrm{Cov}[f,f]$.
Higher probability moments (e.g.~skewness, kurtosis, etc.) are not considered in this work.
However, we do so for the sake of brevity and without loss of generality, chiefly because higher probability moments are not guaranteed to exist for second-order stochastic processes.

Now, let $z=y_k$ be the $k$-th component of output $y=\boldsymbol{\mathcal{M}}[u][x]$ (from \eqref{eq2SetNot1000star}).
If $z\in\mathscr{V}$, then it can be expanded with a polynomial chaos similar to the one set forth in \eqref{eq2SetNot1030} to obtain: 
\begin{equation*}
z(t,\xi)\approx z^{[P]}(t,\xi)=\sum_{j=0}^P z^j(t)\,\Psi_j(\xi)\equiv z^j(t)\,\Psi_j(\xi),
\end{equation*}
where $P$ does not need to be the same as in \eqref{eq2SetNot1030}, and the $j$-th random mode of $z$ is given by:
\begin{equation*}
z^j(t)=\frac{\langle\Psi_j,z(t,\,\cdot\,)\rangle}{\langle\Psi_j,\Psi_j\rangle}.
\end{equation*}
This representation of $z$ will allow us to compute the probability moments of interest with minimal computational effort, as demonstrated below.

The expectation of $z$, $\mathbf{E}[z]:\mathfrak{T}\to\mathbb{R}$, is easy to compute and it is given by the first random mode of $z$:
\begin{equation}\label{eq2SetNot3100}
\mathbf{E}[z](t):=\int z(t,\cdot\,)\,\mathrm{d}\mu=z^j(t)\int\Psi_j\,\mathrm{d}\mu=z^j(t)\,\langle\Psi_j,\Psi_0\rangle=z^0(t).
\end{equation}

The autocovariance of $z$, $\mathrm{Cov}[z,z]:\mathfrak{T}^2\to\mathbb{R}$, is defined as:
\begin{align*}
\mathrm{Cov}[z,z](t,s)&:=\mathbf{E}\big[\big(z(t,\cdot\,)-\mathbf{E}[z](t)\big)\big(z(s,\cdot\,)-\mathbf{E}[z](s)\big)\big]&\\
&=\mathbf{E}\big[\big(z^j(t)\,\Psi_j-z^0(t)\big)\big(z^k(s)\,\Psi_k-z^0(s)\big)\big]&\text{with $j,k\in\{0,1,\ldots,P\}$}\\
&=\mathbf{E}\big[z^j(t)\,z^k(s)\,\Psi_j\Psi_k\big], &\text{with $j,k\in\{1,2,\ldots,P\}$}
\end{align*}
and thus, upon further simplification we get
\begin{equation*}
\mathrm{Cov}[z,z](t,s)=\sum_{j=1}^P\sum_{k=1}^P z^j(t)\,z^k(s)\int\Psi_j\Psi_k\,\mathrm{d}\mu=
\sum_{j=1}^P\sum_{k=1}^P z^j(t)\,z^k(s)\,\langle\Psi_j,\Psi_k\rangle=
\sum_{j=1}^P\Upsilon_{jj}\,z^j(t)\,z^j(s).
\end{equation*}

The variance of $z$, $\mathrm{Var}[z]:\mathfrak{T}\to\mathbb{R}_0^+$, is nothing but:
\begin{equation}\label{eq2SetNot3300}
\mathrm{Var}[z](t):=\mathrm{Cov}[z,z](t,t)=
\sum_{j=1}^P\Upsilon_{jj}\,z^j(t)\,z^j(t).
\end{equation}

\paragraph{Probability distributions}

In this paper, we employ four different probability distributions to characterize the stochasticity in the systems defined in Sections \ref{sec2NumExa} and \ref{sec2StoPro5Dim}, namely: uniform, beta, gamma and normal.
Because the measures associated with these distributions are absolutely continuous with respect to the Lebesgue measure, they possess probability density functions, $f:\Xi\to\mathbb{R}^+_0$, given by:
\begin{gather*}
\mathrm{Uniform}\sim f(\xi)=\frac{1}{b-a}\text{ on $\Xi=[a,b]$},\quad
\mathrm{Beta}(\alpha,\beta)\sim f(\xi)=\frac{(\xi-a)^{\alpha-1}\,(b-\xi)^{\beta-1}}{(b-a)^{\alpha+\beta-1}\,\mathrm{B}(\alpha,\beta)}\text{ on $\Xi=[a,b]$}\\
\mathrm{Gamma}(\alpha,\beta)\sim f(\xi)=\frac{\beta^\alpha}{\Gamma(\alpha)}(\xi-a)^{\alpha-1}\exp(-\beta\,(\xi-a))\text{ on $\Xi=[a,\infty)$},\quad\text{and}\quad\\
\mathrm{Normal}(\mu,\sigma^2)\sim f(\xi)=\frac{1}{\sigma\sqrt{2\pi}}\exp\!\left[-\frac{1}{2}\!\left(\frac{\xi-\mu}{\sigma}\right)^{\!2}\right]\text{ on $\Xi=\mathbb{R}$}.
\end{gather*}

\paragraph{Numerical integration}

The numerical integration of an inner product can be carried out at least in two different ways: using \emph{grid-based integration} \cite{davis2007methods,novak1996high,bungartz2004sparse} or \emph{Monte Carlo-based integration} \cite{leobacher2014introduction}.
The difference between the two lies in how the quadrature points are chosen from the domain of the integral.
In the latter, the quadrature points are randomly sampled from the domain to seek an approximate evaluation of the integral, whereas in the former the quadrature points are selected to be the intersecting points of some predefined regular grid.
It is well-known that when this grid is the Gaussian grid associated with the measure, the grid-based integration method produces the most accurate approximation of the integral.

In grid-based integration, we can either use \emph{full grids} or \emph{sparse grids} to perform the numerical evaluation of the integral.
Using one or the other will depend on the level of accuracy we want to achieve and the computational cost we are willing to pay to estimate the numerical value of the integral.
Popular sparse grids based on the work by S.A.~Smolyak \cite{smolyak1963quadrature} deal with the curse of dimensionality well.
However, in situations where the dimensionality of the integral domain is high, Monte Carlo is usually the preferred integration technique since the convergence rate to the sought integral is dimension independent.

For the numerical examples presented in Section \ref{sec2NumExa}, we use the Gaussian quadrature rule based on full grids to approximate the inner products.
The reason behind this choice is that in those numerical examples, the dimensionality of the random space is at most 2.
For this low-dimensional random space, we can heedlessly define a full grid in the random domain to estimate the integrals with high accuracy.
Consequently, the aforementioned inner products are computed with the following expression:
\begin{equation*}
\langle f,g\rangle:=\int fg\,\mathrm{d}\mu\approx\mathcal{Q}^{[Q]}[fg]:=\sum_{i=1}^{Q} f(\xi_i)\,g(\xi_i)\,w_i,
\end{equation*}
where $w_i\in\mathbb{R}^+$ is the quadrature weight associated with the Gaussian quadrature point $\xi_i\in\Xi$, and $Q\in\mathbb{N}_1$ denotes the number of quadrature points involved in approximating the evaluation of the inner product.
Here the quadrature points are selected from the Gaussian grid associated with the measure $\mu$.

\begin{remark}
We note that when $fg$ is a sufficiently smooth integrand, we have:
\begin{equation*}
\mathcal{Q}^{[Q]}[fg]\to\int fg\,\mathrm{d}\mu\quad\text{as}\quad Q\to\infty,
\end{equation*}
and if $fg$ is a polynomial, then there exists a $Q\in\mathbb{N}_1$ such that $\mathcal{Q}^{[Q+j]}[fg]$ evaluates the integral exactly for all $j\in\mathbb{N}_0$.
\end{remark}

Moreover, we use Monte Carlo integration to approximate the inner products that emerge from solving the parametric, high-dimensional stochastic problem described in Section \ref{sec2StoPro5Dim}.
In this case we choose Monte Carlo integration because the dimensionality of the random space is up to 10.

\section{The Gram-Schmidt process for random function spaces}\label{sec2GraSch}

Suppose that we have a non-orthogonal basis in $\mathscr{Z}$ given by $\{\Phi_j\}_{j=0}^\infty$.
The objective of the \emph{Gram-Schmidt process} is to use this basis to construct an orthogonal basis in the same space with the recursive formula:
\begin{equation}\label{eq2GraSch1200}
\Psi_j:=\Phi_j-\sum_{k=0}^{j-1}\frac{\langle\Phi_j,\Psi_k\rangle}{\langle\Psi_k,\Psi_k\rangle}\Psi_k\quad\forall j\in\mathbb{N}_0.
\end{equation}

Recall that in Section \ref{sec2SetNot} we prescribed the condition that the first basis vector is identically equal to one.
This condition gives rise to the following theorem which is valid for any square-integrable function space defined on a probability space.
We point out that the benefit of using this theorem is that if both the expectation vector and the covariance matrix of the non-orthogonal basis are known beforehand, the orthogonalization process can (in general) be performed faster than the traditional Gram-Schmidt process.\footnote{Please see Appendix \ref{appsec2TimComAnaThe1} for a time-complexity analysis for Theorem \ref{thm2paper100} and the traditional Gram-Schmidt process.}
This is, for example, a typical situation in the area of stochastic modeling where the probability information of the stochastic input is---as often as not---known beforehand, and thence one might be interested in constructing a stochastic-input space based on the available information.
This need is fulfilled by Theorem \ref{thm2paper100}.

\begin{theorem}\label{thm2paper100}
Let $\mathscr{Z}$ be a random function space, and let $\{\Phi_j\}_{j=1}^\infty$ be an ordered set of linearly independent functions in $\mathscr{Z}$ such that the constant functions are not in the set.
Then, $\{\Psi_j\}_{j=0}^\infty$ is an orthogonal basis in $\mathscr{Z}$ given by:
\begin{equation}\label{eq2GraSch1300}
\Psi_j:=\Phi_j-\mathbf{E}[\Phi_j]\,\Psi_0-\sum_{k=1}^{j-1}\frac{\det\triangle_k(j)}{\det\square_k}\Psi_k\quad\text{with}\quad \Psi_0\equiv 1,
\end{equation}
where $\square_k\in\mathscr{M}(k\times k,\mathbb{R})$ is the covariance matrix for the first $k$ elements of $\{\Phi_j\}_{j=1}^\infty$:
\begin{equation*}
\square_k=
\begin{bmatrix}
\mathrm{Cov}[\Phi_1,\Phi_1] & \cdots & \mathrm{Cov}[\Phi_1,\Phi_k]\\
\vdots & \ddots & \vdots\\
\mathrm{Cov}[\Phi_k,\Phi_1] & \cdots & \mathrm{Cov}[\Phi_k,\Phi_k]
\end{bmatrix},
\end{equation*}
and $\triangle_k:\{k+1,k+2,\ldots\}\to\mathscr{M}(k\times k,\mathbb{R})$ is a map defined by
\begin{equation*}
\triangle_k(j)=
\begin{bmatrix}
\mathrm{Cov}[\Phi_1,\Phi_1] & \cdots & \mathrm{Cov}[\Phi_1,\Phi_k]\\
\vdots & \ddots & \vdots\\
\mathrm{Cov}[\Phi_{k-1},\Phi_1] & \cdots & \mathrm{Cov}[\Phi_{k-1},\Phi_k]\\
\mathrm{Cov}[\Phi_j,\Phi_1] & \cdots & \mathrm{Cov}[\Phi_j,\Phi_k]
\end{bmatrix}
\end{equation*}
with $\triangle_1(j)=\mathrm{Cov}[\Phi_j,\Phi_1]$ and
\begin{equation*}
\triangle_2(j)=
\begin{bmatrix}
\mathrm{Cov}[\Phi_1,\Phi_1] & \mathrm{Cov}[\Phi_1,\Phi_2]\\
\mathrm{Cov}[\Phi_j,\Phi_1] & \mathrm{Cov}[\Phi_j,\Phi_2]
\end{bmatrix}.
\end{equation*}
In these expressions $k\in\mathbb{N}_1$.
\end{theorem}

\begin{proof}[Proof sketch]
The proof of this theorem follows from the Gram-Schmidt process applied to the set $\{\Phi_j\}_{j=1}^\infty$.
To begin with, let us consider first the cases when $j\in\{0,1,2,3\}$.
Expression \eqref{eq2GraSch1300} already tells us that for $j=0$, $\Psi_0\equiv1$.
It is for this reason that all constant functions are excluded from $\{\Phi_j\}_{j=1}^\infty$.
For $j\in\{1,2,3\}$, we proceed as follows.
\begin{enumerate}
\item[\ding{111}] {\bfseries When $j=1$.}\quad It is easy to see that \eqref{eq2GraSch1200} yields
\begin{equation}\label{eq2GraSch1310}
\Psi_1=\Phi_1-\frac{\langle\Phi_1,\Psi_0\rangle}{\langle \Psi_0,\Psi_0\rangle}\Psi_0=\Phi_1-\mathbf{E}[\Phi_1]\,\Psi_0,
\end{equation}
because $\langle\Phi_1,\Psi_0\rangle=\mathbf{E}[\Phi_1]$ and $\langle\Psi_0,\Psi_0\rangle=1$.

\item[\ding{111}] {\bfseries When $j=2$.}\quad From \eqref{eq2GraSch1200} we get
\begin{equation}\label{eq2GraSch1311}
\Psi_2=\Phi_2-\frac{\langle\Phi_2,\Psi_0\rangle}{\langle \Psi_0,\Psi_0\rangle}\Psi_0-\frac{\langle\Phi_2,\Psi_1\rangle}{\langle \Psi_1,\Psi_1\rangle}\Psi_1.
\end{equation}
Replacing \eqref{eq2GraSch1310} into \eqref{eq2GraSch1311} gives
\begin{equation*}
\Psi_2=\Phi_2-\mathbf{E}[\Phi_2]\,\Psi_0-\frac{\mathrm{Cov}[\Phi_1,\Phi_2]}{\mathrm{Cov}[\Phi_1,\Phi_1]}\Psi_1=\Phi_2-\mathbf{E}[\Phi_2]\,\Psi_0-\frac{\det\triangle_1(2)}{\det\square_1}\Psi_1,
\end{equation*}
after noting that $\langle\Phi_2,\Phi_1-\mathbf{E}[\Phi_1]\,\Psi_0\rangle$ simplifies to $\mathrm{Cov}[\Phi_1,\Phi_2]$, and $\langle\Phi_1-\mathbf{E}[\Phi_1]\,\Psi_0,\Phi_1-\mathbf{E}[\Phi_1]\,\Psi_0\rangle$ is nothing but $\mathrm{Cov}[\Phi_1,\Phi_1]$.

\item[\ding{111}] {\bfseries When $j=3$.}\quad In a similar fashion, \eqref{eq2GraSch1200} yields
\begin{multline*}
\Psi_3=\Phi_3-\mathbf{E}[\Phi_3]\,\Psi_0-\frac{\mathrm{Cov}[\Phi_1,\Phi_3]}{\mathrm{Cov}[\Phi_1,\Phi_1]}\Psi_1-\frac{\mathrm{Cov}[\Phi_1,\Phi_1]\,\mathrm{Cov}[\Phi_3,\Phi_2]-\mathrm{Cov}[\Phi_1,\Phi_2]\,\mathrm{Cov}[\Phi_3,\Phi_1]}{\mathrm{Cov}[\Phi_1,\Phi_1]\,\mathrm{Cov}[\Phi_2,\Phi_2]-\mathrm{Cov}[\Phi_1,\Phi_2]\,\mathrm{Cov}[\Phi_2,\Phi_1]}\Psi_2\\[1ex]
=\Phi_3-\mathbf{E}[\Phi_3]\,\Psi_0-\frac{\det\triangle_1(3)}{\det\square_1}\Psi_1-\frac{\det\triangle_2(3)}{\det\square_2}\Psi_2.
\end{multline*}
\end{enumerate}

The following formulas can be derived algebraically using mathematical induction.
For brevity, we only provide an insight into how this can be done.

\begin{enumerate}
\item[\ding{111}] {\bfseries Formula for $\langle\Psi_k,\Psi_k\rangle$.}\quad Consider, for example, the case when $k=2$:
\begin{equation}\label{eq2GraSch1320}
\langle\Psi_2,\Psi_2\rangle=\mathbf{E}[(\Psi_2)^2]=\mathbf{E}\!\left[\left(\Phi_2-\mathbf{E}[\Phi_2]\,\Psi_0-\frac{\mathrm{Cov}[\Phi_1,\Phi_2]}{\mathrm{Cov}[\Phi_1,\Phi_1]}\Psi_1\right)^{\!2}\,\right].
\end{equation}
Substituting \eqref{eq2GraSch1310} into \eqref{eq2GraSch1320} yields
\begin{equation*}
\langle\Psi_2,\Psi_2\rangle=\frac{\mathrm{Cov}[\Phi_1,\Phi_1]\,\mathrm{Cov}[\Phi_2,\Phi_2]-\mathrm{Cov}[\Phi_1,\Phi_2]\,\mathrm{Cov}[\Phi_2,\Phi_1]}{\mathrm{Cov}[\Phi_1,\Phi_1]}=\frac{\det\square_2}{\det\square_1}.
\end{equation*}
In general, it is possible to show that the formula for $\langle\Psi_k,\Psi_k\rangle$ is given by:
\begin{equation}\label{eq2GraSch1325}
\langle\Psi_k,\Psi_k\rangle=\frac{\det\square_k}{\det\square_{k-1}}\quad\forall k\in\mathbb{N}_1,
\end{equation}
where we have set: $\det\square_0=1$.

\item[\ding{111}] {\bfseries Formula for $\langle\Phi_j,\Psi_k\rangle$.}\quad As before, let us consider the case when $k=2$ for the sake of illustration:
\begin{equation}\label{eq2GraSch1330}
\langle\Phi_j,\Psi_2\rangle=\left\langle\Phi_j,\Phi_2-\mathbf{E}[\Phi_2]\,\Psi_0-\frac{\mathrm{Cov}[\Phi_1,\Phi_2]}{\mathrm{Cov}[\Phi_1,\Phi_1]}\Psi_1\right\rangle.
\end{equation}
Replacing \eqref{eq2GraSch1310} into \eqref{eq2GraSch1330} gives
\begin{equation*}
\langle\Phi_j,\Psi_2\rangle=\frac{\mathrm{Cov}[\Phi_1,\Phi_1]\,\mathrm{Cov}[\Phi_j,\Phi_2]-\mathrm{Cov}[\Phi_1,\Phi_2]\,\mathrm{Cov}[\Phi_j,\Phi_1]}{\mathrm{Cov}[\Phi_1,\Phi_1]}=\frac{\det\triangle_2(j)}{\det\square_1}.
\end{equation*}
In general, the formula for $\langle\Phi_j,\Psi_k\rangle$ is given by:
\begin{equation}\label{eq2GraSch1335}
\langle\Phi_j,\Psi_k\rangle=\frac{\det\triangle_k(j)}{\det\square_{k-1}}\quad\forall k\in\mathbb{N}_1,
\end{equation}
where once again we have set: $\det\square_0=1$.
Here, $j\in\{k+1,k+2,\ldots\}$ by definition of $\triangle_k$.
\end{enumerate}

Therefore, substituting \eqref{eq2GraSch1325} and \eqref{eq2GraSch1335} into \eqref{eq2GraSch1200} yields expression \eqref{eq2GraSch1300}.
\end{proof}

This theorem will be implemented in Section \ref{sec2FSCmet} to transfer the probability information exactly at the current time of the simulation.

\begin{example}\label{ex2paper100}
Let us consider a one-dimensional random space with domain $\Xi=[-1,1]$ and measure $\mu$ given by $\mathrm{d}\mu(\xi)=\tfrac{3}{2}\xi^2\,\mathrm{d}\xi$, and suppose that we are interested in finding an orthogonal basis in one of the infinitely many $\mathscr{Z}^{[5]}$ spaces that we can construct by applying Theorem \ref{thm2paper100} over a set of 5 monomials in the variable $\xi$.

To this end, let $\{\Phi_j(\xi)=\xi^j\}_{j=1}^{5}$ be this set.
Then, its expectation vector and covariance matrix are:
\begin{equation*}
\big[\mathbf{E}[\Phi_j]\big]=
\begin{bmatrix}
0\\[1ex]
\tfrac{3}{5}\\[1ex]
0\\[1ex]
\tfrac{3}{7}\\[1ex]
0	
\end{bmatrix},\qquad
\big[\mathrm{Cov}[\Phi_i,\Phi_j]\big]=\begin{bmatrix}
\mathrm{Cov}[\Phi_1,\Phi_1] & \cdots & \mathrm{Cov}[\Phi_1,\Phi_5]\\
\vdots & \ddots & \vdots\\
\mathrm{Cov}[\Phi_5,\Phi_1] & \cdots & \mathrm{Cov}[\Phi_5,\Phi_5]
\end{bmatrix}=
\begin{bmatrix}
\tfrac{3}{5} & 0 & \tfrac{3}{7} & 0 & \tfrac{1}{3}\\[1ex]
0 & \tfrac{12}{175} & 0 & \tfrac{8}{105} & 0\\[1ex]
\tfrac{3}{7} & 0 & \tfrac{1}{3} & 0 & \tfrac{3}{11} \\[1ex]
0 & \tfrac{8}{105} & 0 & \tfrac{48}{539} & 0 \\[1ex]
\tfrac{1}{3} & 0 & \tfrac{3}{11} & 0 & \tfrac{3}{13}
\end{bmatrix}.
\end{equation*}
From \eqref{eq2GraSch1300} it follows that an orthogonal basis in the chosen $\mathscr{Z}^{[5]}$ space is given by:
\begin{gather*}
\Psi_0(\xi)=1,\quad\Psi_1(\xi)=\xi,\quad\Psi_2(\xi)=\xi^2-\tfrac{3}{5},\quad\Psi_3(\xi)=\xi^3-\tfrac{5}{7}\xi,\\
\Psi_4(\xi)=\xi^4-\tfrac{10}{9}\xi^2+\tfrac{5}{21},\quad\Psi_5(\xi)=\xi^5-\tfrac{14}{11}\xi^3+\tfrac{35}{99}\xi.
\end{gather*}

By way of illustration, the determinant ratios that appear in \eqref{eq2GraSch1300} were computed with expressions such as this:
\begin{equation*}
\frac{\det\triangle_2(4)}{\det\square_2}=
\frac{\vphantom{\rule[-2.75ex]{1ex}{5ex}}%
\begin{vmatrix}
\mathrm{Cov}[\Phi_1,\Phi_1] & \mathrm{Cov}[\Phi_1,\Phi_2]\\
\mathrm{Cov}[\Phi_4,\Phi_1] & \mathrm{Cov}[\Phi_4,\Phi_2]
\end{vmatrix}}%
{\vphantom{\rule{1ex}{3.9ex}}%
\begin{vmatrix}
\mathrm{Cov}[\Phi_1,\Phi_1] & \mathrm{Cov}[\Phi_1,\Phi_2]\\
\mathrm{Cov}[\Phi_2,\Phi_1] & \mathrm{Cov}[\Phi_2,\Phi_2]
\end{vmatrix}}=
\frac{\vphantom{\rule[-3.4ex]{1ex}{5ex}}%
\begin{vmatrix}
\tfrac{3}{5} & 0 \\[1ex]
0 & \tfrac{8}{105}
\end{vmatrix}}%
{\vphantom{\rule{1ex}{4.5ex}}%
\begin{vmatrix}
\tfrac{3}{5} & 0\\[1ex]
0 & \tfrac{12}{175}
\end{vmatrix}}=\frac{10}{9}.
\end{equation*}
\end{example}

\section{Stochastic flow map}\label{sec2StoFloMap}

Provided sufficient regularity, a stochastic system governed by \eqref{eq2SetNot1000} can be expressed in explicit form as
\begin{subequations}\label{eq2StoFloMap1000}
\begin{align}
\partial_t^n u(t,\xi)=\mathscr{f}(t,\xi,s(t,\xi)) & \qquad\text{on $\mathfrak{T}\times\Xi$}\label{eq2StoFloMap1000a}\\
\big\{\partial_t^{k-1} u(0,\xi)=c_k(\xi)\big\}_{k=1}^n&\qquad\text{on $\{0\}\times\Xi$},\label{eq2StoFloMap1000b}
\end{align}
\end{subequations}
where $s=(u,\partial_t u,\ldots,\partial_t^{n-1}u)\in\prod_{j=1}^n\mathscr{T}({n-j+1})\otimes\mathscr{Z}$ is the configuration state of the system over $\mathfrak{T}\times\Xi$, $\mathscr{f}:\mathfrak{T}\times\Xi\times\mathbb{R}^n\to\mathbb{R}$ is a noisy, non-autonomous function (which can also be regarded as a function in $\mathscr{V}$) concordant with \eqref{eq2SetNot1000a}, and $c_k:\Xi\to\mathbb{R}$ is a function in $\mathscr{Z}$ concordant with \eqref{eq2SetNot1000b}.

If the solution is analytic on $\mathfrak{T}$ for all $\xi\in\Xi$, then it can be represented by the Taylor series:
\begin{equation*}
u(t_i+h,\xi)=\sum_{j=0}^\infty \frac{h^j}{j!}\partial_t^j u(t_i,\xi),
\end{equation*}
where $h:=t-t_i$ is the time-step size used for the simulation around $t_i$ (once $t$ is fixed), and $t_i\in\mathfrak{T}$ is the time instant of the simulation.
Below we use this representation to define a \emph{local} stochastic flow map for the state of the system under consideration.
We notice, however, that to do so, we need to assume that $u\in\mathscr{T}(n+M-1)\otimes\mathscr{Z}\subset\mathscr{U}$, where $M\in\mathbb{N}_1$ denotes the order of the flow map we want to implement.
For most problems encountered in physics and engineering, this requirement does not represent a major drawback if $M$ is taken relatively small.

\begin{description}
\item[First-order ODE]\hspace{1ex} Specializing \eqref{eq2StoFloMap1000} for a first-order ODE ($n=1$) yields
\begin{subequations}\label{eq2StoFloMap1030}
\begin{align}
\partial_t u(t,\xi)=\mathscr{f}(t,\xi,u(t,\xi))&\qquad\text{on $\mathfrak{T}\times\Xi$}\label{eq2StoFloMap1030a}\\
u(0,\xi)=c(\xi) &\qquad\text{on $\{0\}\times\Xi$}.\label{eq2StoFloMap1030b}
\end{align}
\end{subequations}
Differentiating \eqref{eq2StoFloMap1030a} with respect to time three times gives:
\begin{subequations}\label{eq2StoFloMap1035}
\begin{gather}
\partial_t^2 u:=\mathrm{D}_t\mathscr{f}=\partial_t\mathscr{f}+\partial_u\mathscr{f}\,\partial_t u\\
\partial_t^3 u:=\mathrm{D}^2_t\mathscr{f}=\partial_t^2\mathscr{f}+2\,\partial_{tu}^2\mathscr{f}\,\partial_t u+\partial_u\mathscr{f}\,\partial_t^2u+\mathscr{g}_3\\
\partial_t^4 u:=\mathrm{D}^3_t\mathscr{f}=\partial_t^3\mathscr{f}+3\,\partial_{ttu}^3\mathscr{f}\,\partial_t u+3\,\partial_{tu}^2\mathscr{f}\,\partial_t^2 u+\partial_u\mathscr{f}\,\partial_t^3 u+\mathscr{g}_4,
\end{gather}
\end{subequations}
where $\mathscr{g}_3=\partial_u^2\mathscr{f}\,\big(\partial_t u\big)^{\!2}$, and $\mathscr{g}_4=3\,\partial_{tuu}^3\mathscr{f}\,\big(\partial_t u\big)^{\!2}+\partial_u^3\mathscr{f}\,\big(\partial_t u\big)^{\!3}+3\,\partial_u^2\mathscr{f}\,\partial_t^2 u\,\partial_t u$.

A stochastic flow map of order 4, $\varphi(4):\mathbb{R}\times\mathscr{Z}\to\mathscr{Z}$, can then be given by:
\begin{equation*}
\varphi(4)(h,s(t_i,\cdot\,)):=u(t_i+h,\cdot\,)-O(h^5)= \sum_{j=0}^4 \frac{h^j}{j!}\partial_t^j u(t_i,\cdot\,),
\end{equation*}
where the time derivatives of $u$ at $t=t_i$ are computed with \eqref{eq2StoFloMap1030a} and \eqref{eq2StoFloMap1035}.

\item[Second-order ODE]\hspace{1ex} Likewise, specializing \eqref{eq2StoFloMap1000} for a second-order ODE ($n=2$) yields
\begin{subequations}\label{eq2StoFloMap1040}
\begin{align}
\partial_t^2 u(t,\xi)=\mathscr{f}(t,\xi,u(t,\xi),\dot{u}(t,\xi))&\qquad\text{on $\mathfrak{T}\times\Xi$}\label{eq2StoFloMap1040a}\\
\big\{u(0,\xi)=c_1(\xi),\, \dot{u}(0,\xi)=c_2(\xi)\big\} &\qquad\text{on $\{0\}\times\Xi$},\label{eq2StoFloMap1040b}
\end{align}
\end{subequations}
where $\dot{u}:=\partial_t u$.
Differentiating \eqref{eq2StoFloMap1040a} with respect to time three times gives:
\begin{subequations}\label{eq2StoFloMap1045}
\begin{gather}
\partial_t^3 u:=\mathrm{D}_t\mathscr{f}=\partial_t\mathscr{f}+\partial_u\mathscr{f}\,\partial_t u+\partial_{\dot{u}}\mathscr{f}\,\partial_t^2 u\\
\partial_t^4 u:=\mathrm{D}^2_t\mathscr{f}=\partial_t^2\mathscr{f}+2\,\partial_{tu}^2\mathscr{f}\,\partial_t{u}+\big(2\,\partial_{t\dot{u}}^2\mathscr{f}+\partial_{u}\mathscr{f}\big)\,\partial_t^2 u+\partial_{\dot{u}}\mathscr{f}\,\partial_t^3 u+\mathscr{h}_4\\
\partial_t^5 u:=\mathrm{D}^3_t\mathscr{f}=
\partial_t^3\mathscr{f}
+3\,\partial_{ttu}^3\mathscr{f}\,\partial_t u
+3\,\big(\partial_{tt\dot{u}}^3\mathscr{f}+\partial_{tu}^2\mathscr{f}\big)\,\partial_t^2 u
+\big(3\,\partial_{t\dot{u}}^2\mathscr{f}+\partial_u\mathscr{f}\big)\,\partial_t^3 u
+\partial_{\dot{u}}\mathscr{f}\,\partial_t^4 u
+\mathscr{h}_5,
\end{gather}
\end{subequations}
where $\mathscr{h}_4=\partial_u^2\mathscr{f}\,\big(\partial_t u\big)^{\!2}+2\,\partial_{u\dot{u}}^2\mathscr{f}\,\partial_t u\,\partial_t^2 u+\partial_{\dot{u}}^2\mathscr{f}\,\big(\partial_t^2 u\big)^{\!2}$, and 
\begin{multline*}
\mathscr{h}_5=
3\,\partial_{tuu}^3\mathscr{f}\,\big(\partial_t u\big)^{\!2}
+3\,\big(\partial_{t\dot{u}\dot{u}}^3\mathscr{f}+\partial_{u\dot{u}}^2\mathscr{f}\big)\big(\partial_t^2 u\big)^{\!2}
+\partial_u^3\mathscr{f}\,\big(\partial_t u\big)^{\!3}
+\partial_{\dot{u}}^3\mathscr{f}\,\big(\partial_t^2 u\big)^{\!3}\\
+3\,\big(2\,\partial_{tu\dot{u}}^3\mathscr{f}+\partial_u^2\mathscr{f}\big)\,\partial_t u\,\partial_t^2 u
+3\,\partial_{u\dot{u}}^2\mathscr{f}\,\partial_t u\,\partial_t^3 u
+3\,\partial_{\dot{u}}^2\mathscr{f}\,\partial_t^2 u\,\partial_t^3 u\\
+3\,\partial_{uu\dot{u}}^3\mathscr{f}\,\big(\partial_t u\big)^{\!2}\,\partial_t^2 u
+3\,\partial_{u\dot{u}\dot{u}}^3\mathscr{f}\,\partial_t u\,\big(\partial_t^2 u\big)^{\!2}.
\end{multline*}

A stochastic flow map of order 4, $\varphi(4):\mathbb{R}\times\mathscr{Z}^2\to\mathscr{Z}^2$, can then be defined as:
\begin{equation}\label{eq2StoFloMap1050}
\varphi(4)(h,s(t_i,\cdot\,)):=\big(u(t_i+h,\cdot\,),\dot{u}(t_i+h,\cdot\,)\big)-O(h^5)
=\bigg(\sum_{j=0}^4\frac{h^j}{j!}\partial_t^j u(t_i,\cdot\,),\,\sum_{j=0}^4\frac{h^j}{j!}\partial_t^{j+1} u(t_i,\cdot\,)\bigg),
\end{equation}
where the second and higher time derivatives of $u$ at $t=t_i$ are computed with \eqref{eq2StoFloMap1040a} and \eqref{eq2StoFloMap1045}.
\end{description}

\begin{remark}
In these expressions we observe that if \eqref{eq2StoFloMap1000} is a linear ODE, then $\mathscr{g}_3=\mathscr{g}_4=\mathscr{h}_4=\mathscr{h}_5\equiv 0$.
If in addition it is autonomous, expressions \eqref{eq2StoFloMap1035} and \eqref{eq2StoFloMap1045} reduce, respectively, to:
\begin{gather}
\partial_t^2 u=\partial_u\mathscr{f}\,\partial_t u,\quad\partial_t^3 u=\partial_u\mathscr{f}\,\partial_t^2 u\quad\text{and}\quad\partial_t^4 u=\partial_u\mathscr{f}\,\partial_t^3 u.\tag{\ref{eq2StoFloMap1035}*}\\
\partial_t^3 u=\partial_u\mathscr{f}\,\partial_t u+\partial_{\dot{u}}\mathscr{f}\,\partial_t^2 u,\quad\partial_t^4 u=\partial_u\mathscr{f}\,\partial_t^2u+\partial_{\dot{u}}\mathscr{f}\,\partial_t^3u\quad\text{and}\quad\partial_t^5u=\partial_u\mathscr{f}\,\partial_t^3u+\partial_{\dot{u}}\mathscr{f}\,\partial_t^4u.\tag{\ref{eq2StoFloMap1045}*}
\end{gather}
\end{remark}

\begin{description}
\item[High-order ODE]\hspace{1ex} In general, a stochastic flow map of order $M$, $\varphi(M):\mathbb{R}\times\mathscr{Z}^n\to\mathscr{Z}^n$, can be taken as:
\begin{equation}\label{eq2StoFloMap1080}
\varphi(M)(h,s(t_i,\cdot\,)):=\big(u(t_i+h,\cdot\,),\ldots,\partial_t^{k-1} u(t_i+h,\cdot\,),\ldots,\partial_t^{n-1} u(t_i+h,\cdot\,)\big)-O(h^{M+1}),
\end{equation}
where its $k$-th component, $\varphi^k(M):\mathbb{R}\times\mathscr{Z}^n\to\mathscr{Z}$, is given by:
\begin{equation*}\label{eq2StoFloMap1085}
\varphi^k(M)(h,s(t_i,\cdot\,)):=\partial_t^{k-1}u(t_i+h,\cdot\,)-O(h^{M+1})=\sum_{j=0}^M\frac{h^j}{j!}\partial_t^{j+k-1} u(t_i,\cdot\,)
\end{equation*}
with $k\in\{1,2,\ldots,n\}$.

In expression \eqref{eq2StoFloMap1080}, the $n$-th time derivative of $u$ at $t=t_i$ is computed with \eqref{eq2StoFloMap1000a}, and if $M=4$, then the next time derivatives of $u$ are given by:
\begin{subequations}\label{eq2StoFloMap1090}
\begin{gather}
\partial_t^{n+1} u:=\frac{\mathrm{d}\mathscr{f}}{\mathrm{d}t}=\frac{\partial\mathscr{f}}{\partial t}+\frac{\partial\mathscr{f}}{\partial s^k}\partial_t^k u,\\
\partial_t^{n+2} u:=\frac{\mathrm{d}^2\mathscr{f}}{\mathrm{d}t^2}=\frac{\partial^2\mathscr{f}}{\partial t^2}+2\frac{\partial^2\mathscr{f}}{\partial t\partial s^k}\partial_t^k u+\frac{\partial\mathscr{f}}{\partial s^k}\partial_t^{k+1}u+\frac{\partial^2\mathscr{f}}{\partial s^k\partial s^l}\partial_t^k u\,\partial_t^l u,
\end{gather}
\end{subequations}
and
\begin{multline}
\partial_t^{n+3} u:=\frac{\mathrm{d}^3\mathscr{f}}{\mathrm{d}t^3}=\frac{\partial^3\mathscr{f}}{\partial t^3}+3\frac{\partial^3\mathscr{f}}{\partial t^2\partial s^k}\partial_t^k u+3\frac{\partial^2\mathscr{f}}{\partial t\partial s^k}\partial_t^{k+1} u+\frac{\partial\mathscr{f}}{\partial s^k}\partial_t^{k+2}u\\
+3\frac{\partial^3\mathscr{f}}{\partial t\partial s^k\partial s^l}\partial_t^k u\,\partial_t^l u+3\frac{\partial^2\mathscr{f}}{\partial s^k\partial s^l}\partial_t^{k+1} u\,\partial_t^l u+\frac{\partial^3\mathscr{f}}{\partial s^k\partial s^l\partial s^m}\partial_t^k u\,\partial_t^l u\,\partial_t^m u,\tag{\ref{eq2StoFloMap1090}c}
\end{multline}
where $s^k=\partial_t^{k-1}u$ is the $k$-th component of $s$, and a summation sign is implied over every repeated index $k,l,m\in\{1,2,\ldots,n\}$.

For an autonomous, $n$-th-order linear ODE, the first $M-1$ time derivatives of \eqref{eq2StoFloMap1000a} reduce to:
\begin{equation*}
\partial_t^{n+m}u=\sum_{k=1}^n \frac{\partial\mathscr{f}}{\partial s^k}\,\partial_t^{m+k-1}u\quad\forall m\in\{1,2,\ldots,M-1\}.
\end{equation*}
\end{description}

It is worth noting that the goal of a (local) stochastic flow map is to push the state of the system one-time step forward in $\mathscr{Z}^n$ (i.e.~in the \emph{random phase space} of the system\footnote{Note that $\mathscr{Z}^n$ is the $n$-fold cartesian product of $\mathscr{Z}$, where $n$ denotes the order of the system's governing ODE with respect to time.}), provided that the time-step size used is greater than zero.
For example, Fig.~\ref{fig2StoFloMap1200} depicts the case of a system governed by a second-order ODE whose state motion in $\mathscr{Z}^2$ starts at $t=0$ and ends at $t=T$; pretty much in the same way a two-dimensional, time-dependent stochastic input would evolve in $\mathscr{Z}^2$ if its map were known beforehand.

\begin{remark}\label{rmk2paper300}
In practice, it is easier to compute the time derivatives of $\mathscr{f}$ approximately, by using any of the standard numerical methods available in the literature, such as the central difference method.
When the central difference method is, say, used to compute approximately the first time derivative of $\mathscr{f}$ at $t=t_i$, we get:
\begin{equation*}
\frac{\mathrm{d}\mathscr{f}}{\mathrm{d}t}(t_i,\cdot\,,s(t_i,\cdot\,))\approx\frac{\mathscr{f}(t_i+h,\cdot\,,s(t_i+h,\cdot\,))-\mathscr{f}(t_i-h,\cdot\,,s(t_i-h,\cdot\,))}{2h}
\end{equation*}
for some small $h\in\mathbb{R}\setminus\{0\}$.
Therefore, the condition that we made earlier that $u\in\mathscr{T}(n+M-1)\otimes\mathscr{Z}$ now drops to its natural condition that $u\in\mathscr{U}$.
\end{remark}

\begin{figure}
\centering
\includegraphics[]{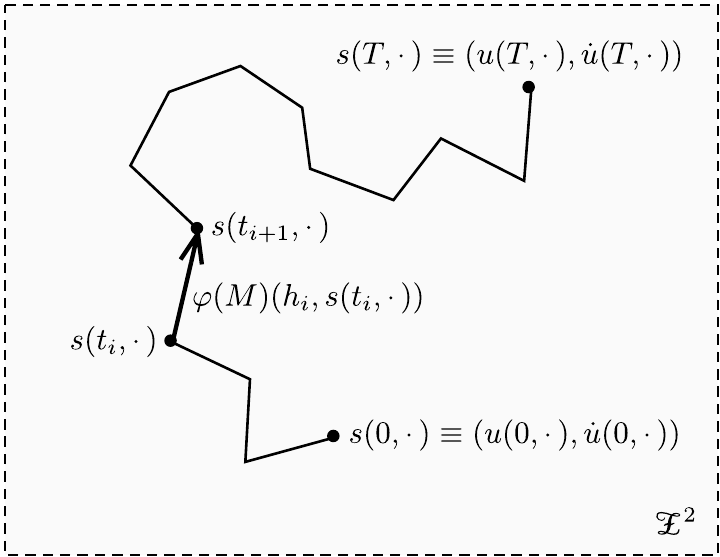}
\caption{Evolution of a second-order dynamical system via a stochastic flow map of order $M$ (with $h_i>0$)}
\label{fig2StoFloMap1200}
\end{figure}

\begin{example}\label{ex2paper200}
Consider a specialization of the stochastic system given by \eqref{eq2SetNot1000} for the case when $n=2$, and without loss of generality, suppose that $\mathcal{L}$, $\mathcal{B}_1$ and $\mathcal{B}_2$ are linear operators.
Then, under this setting, one obtains:
\begin{subequations}\label{eq2StoFloMap7000}
\begin{align}
a_2(t,\xi)\,\ddot{u}(t,\xi)+a_1(t,\xi)\,\dot{u}(t,\xi)+a_0(t,\xi)\,u(t,\xi)=f(t,\xi)&\qquad\text{on $\mathfrak{T}\times\Xi$}\label{eq2StoFloMap7000a}\\
\left\{\!\!\begin{array}{c}
b_{11}(\xi)\,\dot{u}(0,\xi)+b_{10}(\xi)\,u(0,\xi)=b_1(\xi)\\
b_{21}(\xi)\,\dot{u}(0,\xi)+b_{20}(\xi)\,u(0,\xi)=b_2(\xi)
\end{array}\!\!\right\}&\qquad\text{on $\{0\}\times\Xi$},\label{eq2StoFloMap7000b}
\end{align}
\end{subequations}
where the dimensionality of the random space, $d$, can be assumed to be any natural number.

Provided sufficient regularity, an explicit form of \eqref{eq2StoFloMap7000} can be established by:
\begin{subequations}\label{eq2StoFloMap7010}
\begin{align}
\partial_t^2 u(t,\xi):=\mathscr{f}(t,\xi,u(t,\xi),\dot{u}(t,\xi))=\bar{f}(t,\xi)+\bar{a}_0(t,\xi)\,u(t,\xi)+\bar{a}_1(t,\xi)\,\dot{u}(t,\xi)&\qquad\text{on $\mathfrak{T}\times\Xi$}\label{eq2StoFloMap7010a}\\
\big\{u(0,\xi)=\bar{b}_1(\xi),\,\dot{u}(0,\xi)=\bar{b}_2(\xi)\big\}&\qquad\text{on $\{0\}\times\Xi$},\label{eq2StoFloMap7010b}
\end{align}
\end{subequations}
where $\bar{f}=f/a_2$, $\bar{a}_0=-a_0/a_2$ and $\bar{a}_1=-a_1/a_2$ are elements of $\mathscr{V}$, and $\bar{b}_1=(b_1b_{21}-b_2b_{11})/(b_{10}b_{21}-b_{11}b_{20})$ and $\bar{b}_2=(b_2b_{10}-b_1b_{20})/(b_{10}b_{21}-b_{11}b_{20})$ are elements of $\mathscr{Z}$.

This is, of course, a non-autonomous stochastic system because $\mathscr{f}$ depends explicitly on the time variable $t$, and what is more, it features a time-dependent stochastic input $x=(x_1,x_2,x_3)\in\mathscr{V}^3$ given by:
\begin{equation*}
x_1(t,\xi)=\bar{f}(t,\xi),\quad x_2(t,\xi)=\bar{a}_0(t,\xi),\quad x_3(t,\xi)=\bar{a}_1(t,\xi).
\end{equation*}

A stochastic flow map, $\varphi(4)=(\varphi^1(4),\varphi^2(4))$, for the system in hand can then be specified with \eqref{eq2StoFloMap1050} to produce:
\begin{subequations}\label{eq2StoFloMap7020}
\begin{gather}
\varphi^1(4)(h,s(t_i,\cdot\,))=u(t_i,\cdot\,)+h\,\dot{u}(t_i,\cdot\,)+\tfrac{1}{2}h^2\,\partial_t^2u(t_i,\cdot\,)+\tfrac{1}{6}h^3\,\partial_t^3u(t_i,\cdot\,)+\tfrac{1}{24}h^4\,\partial_t^4u(t_i,\cdot\,)\\
\varphi^2(4)(h,s(t_i,\cdot\,))=\dot{u}(t_i,\cdot\,)+h\,\partial_t^2u(t_i,\cdot\,)+\tfrac{1}{2}h^2\,\partial_t^3u(t_i,\cdot\,)+\tfrac{1}{6}h^3\,\partial_t^4u(t_i,\cdot\,)+\tfrac{1}{24}h^4\,\partial_t^5u(t_i,\cdot\,),
\end{gather}
\end{subequations}
where the second time derivative $\partial_t^2 u(t_i,\cdot\,)$ is computed with \eqref{eq2StoFloMap7010a}, and the next time derivatives $\{\partial_t^j u(t_i,\cdot\,)\}_{j=3}^5$ with \eqref{eq2StoFloMap1045} from where one obtains: $\mathscr{h}_4=\mathscr{h}_5\equiv 0$ and
\begin{subequations}\label{eq2StoFloMap7025}
\begin{gather}
\partial_t\mathscr{f}=\partial_t\bar{f}+\partial_t\bar{a}_0\,u+\partial_t\bar{a}_1\,\dot{u},\quad
\partial_t^2\mathscr{f}=\partial_t^2\bar{f}+\partial_t^2\bar{a}_0\,u+\partial_t^2\bar{a}_1\,\dot{u},\quad
\partial_t^3\mathscr{f}=\partial_t^3\bar{f}+\partial_t^3\bar{a}_0\,u+\partial_t^3\bar{a}_1\,\dot{u},\\
\partial_u\mathscr{f}=\bar{a}_0,\quad\partial_{tu}^2\mathscr{f}=\partial_t\bar{a}_0,\quad\partial_{ttu}^3\mathscr{f}=\partial_t^2\bar{a}_0,\\
\partial_{\dot{u}}\mathscr{f}=\bar{a}_1,\quad\partial_{t\dot{u}}^2\mathscr{f}=\partial_t\bar{a}_1,\quad\partial_{tt\dot{u}}^3\mathscr{f}=\partial_t^2\bar{a}_1.
\end{gather}
\end{subequations}
\end{example}

\section{Enriched stochastic flow map}\label{sec2EnrStoFloMap}

In the enriched version of the stochastic flow map we are not only concerned with pushing the state of the system one-time step forward in $\mathscr{Z}^n$, but also $\mathscr{f}$ (as displayed in Eq.~\eqref{eq2StoFloMap1000a}) and its first $M-1$ time derivatives.
Because of this, we define the \emph{enriched stochastic flow map} of order $M$, $\hat{\varphi}(M):\mathbb{R}\times\mathscr{Z}^{n+M}\to\mathscr{Z}^{n+M}$, such that its $k$-th component is given by:
\begin{equation*}
\hat{\varphi}^k(M)(h,\hat{s}(t_i,\cdot\,))=:\hat{s}^k(t_i+h,\cdot\,)=%
\begin{cases}
\varphi^k(M)(h,s(t_i,\cdot\,))& \text{for $k\in\{1,2,\ldots,n\}$}\\
\mathrm{D}_t^{k-n-1}\mathscr{f}(t_i+h,\cdot\,,s(t_i+h,\cdot\,)) & \text{otherwise}
\end{cases}
\end{equation*}
with $k\in\{1,2,\ldots,n+M\}$.
Here $\hat{s}=(u,\partial_t u,\ldots,\partial_t^{n+M-1}u)\in\prod_{j=1}^{n+M}\mathscr{T}(n+M-j)\otimes\mathscr{Z}$ is called the \emph{enriched configuration state} of the system over $\mathfrak{T}\times\Xi$.
Observe that $\hat{s}^{n+1}:=\partial_t^n u=\mathscr{f}$ is defined by \eqref{eq2StoFloMap1000a}, and that if $M=4$, then $\{\hat{s}^k:=\partial_t^{k-1} u=\mathrm{D}_t^{k-n-1}\mathscr{f}\}_{k=n+2}^{n+4}$ is given by the expressions prescribed in \eqref{eq2StoFloMap1090}.

\begin{example}\label{ex2paper300}
Consider the stochastic system presented in Example \ref{ex2paper200}. 
The associated enriched stochastic flow map for that system, $\hat{\varphi}(4)$, is found to be:
\begin{subequations}
\begin{gather}
\hat{\varphi}^1(4)(h,s(t_i,\cdot\,)):=\varphi^1(4)(h,s(t_i,\cdot\,))=s^1(t_i+h,\cdot\,)=u(t_i+h,\cdot\,)-O(h^5),\\
\hat{\varphi}^2(4)(h,s(t_i,\cdot\,)):=\varphi^2(4)(h,s(t_i,\cdot\,))=s^2(t_i+h,\cdot\,)=\dot{u}(t_i+h,\cdot\,)-O(h^5),\\
\hat{\varphi}^3(4):=\mathscr{f}=\partial_t^2 u,\quad \hat{\varphi}^4(4):=\mathrm{D}_t\mathscr{f}=\partial_t^3 u,\quad \hat{\varphi}^5(4):=\mathrm{D}_t^2\mathscr{f}=\partial_t^4 u\quad\text{and}\quad\hat{\varphi}^6(4):=\mathrm{D}_t^3\mathscr{f}=\partial_t^5 u,
\end{gather}
\end{subequations}
where $\hat{\varphi}^1(4)$ and $\hat{\varphi}^2(4)$ are computed with \eqref{eq2StoFloMap7020}, $\hat{\varphi}^3(4)$ with \eqref{eq2StoFloMap7010a}, and $\{\hat{\varphi}^k(4)\}_{k=4}^6$ with \eqref{eq2StoFloMap1045}.
In Example \ref{ex2paper200} we found that from \eqref{eq2StoFloMap1045} one obtains $\mathscr{h}_4=\mathscr{h}_5\equiv 0$ and \eqref{eq2StoFloMap7025}.
\end{example}

\section{Flow-driven spectral chaos (FSC) method}\label{sec2FSCmet}

As already mentioned in the introduction, the FSC method uses the concept of enriched stochastic flow maps to track the evolution of the stochastic part of the solution space efficiently in time.
In Section \ref{sec2StoFloMap}, we assumed that the solution of the system governed by \eqref{eq2StoFloMap1000} was analytic on the temporal domain.
This assumption implies that $u$ can be represented by a Taylor series centered at $t=t_i\in\mathfrak{T}$ for all $\xi\in\Xi$:
\begin{equation}\label{eq2FSCMet1000}
u(t,\xi)=\sum_{j=0}^\infty \frac{(t-t_i)^j}{j!}\partial_t^j u(t_i,\xi),
\end{equation}
where $(t-t_i)^j/j!$ is nothing but a temporal function in $\mathscr{T}$, and $\partial_t^j u(t_i,\xi)$ is a random function in $\mathscr{Z}$.
From this it follows that if $\{\partial_t^j u(t_i,\xi)\}_{j=0}^\infty$ is orthogonalized with respect to the measure in $\mathscr{Z}$, one can write expression \eqref{eq2FSCMet1000} in the following way:
\begin{equation}
u(t,\xi)=\sum_{j=0}^\infty u^j(t)\,\Psi_j(\xi).\tag{\ref{eq2SetNot1020}}
\end{equation}

For a system driven by a stochastic flow map of order $M$, it is apparent that the infinitely many basis vectors in $\{\partial_t^j u(t_i,\cdot\,)\}_{j=0}^\infty$ do not all need to be orthogonalized, but only the $n+M$ components of the enriched stochastic flow map: $\{\hat{\varphi}^k(M)(0,\hat{s}(t_i,\cdot\,))\equiv \partial_t^{k-1} u(t_i,\cdot\,)\}_{k=1}^{n+M}$.
In this sense, the order of the stochastic flow map determines the maximum number of basis vectors to use in the simulation.
This explains why the FSC method does not suffer from the curse of dimensionality at the random-function-space level, even when the probability space is high-dimensional.

However, to reduce the numerical cost associated with the orthogonalization of $n+M$ basis vectors as indicated above, one can choose to start the FSC analysis by considering first the lowest value for $M$, i.e.~$M=1$.
Then, if more accurate results are needed, the $M$-value can be incremented progressively, provided that it does not exceed the order of the stochastic flow map we are targeting.
Hence, in the FSC method the discretization of the random function space $\mathscr{Z}^{[P]}$ is bounded by $n+1\leq P\leq n+M$ as long as we presume that the system is driven by a stochastic flow map of order $M$.

\begin{remark}
We note that once this $M$ is fixed to solve the problem in hand numerically, the probability information of the system's state can be pushed exactly in time if the discretization level of $\mathscr{Z}$ is such that $P=n+M$.
However, as we will see in Section \ref{sec2DisNumRes}, accurate results are also achievable for a lower discretization level of $\mathscr{Z}$, and thus, we do not always need to run our simulations with a relatively high value of $P$.
\end{remark}

\paragraph{FSC scheme}

Let us consider the stochastic system given by \eqref{eq2SetNot1000}.
Let $\{\mathfrak{T}_i\}_{i=0}^{N-1}$ be a partition of the temporal domain, where $\mathfrak{T}_i\neq\emptyset$ represents the $i$-th interval of the partition, and define $s_{.i}=s|_{\mathrm{cl}(\mathfrak{T}_i)\times\Xi}$ to be the restriction of $s$ to $\mathfrak{R}_i:=\mathrm{cl}(\mathfrak{T}_i)\times\Xi$.
(Please see Fig.~\ref{fig2StochasticFlowMap}, and recall that $s$ represents the configuration state of the system over $\mathfrak{T}\times\Xi$.)
Then, if the system is driven by a stochastic flow map of order $M$, proceed as below.

\begin{figure}
\centering
\includegraphics[]{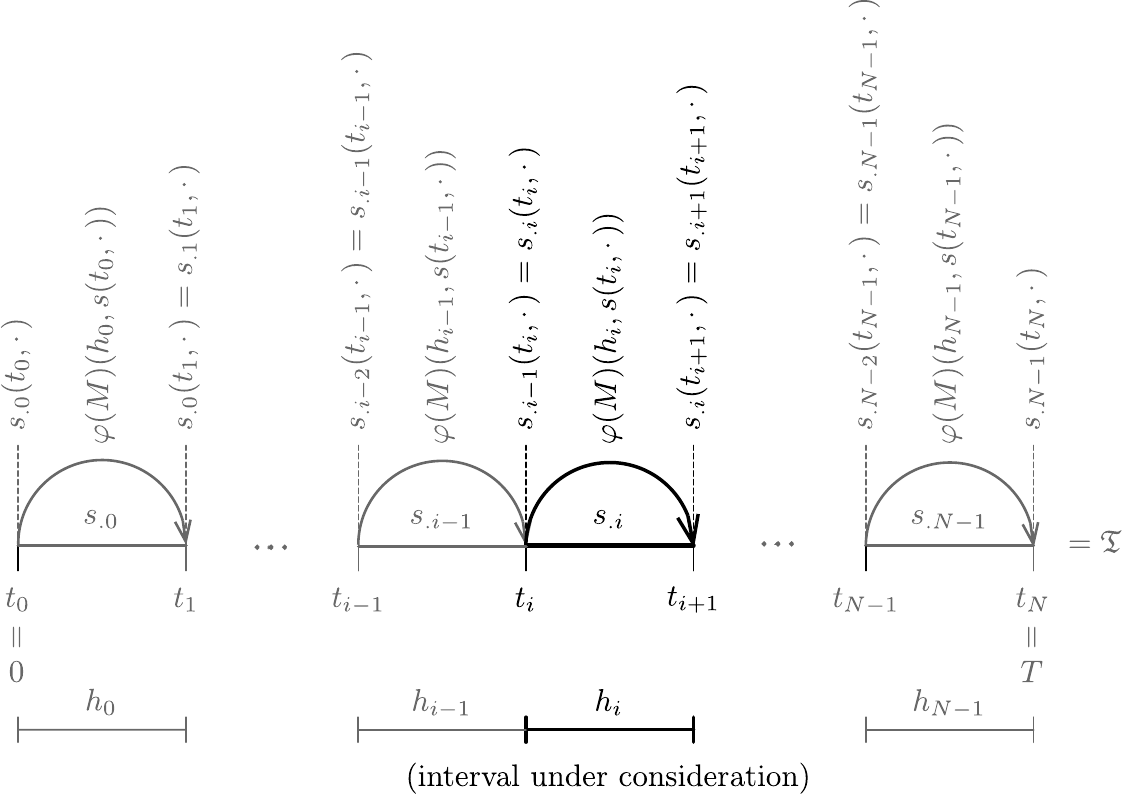}
\caption{Evolution of a dynamical system via a stochastic flow map of order $M$ (with $h_i>0$)}
\label{fig2StochasticFlowMap}
\end{figure}

\begin{enumerate}
\item Loop across the temporal domain from $i=0$ to $i=N-1$.
\begin{enumerate}
\item Define a solution representation for the configuration state $s_{.i}$ in the following way.
\begin{itemize}
\item Take $\Phi_{0.i}\equiv 1$ and $\{\Phi_{j.i}:=\hat{\varphi}^j(M)(0,\hat{s}(t_i,\cdot\,))\}_{j=1}^P$ to be an ordered set of linearly independent functions in $\mathscr{Z}$ with $n+1\leq P\leq n+M$.
Observe that $\hat{\varphi}(M)(0,\hat{s}(t_i,\cdot\,))\equiv\hat{s}_{.i}(t_i,\cdot\,)=\hat{s}_{.i-1}(t_i,\cdot\,)$ for $i\geq1$.
However, if $i=0$, then $\hat{\varphi}(M)(0,\hat{s}(t_0,\cdot\,))\equiv \hat{s}(0,\cdot\,)$.
(\emph{Note:} When the initial conditions over $\mathfrak{R}_i$ are linearly dependent, please see Remark \ref{rmk2paper500}.)
It is worth mentioning that, for computational efficiency, we have chosen $h_i=0$ for the definition of $\{\Phi_{j.i}\}_{j=1}^P$, but in principle $h_i$ is any number between 0 and the length of $\mathfrak{T}_i$.
For higher accuracy, $h_i$ is recommended to be taken as the distance between $t_i$ and the midpoint in $\mathfrak{T}_i$.
This recommendation is especially important when the length of $\mathfrak{T}_i$ is relatively large, such as in multi-time-step simulations.
\item Orthogonalize the set $\{\Phi_{j.i}\}_{j=0}^P$ using the Gram-Schmidt process \cite{cheney2010linear}, so that the resulting set $\{\Psi_{j.i}\}_{j=0}^P$ is an orthogonal basis in $\mathscr{Z}$.
That is, for $j\in\{0,1,\ldots,P\}$:
\begin{equation*}
\Psi_{j.i}:=\Phi_{j.i}-\sum_{k=0}^{j-1}\frac{\langle\Phi_{j.i},\Psi_{k.i}\rangle}{\langle\Psi_{k.i},\Psi_{k.i}\rangle}\Psi_{k.i}.
\end{equation*}
(Equivalently, Theorem \ref{thm2paper100} can be employed in this step.)
\item Define $\mathscr{Z}_i^{[P]}=\mathrm{span}\{\Psi_{j.i}\}_{j=0}^P$ to be a $p$-discretization of $\mathscr{Z}$ over the region $\mathfrak{R}_i$.
Since $\mathscr{Z}_i^{[P]}$ is an evolving function space, expansion \eqref{eq2SetNot1030} is now to be read as:
\begin{equation}\label{eq2FSCMet3020}
u_{.i}(t,\xi)\approx u^{[P]}_{.i}(t,\xi)=\sum_{j=0}^P u{^j}_{\!.i}(t)\,\Psi_{j.i}(\xi)\equiv u{^j}_{\!.i}(t)\,\Psi_{j.i}(\xi).\tag{\ref{eq2SetNot1030}*}
\end{equation}
From this it follows that the $l$-th component of the configuration state, $s{^l}_{\!.i}$, can be computed by taking the $(l-1)$-th time derivative of this representation.
Here $l\in\{1,2,\ldots,n\}$.
\end{itemize}
\item Transfer the random modes of $s_{.i-1}=(u_{.i-1},\partial_t u_{.i-1},\ldots,\partial_t^{n-1} u_{.i-1})$ to $s_{.i}=(u_{.i},\partial_t u_{.i},\ldots,\partial_t^{n-1} u_{.i})$ at $t=t_i$, given that $i\geq1$.
To this end, any of the following two approaches can be adopted.
\begin{description}
\item[FSC-1.] This approach transfers the probability information in the mean-square sense, by ensuring that the equalities shown below hold in the mean-square sense for each of the components of $s_{.i}$ and $s_{.i-1}$ at $t=t_i$ (summation sign implied only over repeated index $k$):
\begin{equation}
s{^l}_{\!.i}(t_i,\xi)=s{^l}_{\!.i-1}(t_i,\xi)\quad\iff\quad (s^l){^k}_{\!.i}(t_i)\,\Psi_{k.i}(\xi)=(s^l){^k}_{\!.i-1}(t_i)\,\Psi_{k.i-1}(\xi).\label{eq2FSCMet3030}
\end{equation}
Projecting \eqref{eq2FSCMet3030} onto $\mathscr{Z}_i^{[P]}$ gives the random modes of each of the components of $s_{.i}$ at $t=t_i$:
\begin{equation}\label{eq2FSCMet3040}
(s^l){^j}_{\!.i}(t_i)=\sum_{k=0}^P\frac{\langle\Psi_{j.i},\Psi_{k.i-1}\rangle}{\langle\Psi_{j.i},\Psi_{j.i}\rangle} (s^l){^k}_{\!.i-1}(t_i),\tag{\ref{eq2SetNot1050b}*}
\end{equation}
where $l\in\{1,2,\ldots,n\}$ and $j\in\{0,1,\ldots,P\}$. 
This approach was first introduced by Heuveline and Schick in \cite{heuveline2014hybrid}.
\item[FSC-2.] This approach transfers the probability information exactly.
In order to do so, Theorem \ref{thm2paper100} is implemented to obtain the random modes of each of the components of $s_{.i}$ at $t=t_i$.
Thus, from \eqref{eq2GraSch1300} it follows that:
\begin{equation*}
\Phi_{l.i}=\mathbf{E}[\Phi_{l.i}]\,\Psi_{0.i}+\sum_{j=1}^{l-1}\frac{\det\triangle_j(l)}{\det\square_j}\Psi_{j.i}+\Psi_{l.i}\quad\text{with}\quad \Psi_{0.i}\equiv 1,
\end{equation*}
which after taking $\{\Phi_{l.i}:=\varphi^l(M)(0,s(t_i,\cdot\,))\equiv s{^l}_{\!.i}(t_i,\cdot\,)=s{^l}_{\!.i-1}(t_i,\cdot\,)\}_{l=1}^n$ yields
\begin{equation}\label{eq2FSCMet3045}
(s^l){^j}_{\!.i}(t_i)=
\begin{cases}
\mathbf{E}[s{^l}_{\!.i-1}(t_i,\cdot\,)] & \text{for $j=0$}\\[1.25ex]
\displaystyle\frac{\det\triangle_j(l)}{\det\square_j} & \text{for $0<j<l$}\\[1.25ex]
1 & \text{for $j=l$}\\
0 & \text{otherwise,}
\end{cases}\tag{\ref{eq2SetNot1050b}*}
\end{equation}
where $l\in\{1,2,\ldots,n\}$ and $j\in\{0,1,\ldots,P\}$.
This approach was first introduced by Gerritsma et al.~in \cite{gerritsma2010time}, but it had not been generalized for high-order stochastic ODEs until now.
\end{description}
If $i=0$, the initial conditions are computed with \eqref{eq2SetNot1050b} directly.
\item Substitute \eqref{eq2FSCMet3020} into \eqref{eq2SetNot1000} to obtain \eqref{eq2SetNot1040}.
\item Project \eqref{eq2SetNot1040a} onto $\mathscr{Z}_i^{[P]}$ to obtain \eqref{eq2SetNot1050a} subject to \eqref{eq2FSCMet3040}.
Note that if $i=0$, \eqref{eq2SetNot1050a} is subject to \eqref{eq2SetNot1050b}.
\item Integrate \eqref{eq2SetNot1050} over time, as long as a suitable time integration method has been selected for solving the resulting system of equations.
This step requires to find the random modes of each of the components of the configuration state $s_{.i}$ at $t=t_{i+1}$; that is, $\{(s^l){^j}_{.i}(t_{i+1})\}_{l=1,j=0}^{n,P}$.
\item Compute both the mean and the variance of each of the components of output $y=\boldsymbol{\mathcal{M}}[u][x]$ over $\mathfrak{R}_i$, by recurring to the formulas prescribed by \eqref{eq2SetNot3100} and \eqref{eq2SetNot3300}.
\end{enumerate}
\item Post-process results.
\end{enumerate}

\begin{remark}\label{rmk2paper500}
When the initial conditions are linearly dependent over the region $\mathfrak{R}_i$, any of the following two approaches can be used:
\begin{itemize}
\item If at the start of the simulation the initial conditions over $\mathfrak{R}_1$ are deterministic, the first $n$ vectors in $\{\Phi_{j.0}:=\hat{\varphi}^j(M)(0,\hat{s}(t_0,\cdot\,))\}_{j=1}^P$ are required to be removed from the set, for they are linearly dependent to $\Phi_{0.0}\equiv 1$.
However, if more generally the initial conditions are linearly dependent over $\mathfrak{R}_i$, one can find a linear map $A$ such that its image produces a set of linearly independent vectors for use in the definition of $\Phi_{j.i}$.
That is, if the rank of $A$ is $r+1$, then $A:\mathscr{Z}^{n+1}\to\mathscr{Z}^{r+1}$ is a linear map given by:
\begin{equation*}
\big(1,\varphi(M)(0,s(t_i,\cdot\,))\big)\mapsto (b^0\equiv1,b^1,\ldots,b^r)=A\big(1,\varphi(M)(0,s(t_i,\cdot\,))\big).
\end{equation*}
For example, if $n=4$ and $r=2$, the map $A:\mathscr{Z}^5\to\mathscr{Z}^3$ can be given in matrix form by:
\begin{equation*}
\begin{bmatrix}
b^0 \\
b^1 \\
b^2
\end{bmatrix}=
\begin{bmatrix}
1 & 0 & 0 & 0 & 0 \\
0 & A^1_1 & A^1_2 & A^1_3 & A^1_4 \\
0 & A^2_1 & A^2_2 & A^2_3 & A^2_4
\end{bmatrix}\!
\begin{bmatrix}
1 \\
\varphi^1(M)(0,s(t_i,\cdot\,)) \\
\varphi^2(M)(0,s(t_i,\cdot\,)) \\
\varphi^3(M)(0,s(t_i,\cdot\,)) \\
\varphi^4(M)(0,s(t_i,\cdot\,))
\end{bmatrix}
\end{equation*}
with $A^j_k\in\mathbb{R}$ and such that $\mathrm{rank}(A)=r+1=3$. In general, this means that the ordered set needed in step 1(a) reduces to:
\begin{equation*}
\{\Phi_{j.i}\}_{j=0}^{P-n+r}\equiv
\{b^k\}_{k=0}^r\cup\{\hat{\varphi}^k(M)(0,\hat{s}(t_i,\cdot\,))\}_{k=n+1}^P.
\end{equation*}
\item As an alternative, the gPC method can be employed to advance the state of the system one-time step forward in the simulation, and then switch over FSC once this is done.
The disadvantage of this alternative is that it only works at early times of the simulation.
\end{itemize}
\end{remark}

\begin{remark}\label{rmk2paper600}
A stopping condition can be defined within the FSC scheme to help reduce the computational cost associated with the creation of a new random function space in the simulation.
However, a major drawback of incorporating a stopping condition within the scheme is that, if not chosen well, it can lead to significant degradation of the solution over time.
Stopping conditions such as those addressed by Gerritsma et al.~\cite{gerritsma2010time} were implemented in this work, but led to high errors in the response at the end of the simulations.
This is the primary reason why we did not provide one within the FSC scheme.
\end{remark}

\begin{example}\label{ex2paper400}
Consider Example \ref{ex2paper200} one more time.
If {\bf FSC-1} is used to transfer the probability information of the system's state at $t=t_i$, we get the following expressions for the random modes of $s_{.i}=(u_{.i},\dot{u}_{.i})$:
\begin{equation*}
u{^j}_{\!.i}(t_i)=\sum_{k=0}^P\frac{\langle\Psi_{j.i},\Psi_{k.i-1}\rangle}{\langle\Psi_{j.i},\Psi_{j.i}\rangle} u{^k}_{\!.i-1}(t_i)\quad\text{and}\quad 
\dot{u}{^j}_{\!.i}(t_i)=\sum_{k=0}^P\frac{\langle\Psi_{j.i},\Psi_{k.i-1}\rangle}{\langle\Psi_{j.i},\Psi_{j.i}\rangle} \dot{u}{^k}_{\!.i-1}(t_i)
\end{equation*}
with $j\in\{0,1,\ldots,P\}$.
However, if {\bf FSC-2} is employed instead, the probability information is transferred with the following expressions.
For the random modes of $u_{.i}$, the expressions are:
\begin{equation*}
u{^j}_{\!.i}(t_i)=
\begin{cases}
\mathbf{E}[u_{.i-1}(t_i,\cdot\,)] & \text{for $j=0$}\\
1 & \text{for $j=1$}\\
0 & \text{for $2\leq j\leq P$,}
\end{cases}
\end{equation*}
and those corresponding to the random modes of $\dot{u}_{.i}$ are:
\begin{equation*}
\dot{u}{^j}_{\!.i}(t_i)=
\begin{cases}
\mathbf{E}[\dot{u}_{.i-1}(t_i,\cdot\,)] & \text{for $j=0$}\\[1ex]
\displaystyle\frac{\mathrm{Cov}[u_{.i-1}(t_i,\cdot\,),\dot{u}_{.i-1}(t_i,\cdot\,)]}{\mathrm{Cov}[u_{.i-1}(t_i,\cdot\,),u_{.i-1}(t_i,\cdot\,)]} & \text{for $j=1$}\\[1.25ex]
1 & \text{for $j=2$}\\
0 & \text{for $3\leq j\leq P$.}
\end{cases}
\end{equation*}
From this example, it is easy to see why the probability information is transferred faster if {\bf FSC-2} is used.
\end{example}

\section{Numerical examples}\label{sec2NumExa}

To better investigate the performance of the FSC method in the resolution of problems involving stochastic input data, six problems are selected and described in this section.

\subsection[Problem 1 (linear system)]{Problem 1: A linear system governed by a 1st-order stochastic ODE}\label{sec2NumExa10}

We consider the problem of a falling body under stochastic air resistance as follows.

Find the velocity of the falling body $v:\mathfrak{T}\times\Xi\to\mathbb{R}$ in $\mathscr{U}$, such that ($\mu$-a.e.):
\begin{align*}
m\dot{v}+kv=mg&\qquad\text{on $\mathfrak{T}\times\Xi$}\\
v(0,\cdot\,)=\mathscr{v}&\qquad\text{on $\{0\}\times\Xi$},
\end{align*}
where $m=4$ kg is the mass of the falling body, $k:\Xi\to\mathbb{R}^+$ is the air resistance given by $k(\xi)=\xi$, $g=9.81$ m/s$^2$ is the gravity acceleration, $\mathscr{v}=50$ m/s is the initial condition of the falling body, and $\mathfrak{T}=[0,150]$ s.
Here $\dot{v}:=\partial_t v$ denotes the acceleration of the falling body.

Two probability distributions are considered for this system.
The first distribution is a \emph{uniform distribution}, $\mathrm{Uniform}\sim\xi\in\Xi=[a,b]$, and the second distribution is a \emph{beta distribution}, $\mathrm{Beta}(\alpha,\beta)\sim\xi\in\Xi=[a,b]$.
The parameters for both distributions are: $(a,b)=(1,2)$ kg/s and $(\alpha,\beta)=(2,5)$.

\subsection[Problem 2 (linear system)]{Problem 2: A linear system governed by a 2nd-order stochastic ODE with one random variable}\label{sec2NumExa20}

We consider the problem of a single-degree-of-freedom system with stochastic stiffness under free vibration.

Find the displacement of the system $u:\mathfrak{T}\times\Xi\to\mathbb{R}$ in $\mathscr{U}$, such that ($\mu$-a.e.):
\begin{align*}
m\ddot{u}+ku=0&\qquad\text{on $\mathfrak{T}\times\Xi$}\\
\big\{u(0,\cdot\,)=\mathscr{u},\,\dot{u}(0,\cdot\,)=\mathscr{v}\big\}&\qquad\text{on $\{0\}\times\Xi$},
\end{align*}
where $m=100$ kg is the mass of the system, $k:\Xi\to\mathbb{R}^+$ is the stiffness of the system given by $k(\xi)=\xi$, $\mathscr{u}=0.05$ m and $\mathscr{v}=0.20$ m/s are the initial conditions of the system, and $\mathfrak{T}=[0,150]$ s.
Note that $\dot{u}:=\partial_t u$ and $\ddot{u}:=\partial_t^2 u$ denote the velocity and acceleration of the system, respectively.

Three probability distributions are considered for this system.
The first distribution is a \emph{uniform distribution}, $\mathrm{Uniform}\sim\xi\in\Xi=[a,b]$, the second distribution is a \emph{beta distribution}, $\mathrm{Beta}(\alpha_1,\beta_1)\sim\xi\in\Xi=[a,b]$, and the third distribution is a \emph{gamma distribution}, $\mathrm{Gamma}(\alpha_2,\beta_2)\sim\xi\in\Xi=[a,\infty)$.
The parameters for these three distributions are selected to be: $(a,b)=(340,460)$ N/m and $(\alpha_1,\beta_1,\alpha_2,\beta_2)=(2,5,10,0.1)$.

\subsection[Problem 3 (linear system)]{Problem 3: A linear system governed by a 2nd-order stochastic ODE with two random variables}\label{sec2NumExa30}

In this section we consider the single-degree-of-freedom system previously defined, but in addition to having a stochastic stiffness for the system we also have a stochastic mass.
In this setting, the mass of the system, $m:\Xi\to\mathbb{R}^+$, and the stiffness of the system, $k:\Xi\to\mathbb{R}^+$, are given by $m(\xi)=\xi^1$ and $k(\xi)=\xi^2$, respectively, with 1 and 2 not denoting an exponentiation.
We note that $\xi=(\xi^1,\xi^2)$ is a 2-tuple random variable, and thus, $\Xi$ is a two-dimensional random domain.

For this system, two probability distributions are explored.
The first distribution is a \emph{uniform-uniform distribution}, $\mathrm{Uniform}\otimes\mathrm{Uniform}\sim\xi\in\Xi=[a_1,b_1]\times[a_2,b_2]$, and the second distribution is a \emph{uniform-beta distribution}, $\mathrm{Uniform}\otimes\mathrm{Beta}(\alpha,\beta)\sim\xi\in\Xi=[a_1,b_1]\times[a_2,b_2]$.
The parameters for both distributions are taken as: $(a_1,b_1)=(85,115)$ kg, $(a_2,b_2)=(340,460)$ N/m and $(\alpha,\beta)=(2,5)$.

\subsection[Problem 4 (linear system)]{Problem 4: A linear system governed by a 3rd-order stochastic ODE}\label{sec2NumExa40}

Here we consider the problem of a linear mechanical system governed by a 3rd-order stochastic ODE.

Find the displacement of the system $u:\mathfrak{T}\times\Xi\to\mathbb{R}$ in $\mathscr{U}$, such that ($\mu$-a.e.):
\begin{align*}
\partial_t^3u+\tfrac{1}{2}\,\partial_t^2u+k\,\partial_tu+u=0&\qquad\text{on $\mathfrak{T}\times\Xi$}\\
\big\{u(0,\cdot\,)=\mathscr{u},\,\partial_t u(0,\cdot\,)=\mathscr{v},\,\partial_t^2 u(0,\cdot\,)=\mathscr{a}\big\}&\qquad\text{on $\{0\}\times\Xi$},
\end{align*}
where $k:\Xi\to\mathbb{R}$ is a mechanical parameter given by $k(\xi)=\xi$; $\mathscr{u}=1$ m, $\mathscr{v}=-1$ m/s and $\mathscr{a}=2$ m/s$^2$ are the initial conditions of the system; and $\mathfrak{T}=[0,150]$ s.

Three probability distributions are investigated for this system.
The first distribution is a \emph{uniform distribution}, $\mathrm{Uniform}\sim\xi\in\Xi=[a,b]$, the second distribution is a \emph{beta distribution}, $\mathrm{Beta}(\alpha,\beta)\sim\xi\in\Xi=[a,b]$, and the third distribution is a \emph{normal distribution}, $\mathrm{Normal}(\mu,\sigma^2)\sim\xi\in\Xi=\mathbb{R}$.
The parameters for these three distributions are: $(a,b)=(2,3)$ N/m, $(\mu,\sigma)=(2.5,0.125)$ N/m and $(\alpha,\beta)=(2,5)$.

\subsection[Problem 5 (linear system)]{Problem 5: A linear system governed by a 4th-order stochastic ODE}\label{sec2NumExa50}

Next, we consider the problem of a linear mechanical system governed by a 4th-order stochastic ODE.

Find the displacement of the system $u:\mathfrak{T}\times\Xi\to\mathbb{R}$ in $\mathscr{U}$, such that ($\mu$-a.e.):
\begin{align*}
\partial_t^4u+k\,\partial_t^2u+u=0&\qquad\text{on $\mathfrak{T}\times\Xi$}\\
\big\{u(0,\cdot\,)=\mathscr{u},\,\partial_t u(0,\cdot\,)=\mathscr{v},\,\partial_t^2 u(0,\cdot\,)=\mathscr{a},\,\partial_t^3 u(0,\cdot\,)=\mathscr{j}\big\}&\qquad\text{on $\{0\}\times\Xi$},
\end{align*}
where $k:\Xi\to\mathbb{R}$ is a mechanical parameter given by $k(\xi)=\xi$; $\mathscr{u}=1$ m, $\mathscr{v}=-1$ m/s, $\mathscr{a}=2$ m/s$^2$ and $\mathscr{j}=-3$ m/s$^3$ are the initial conditions of the system; and $\mathfrak{T}=[0,150]$ s.

The same three probability distributions mentioned in Problem 4 are considered here, but the parameters of the distributions now take the following values: $(a,b)=(3,5)$ kg, $(\mu,\sigma)=(4,0.2)$ kg and $(\alpha,\beta)=(2,5)$.

\subsection[Problem 6 (nonlinear system)]{Problem 6: A nonlinear system governed by a 2nd-order stochastic ODE}\label{sec2NumExa60}

In this last section, we study the stochastic behavior of a Van-der-Pol oscillator.
Because this oscillator is highly nonlinear over the temporal-random space, we use it herein as a toy problem to test the performance of the FSC method more thoroughly.
For reference, the spectral discretization of this problem is derived in detail in Appendix \ref{appsec2DerVanderPolOsc}.
The problem can be stated as follows.

Find the displacement of the oscillator $u:\mathfrak{T}\times\Xi\to\mathbb{R}$ in $\mathscr{U}$, such that ($\mu$-a.e.):
\begin{subequations}\label{eq2NumExa500}
\begin{align}
m\ddot{u}-(1-\rho u^2)\,c\dot{u}+ku=0&\qquad\text{on $\mathfrak{T}\times\Xi$}\label{eq2NumExa500a}\\
\big\{u(0,\cdot\,)=\mathscr{u},\,\dot{u}(0,\cdot\,)=\mathscr{v}\big\}&\qquad\text{on $\{0\}\times\Xi$},\label{eq2NumExa500b}
\end{align}
\end{subequations}
where $m=100$ kg is the mass of the oscillator, $c$ is a \emph{uniformly-distributed} random variable representing the strength of the damping and defined in $[150,450]$ kg/s, $\rho=150$ m$^{-2}$ is the contributing factor to the nonlinearity of the oscillator, $k=400$ N/m is the stiffness of the oscillator, $\mathscr{u}\sim\mathrm{Beta}(2,5)$ is a \emph{beta-distributed} random variable defined in $[0.05,0.25]$ m, $\mathscr{v}=2\mathscr{u}-0.10$ is another \emph{beta-distributed} random variable expressed in m/s, and $\mathfrak{T}=[0,150]$ s.
Observe that $\dot{u}:=\partial_t u$ and $\ddot{u}:=\partial_t^2 u$ denote the velocity and acceleration of the oscillator, respectively.

\section{Discussion on numerical results}\label{sec2DisNumRes}

In this section, we demonstrate and compare the performance of the FSC scheme using the two approaches developed in Section \ref{sec2FSCmet} for the transfer of the probability information, namely FSC-1 and FSC-2.
We do this for each of the problems described in Section \ref{sec2NumExa}, followed by a computational-cost comparison between FSC and mTD-gPC \cite{heuveline2014hybrid} for Problem 2 of Section \ref{sec2NumExa20}.

The local error, $\epsilon:\mathscr{T}\to\mathscr{T}$, and the global error, $\epsilon_G:\mathscr{T}\to\mathbb{R}$, are defined with these expressions:
\begin{gather*}
\epsilon[f](t)=|f(t)-f_\mathrm{exact}(t)|\\
\epsilon_G[f]=\frac{1}{T}\int_\mathfrak{T}|f(t)-f_\mathrm{exact}(t)|\,\mathrm{d}t\approx\frac{\Delta t}{T}\sum_{i=0}^N|f(t_i)-f_\mathrm{exact}(t_i)|,
\end{gather*}
where $\Delta t$ is the time-step size used for the simulation, $t_i\in\mathfrak{T}$ is the time instant of the simulation, and $N$ denotes the number of time steps employed in the simulation (with $t_0=0$ and $t_N=N\,\Delta t=T$).

In this work we use the Runge-Kutta method \cite{butcher1987numerical} of fourth-order (aka RK4 method) to push the state of the system forward in time.
The time-step size used is $\Delta t=0.001$ s (unless indicated otherwise), which means that $N=150\,000$ time steps are used in the simulations\footnote{For numerical reasons, if in the plots an asterisk is displayed next to a probability distribution, it means that the simulation was conducted for the first 100 seconds only.
In such cases, $N=100\,000$ for a time-step size of 0.001 seconds.}.
The time-step size is taken this small to attenuate as much as possible the errors coming from the discretization of $\mathscr{T}(n)$.
As pointed out in Remark \ref{rmk2paper600}, we update the stochastic part of the solution space at every time step in an attempt to curtail the degradation of the solution over time.
Finally, the gPC method (with $P=6$) is implemented for the first second of the simulation to ensure that the stochasticity of the system's state is well developed for the analysis with FSC or mTD-gPC.

To evaluate the inner products numerically, we use the following quadrature rules on each random axis:
\begin{gather*}
\mathrm{Uniform}\sim\text{Gauss-Legendre(100 points)},\quad\mathrm{Beta}\sim\text{Gauss-Jacobi(80 points)},\\
\mathrm{Gamma}\sim\text{Gauss-Laguerre(140 points)}\quad\text{and}\quad\mathrm{Normal}\sim\text{Gauss-Hermite(110 points).}
\end{gather*}

All problems are run in MATLAB R2016b \cite{matlabr2016b} on a 2017 MacBook Pro with quad-core \emph{3.1 GHz Intel Core i7} processor (hyper-threading technology enabled), \emph{16 GB 2133 MHz LPDDR3} memory and \emph{1 TB PCI-Express SSD} storage (APFS-formatted), running macOS Mojave (version 10.14.6).

\subsection{Numerical results for the five linear systems}

Figs.~\ref{fig2System1_Uniform_FSC_Vel_6} to \ref{fig2System4_Uniform_FSC_Jerk_9} show the evolution of the mean and the variance of one of the system's responses using FSC-2 and the exact solution\footnote{To obtain both the `exact' mean and the `exact' variance, MATLAB's Symbolic Math Toolbox \cite{matlabr2016b} is first used to find the response of interest analytically (e.g.~the solution of the stochastic ODE).
Then, the {\tt vpaintegral} is called with {\tt RelTol} set equal to $10^{-16}$ to compute the mean and the variance of the response numerically at every time instant of the simulation.} for sake of comparison.
In particular, Fig.~\ref{fig2System1_Uniform_FSC_Vel_6} shows the evolution of the system's velocity for Problem 1, Figs.~\ref{fig2System21_Uniform_FSC_Disp_7} to \ref{fig2System3_Uniform_FSC_Disp_8} the evolution of the system's displacement for Problems 2 to 4, and Fig.~\ref{fig2System4_Uniform_FSC_Jerk_9} the evolution of the system's jerk for Problem 5.
As observed, responses obtained with both the FSC-1 and FSC-2 methods approach the exact solution with high fidelity. 
We emphasize that these responses are obtained by using only a few number of basis vectors with $P$ set equal to $4+n$ after the first second of the simulation.
(Recall that $n$ denotes the order of the governing ODE with respect to time.)

\begin{figure}
\centering
\begin{subfigure}[b]{0.495\textwidth}
\includegraphics[width=\textwidth]{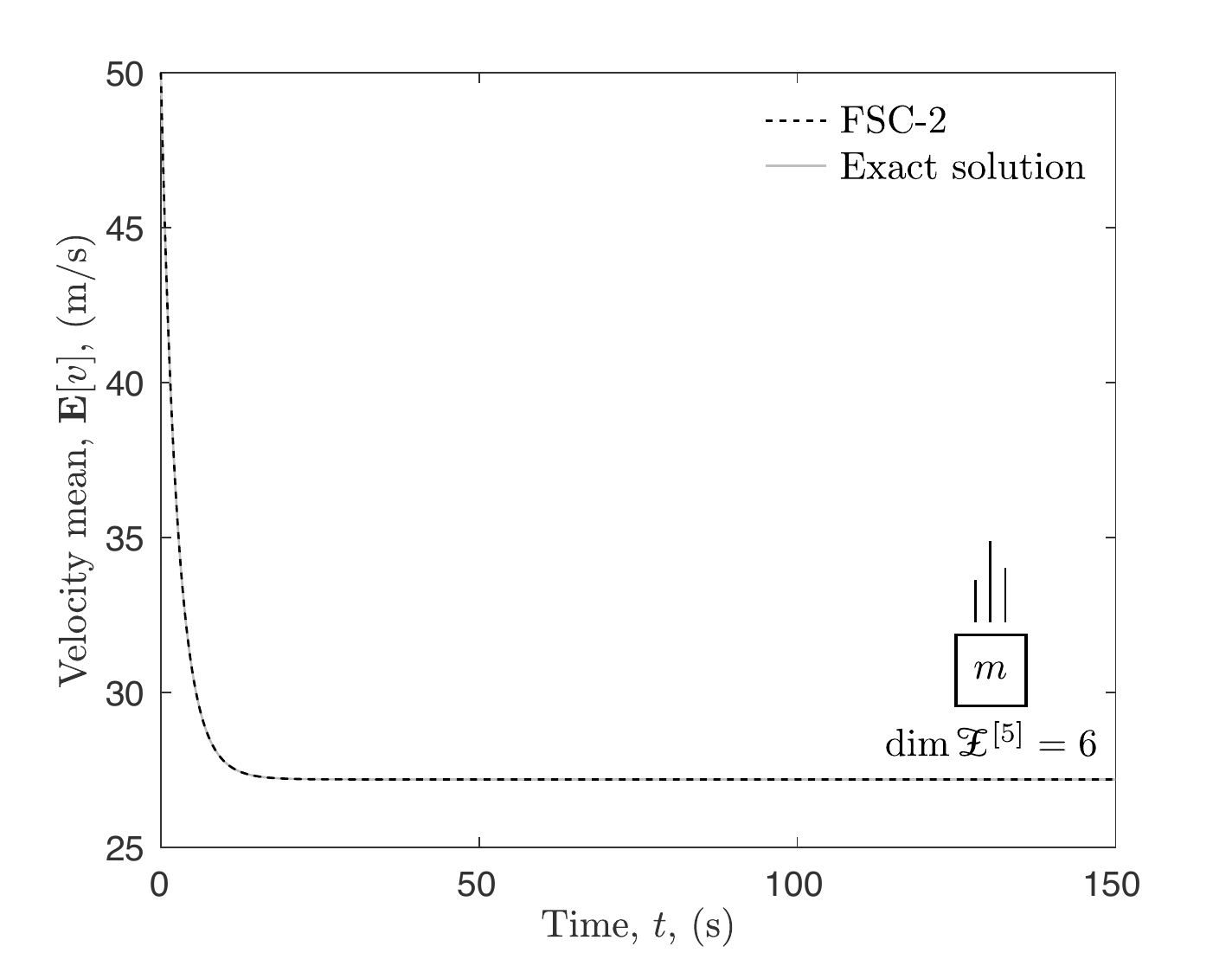}
\caption{Mean}
\label{fig2System1_Uniform_FSC_Vel_Mean_6}
\end{subfigure}\hfill
\begin{subfigure}[b]{0.495\textwidth}
\includegraphics[width=\textwidth]{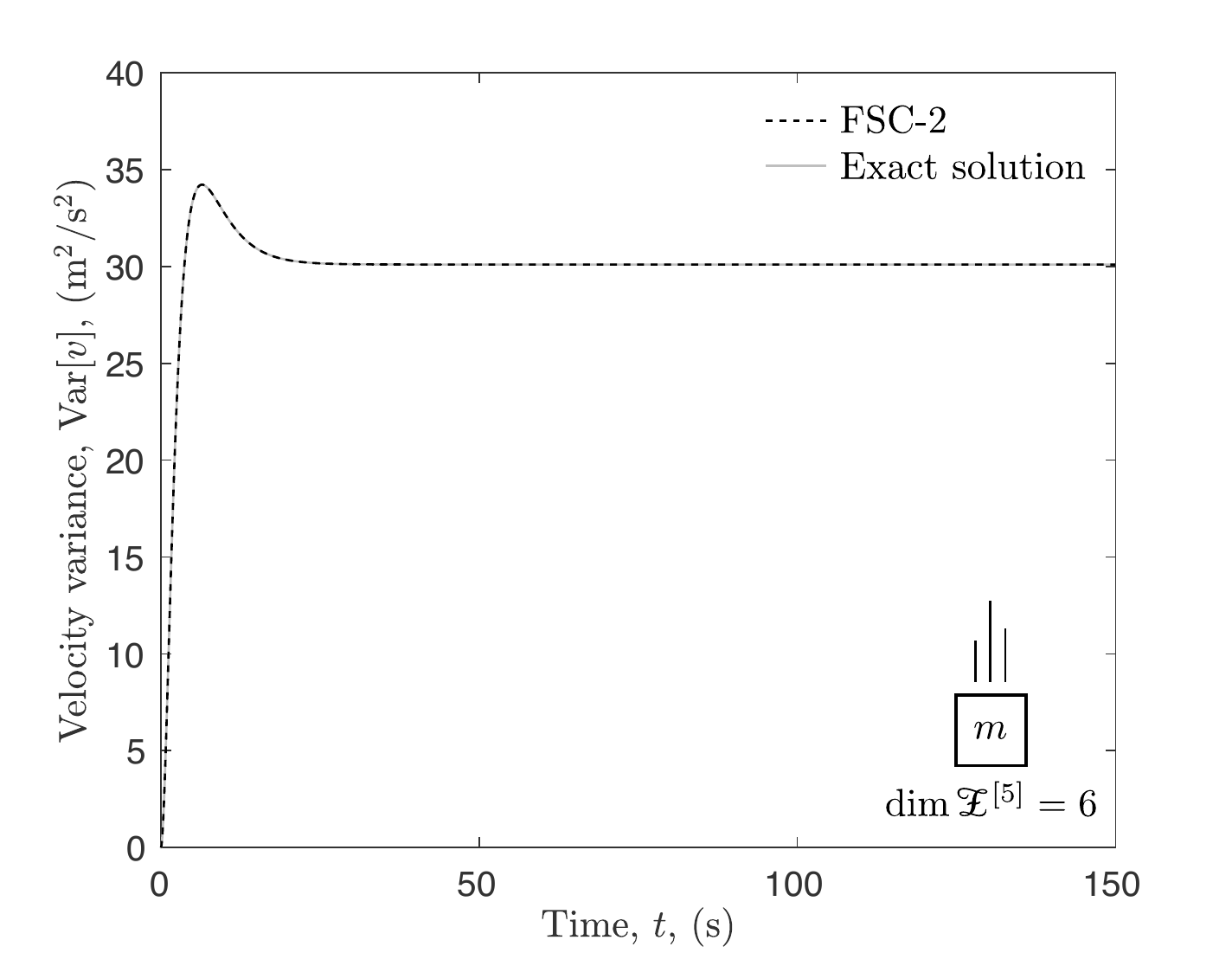}
\caption{Variance}
\label{fig2System1_Uniform_FSC_Vel_Var_6}
\end{subfigure}
\caption{\emph{Problem 1} --- Evolution of $\mathbf{E}[v]$ and $\mathrm{Var}[v]$ for the case when the $p$-discretization level of RFS is $\mathscr{Z}^{[5]}$ and $\mu\sim\mathrm{Uniform}$}
\label{fig2System1_Uniform_FSC_Vel_6}
\end{figure}

\begin{figure}
\centering
\begin{subfigure}[b]{0.495\textwidth}
\includegraphics[width=\textwidth]{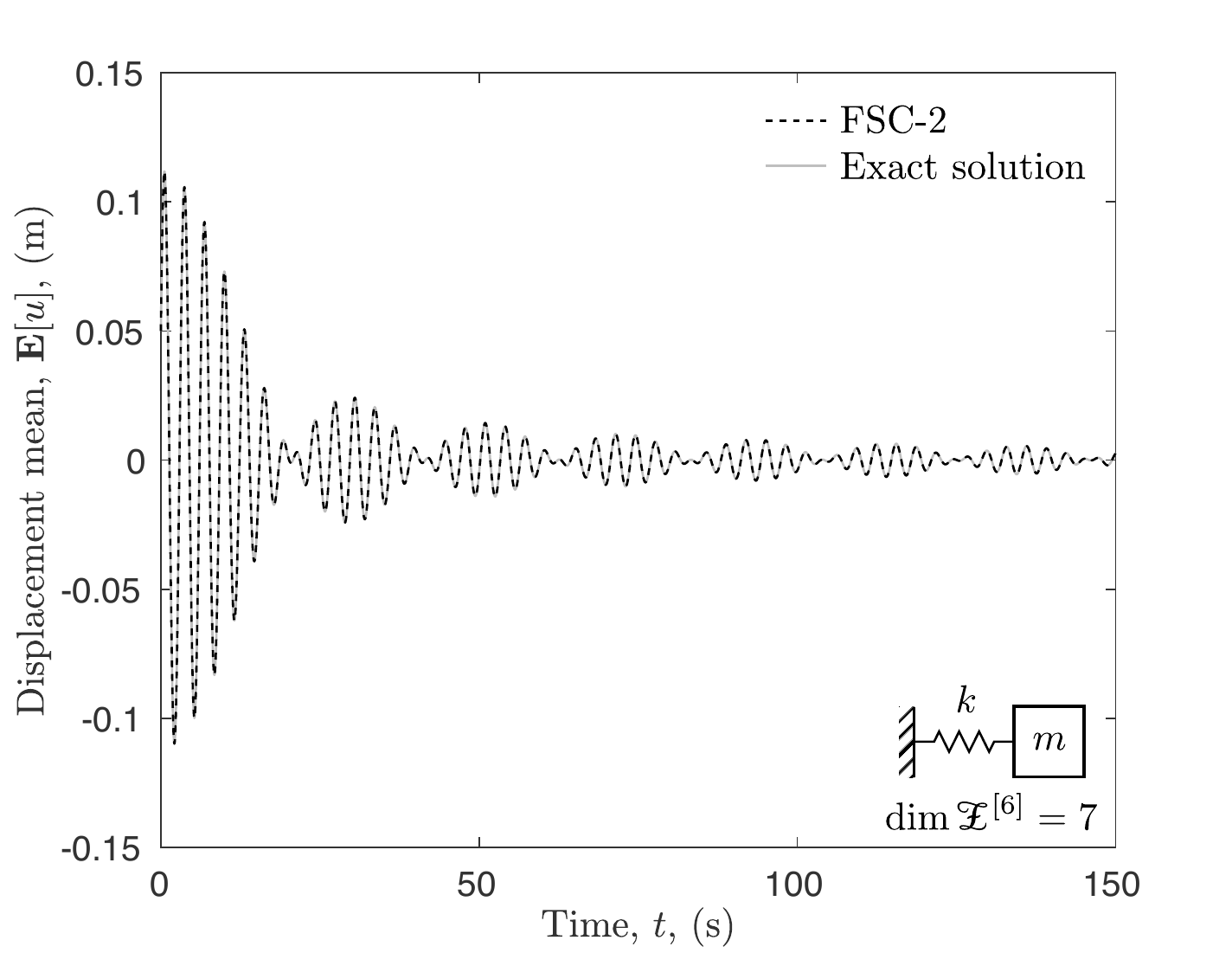}
\caption{Mean}
\label{fig2System21_Uniform_FSC_Disp_Mean_7}
\end{subfigure}\hfill
\begin{subfigure}[b]{0.495\textwidth}
\includegraphics[width=\textwidth]{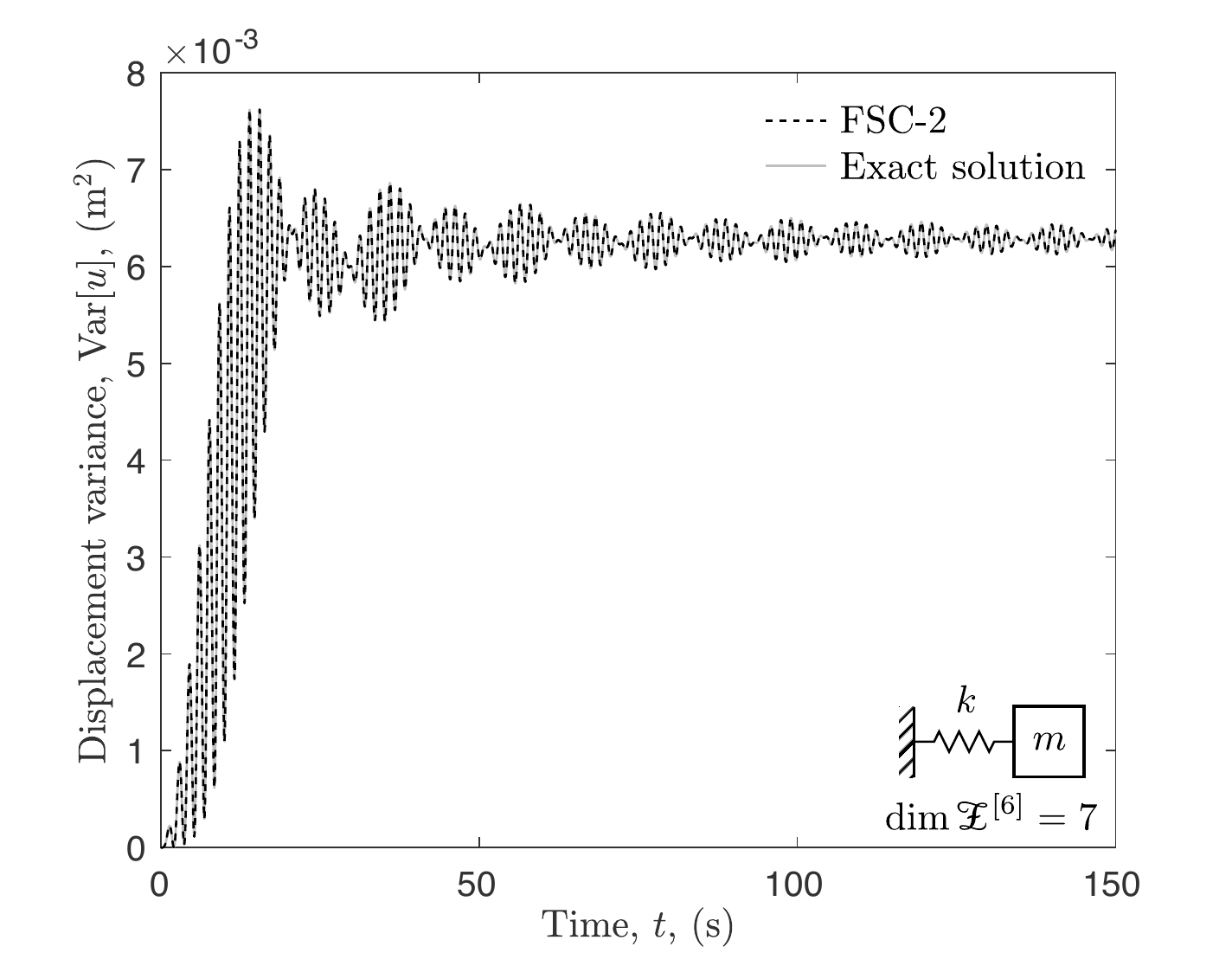}
\caption{Variance}
\label{fig2System21_Uniform_FSC_Disp_Var_7}
\end{subfigure}
\caption{\emph{Problem 2} --- Evolution of $\mathbf{E}[u]$ and $\mathrm{Var}[u]$ for the case when the $p$-discretization level of RFS is $\mathscr{Z}^{[6]}$ and $\mu\sim\mathrm{Uniform}$}
\label{fig2System21_Uniform_FSC_Disp_7}
\end{figure}

\begin{figure}
\centering
\begin{subfigure}[b]{0.495\textwidth}
\includegraphics[width=\textwidth]{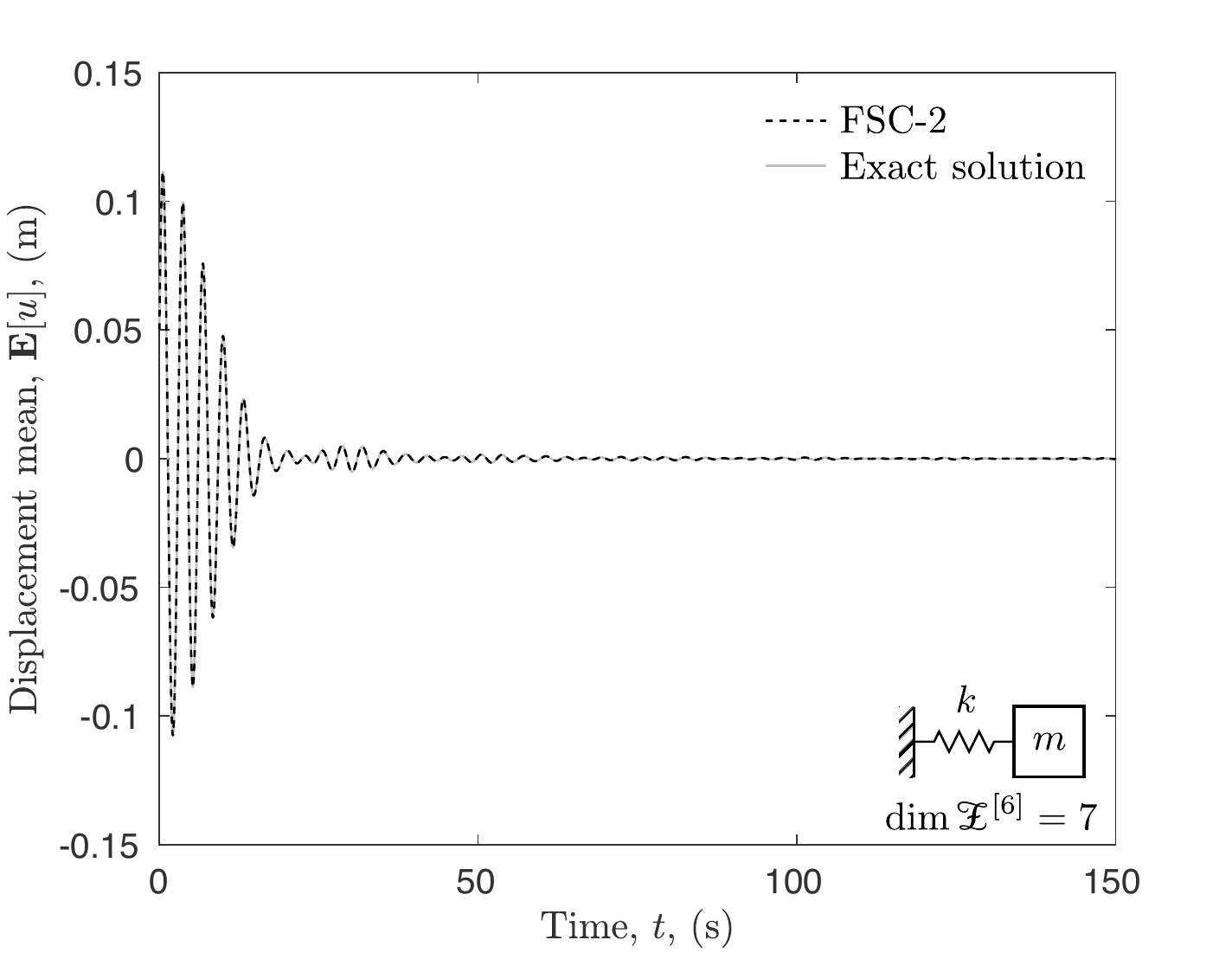}
\caption{Mean}
\label{fig2System22_UniformUniform_FSC_Disp_Mean_7}
\end{subfigure}\hfill
\begin{subfigure}[b]{0.495\textwidth}
\includegraphics[width=\textwidth]{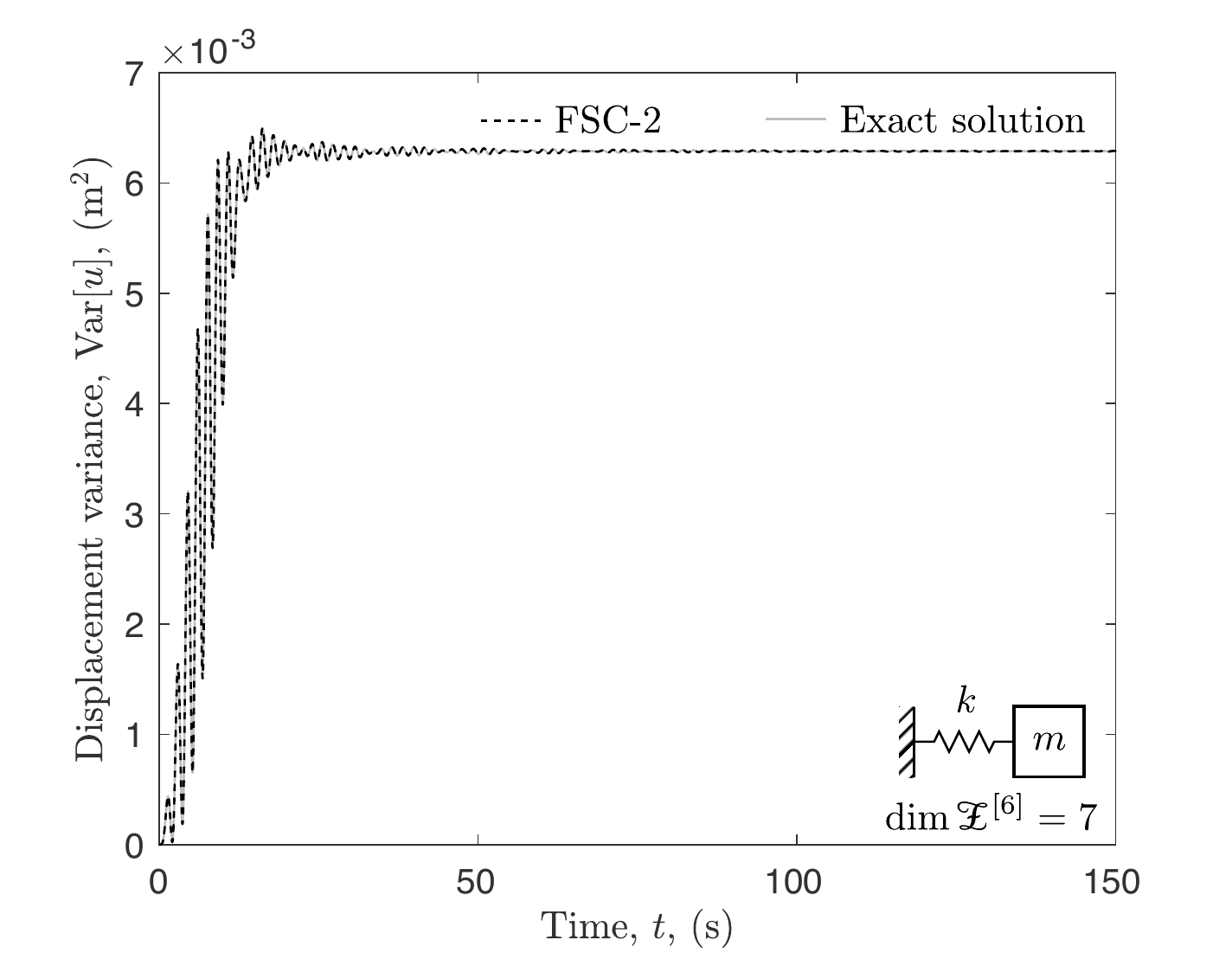}
\caption{Variance}
\label{fig2System22_UniformUniform_FSC_Disp_Var_7}
\end{subfigure}
\caption{\emph{Problem 3} --- Evolution of $\mathbf{E}[u]$ and $\mathrm{Var}[u]$ for the case when the $p$-discretization level of RFS is $\mathscr{Z}^{[6]}$ and $\mu\sim\mathrm{Uniform}\otimes\mathrm{Uniform}$}
\label{fig2System22_UniformUniform_FSC_Disp_7}
\end{figure}

\begin{figure}
\centering
\begin{subfigure}[b]{0.495\textwidth}
\includegraphics[width=\textwidth]{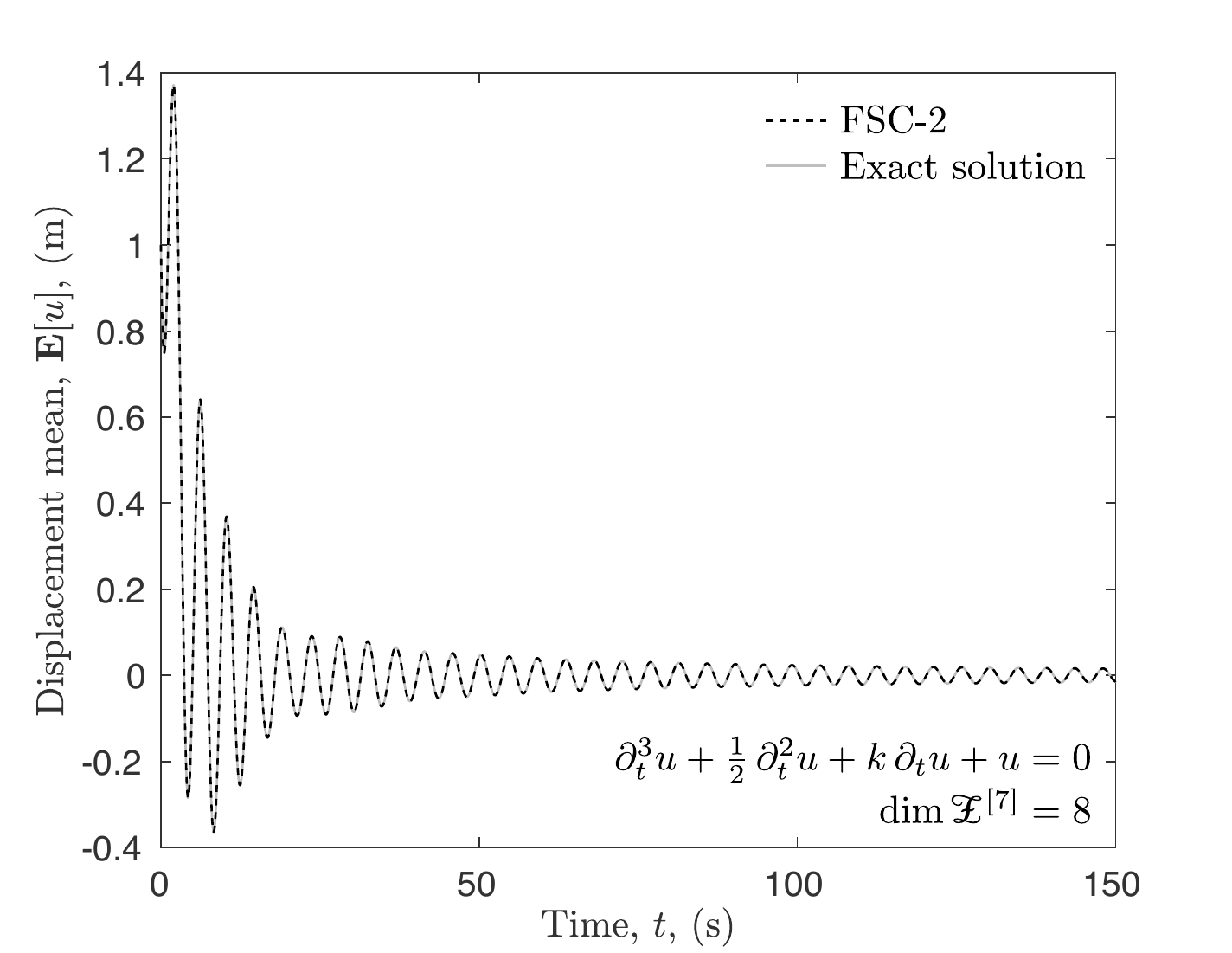}
\caption{Mean}
\label{fig2System3_Uniform_FSC_Disp_Mean_8}
\end{subfigure}\hfill
\begin{subfigure}[b]{0.495\textwidth}
\includegraphics[width=\textwidth]{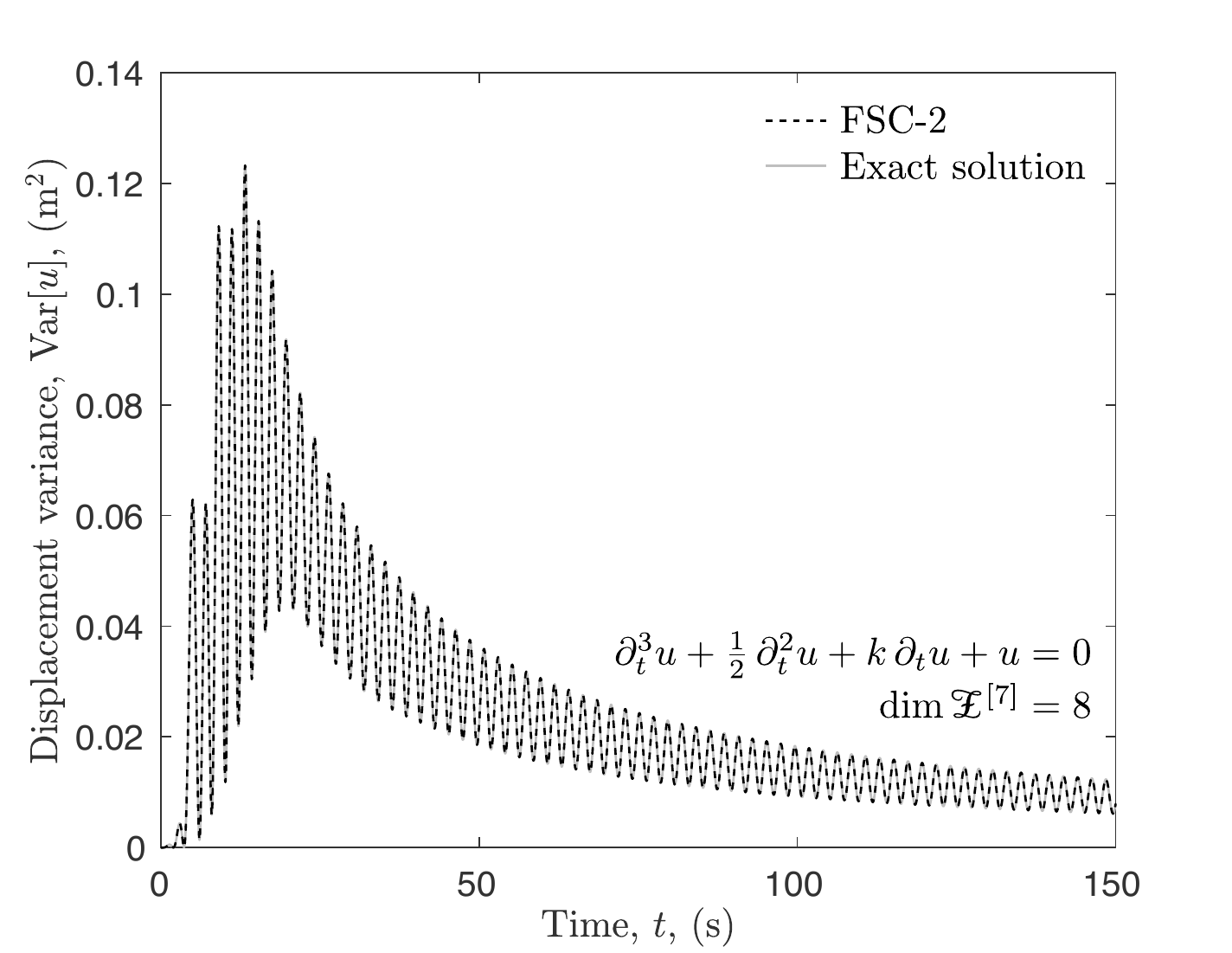}
\caption{Variance}
\label{fig2System3_Uniform_FSC_Disp_Var_8}
\end{subfigure}
\caption{\emph{Problem 4} --- Evolution of $\mathbf{E}[u]$ and $\mathrm{Var}[u]$ for the case when the $p$-discretization level of RFS is $\mathscr{Z}^{[7]}$ and $\mu\sim\mathrm{Uniform}$}
\label{fig2System3_Uniform_FSC_Disp_8}
\end{figure}

\begin{figure}
\centering
\begin{subfigure}[b]{0.495\textwidth}
\includegraphics[width=\textwidth]{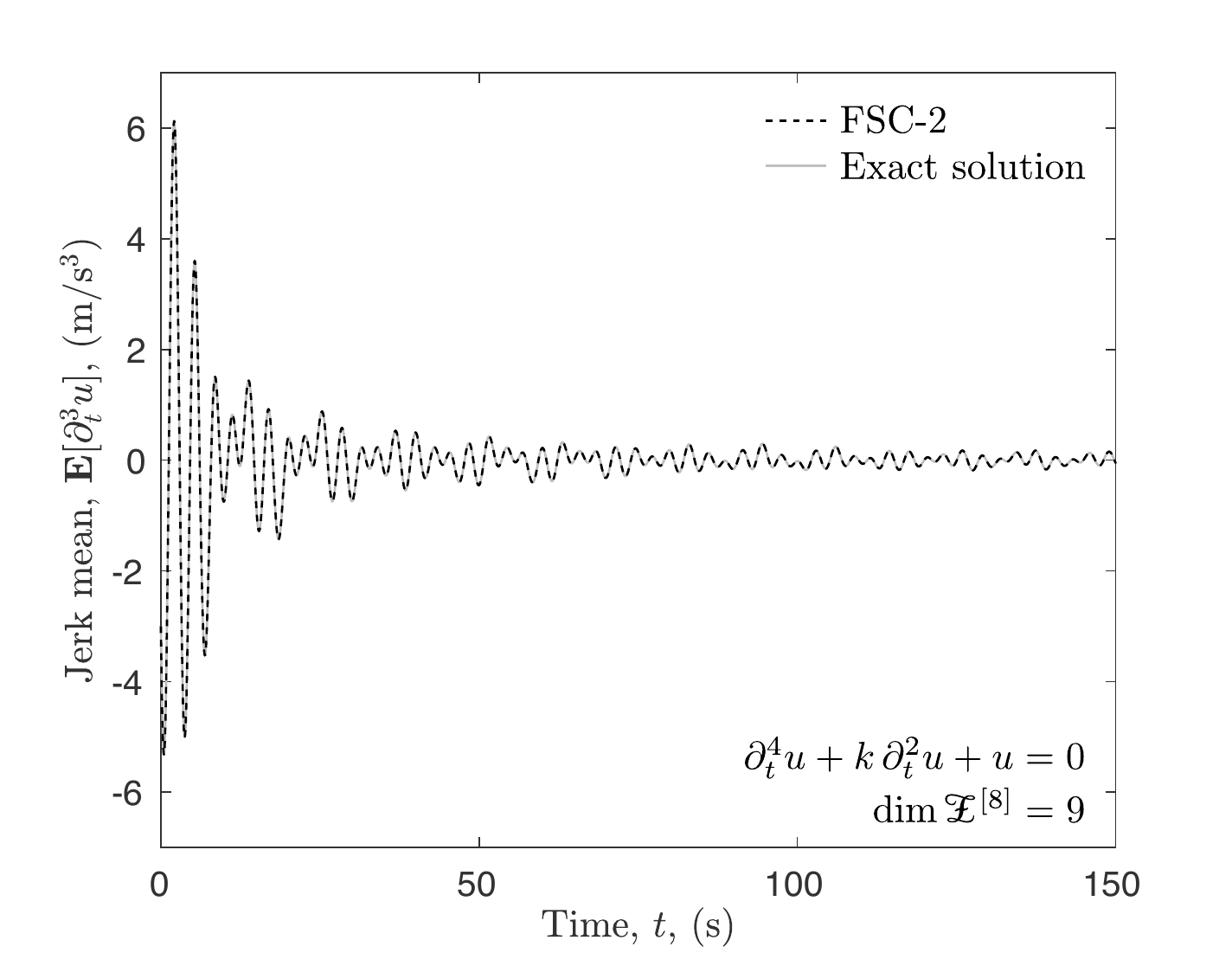}
\caption{Mean}
\label{fig2System4_Uniform_FSC_Jerk_Mean_9}
\end{subfigure}\hfill
\begin{subfigure}[b]{0.495\textwidth}
\includegraphics[width=\textwidth]{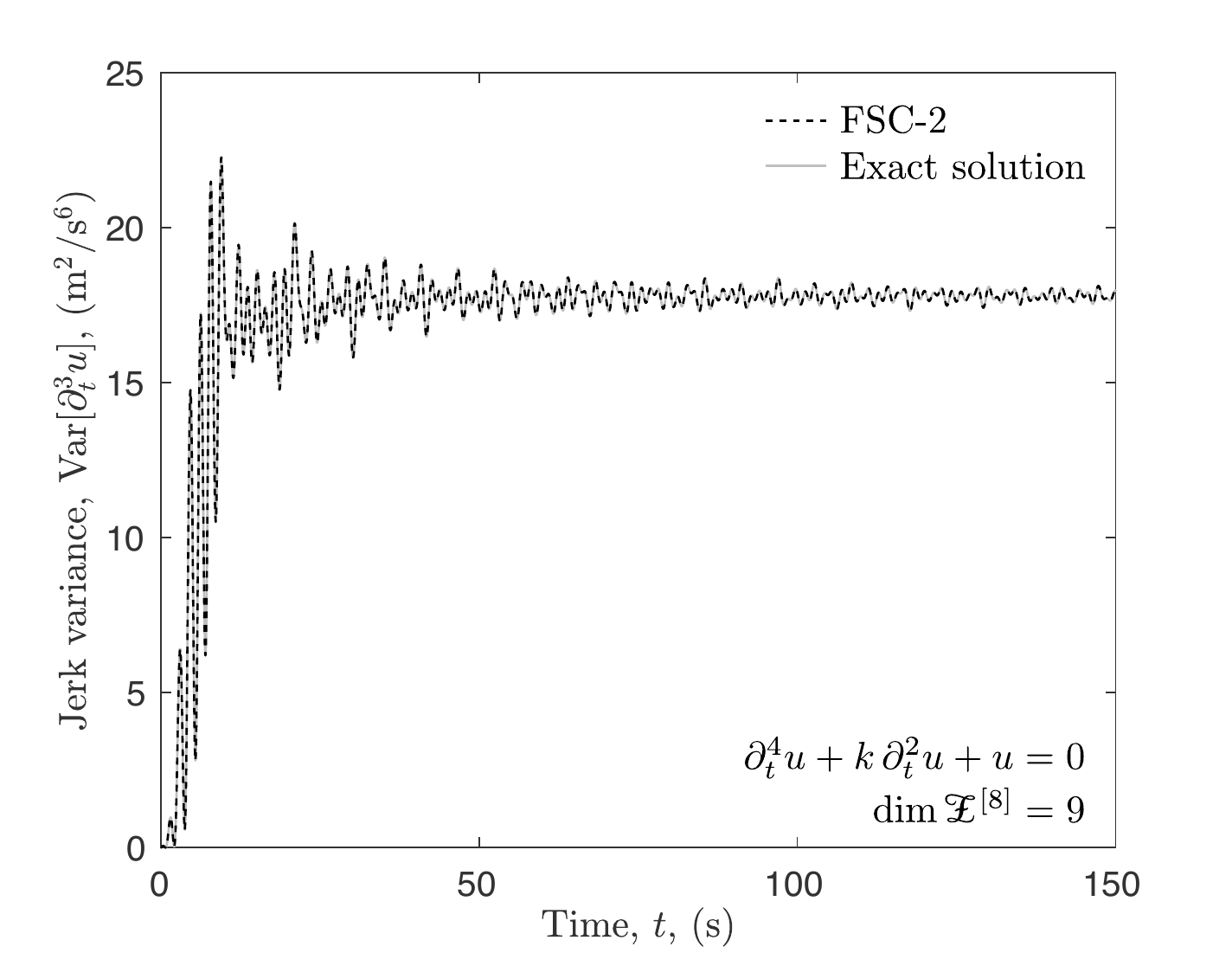}
\caption{Variance}
\label{fig2System4_Uniform_FSC_Jerk_Var_9}
\end{subfigure}
\caption{\emph{Problem 5} --- Evolution of $\mathbf{E}[\partial_t^3u]$ and $\mathrm{Var}[\partial_t^3u]$ for the case when the $p$-discretization level of RFS is $\mathscr{Z}^{[8]}$ and $\mu\sim\mathrm{Uniform}$}
\label{fig2System4_Uniform_FSC_Jerk_9}
\end{figure}

Figs.~\ref{fig2System1_Uniform_FSC_Vel_Error} to \ref{fig2System22_UniformUniform_FSC_Disp_Error} present the local errors in mean and variance of each of the responses mentioned above but only for Problems 1 and 2 for sake of brevity.
The errors are depicted for both FSC-1 and FSC-2.
To compare, we also include the case when $P=n$ even though the FSC scheme requires that $P$ is taken at least equal to $n+1$.
We do so to test the implications of spanning the RFS with the state variables of the system only.
The cases when $P=n+2$ and $P=n+4$ are also provided for the sake of comparison.
From these figures, it is apparent that as the number of basis vectors increases, so does the accuracy of the results.
In particular, when the FSC-1 approach is used, the following observation can be made.
By increasing the number of basis vectors from $n+1$ to $n+3$, the accuracy of the results improves significantly by 6 orders of magnitude.
However, when the number of basis vectors is increased from $n+3$ to $n+5$, the results either do not improve noticeably or worsen a bit (as in Fig.~\ref{fig2System22_UniformUniform_FSC_Disp_Error}).
This is in contrast to the FSC-2 approach.
When FSC-2 is used, the accuracy of the results improves not only significantly but also consistently as the number of basis vectors increases.
The figures also indicate that FSC-2 can achieve in general a higher level of accuracy than FSC-1 as time progresses in the simulation.
However, we do notice that whenever $P=n$, no difference between the two approaches can be discerned.

\begin{figure}
\centering
\begin{subfigure}[b]{0.495\textwidth}
\includegraphics[width=\textwidth]{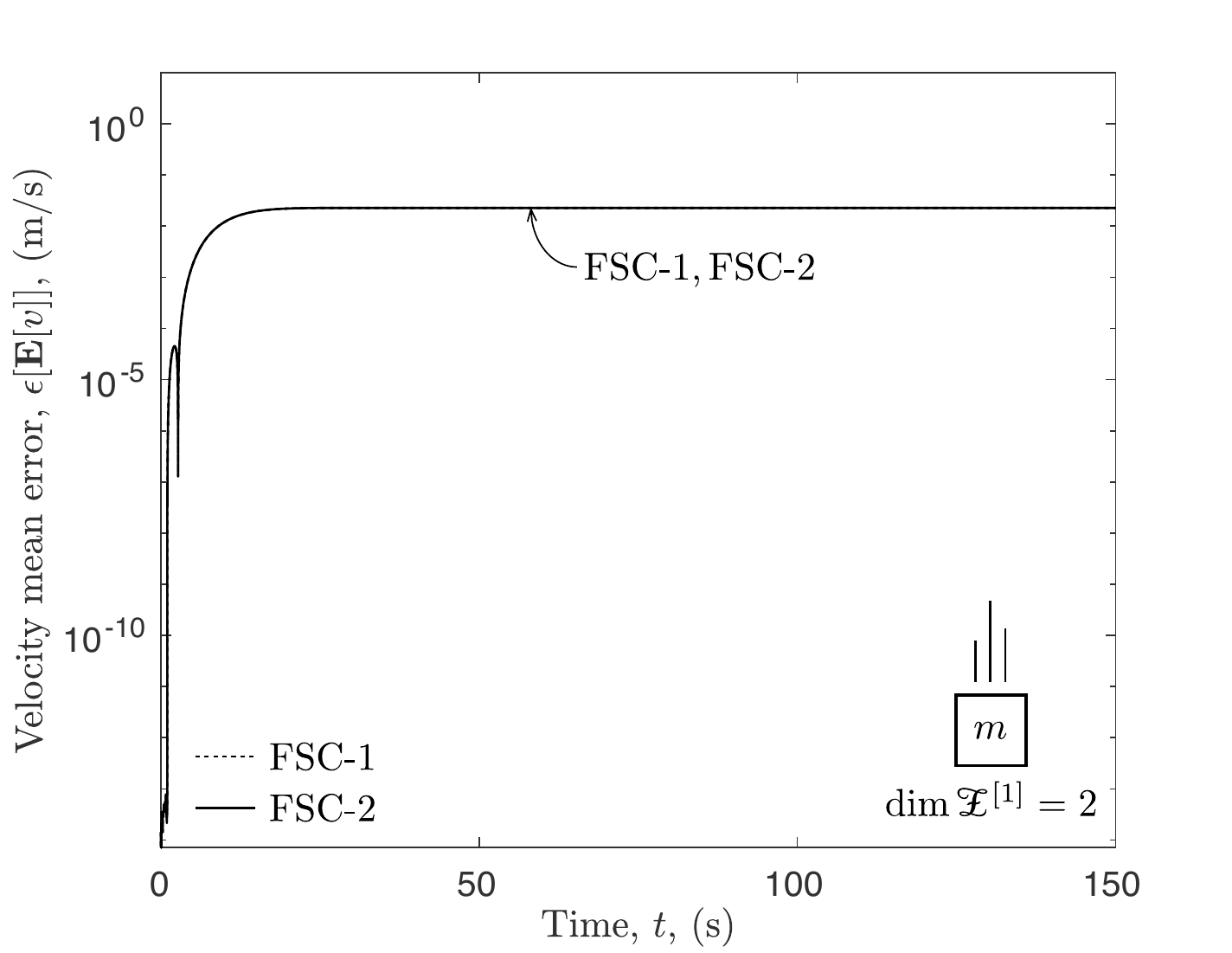}
\caption{Mean error for $\mathscr{Z}^{[1]}$}
\label{fig2System1_Uniform_FSC_Vel_Mean_2_Error}
\end{subfigure}\hfill
\begin{subfigure}[b]{0.495\textwidth}
\includegraphics[width=\textwidth]{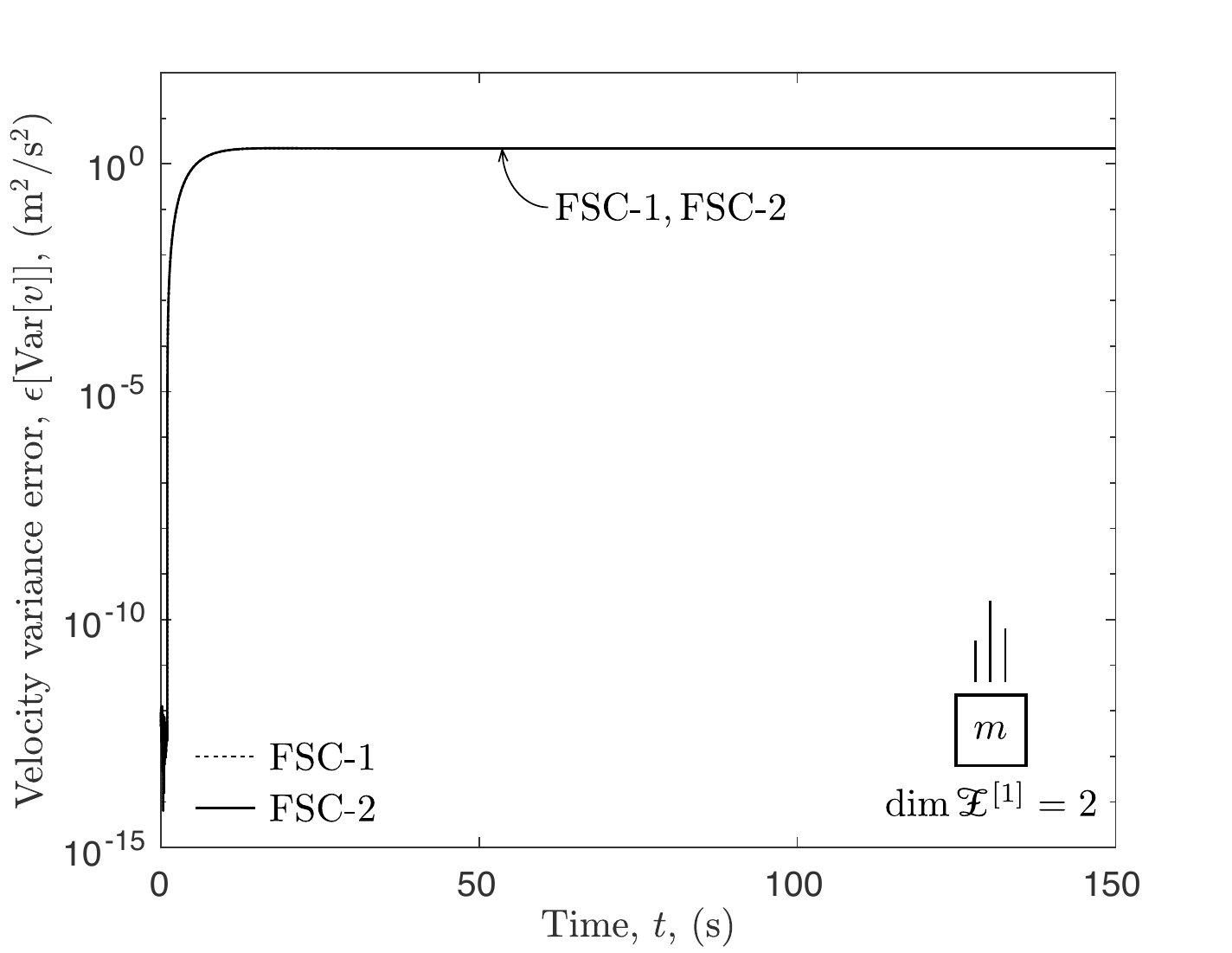}
\caption{Variance error for $\mathscr{Z}^{[1]}$}
\label{fig2System1_Uniform_FSC_Vel_Var_2_Error}
\end{subfigure}\quad
\begin{subfigure}[b]{0.495\textwidth}
\includegraphics[width=\textwidth]{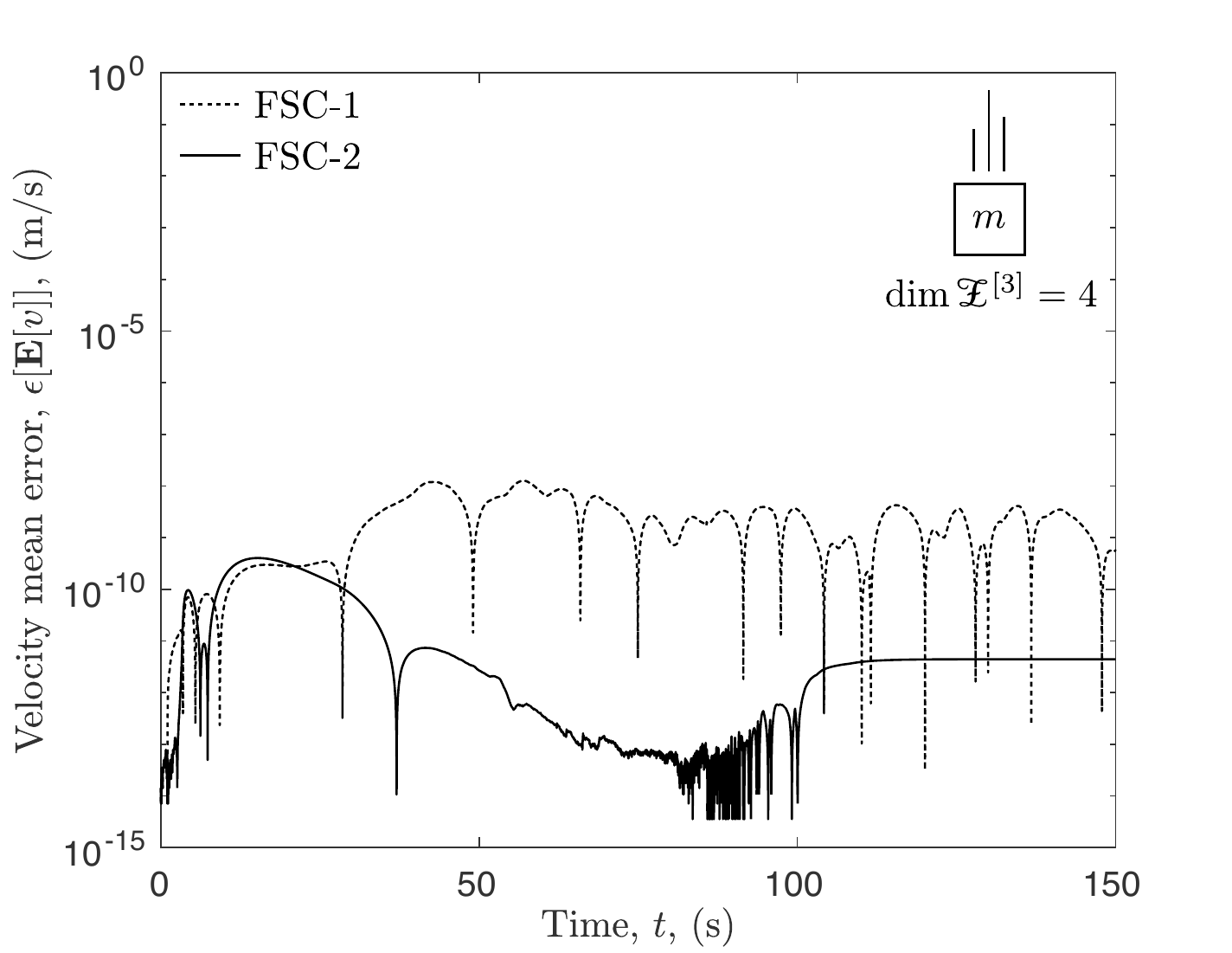}
\caption{Mean error for $\mathscr{Z}^{[3]}$}
\label{fig2System1_Uniform_FSC_Vel_Mean_4_Error}
\end{subfigure}\hfill
\begin{subfigure}[b]{0.495\textwidth}
\includegraphics[width=\textwidth]{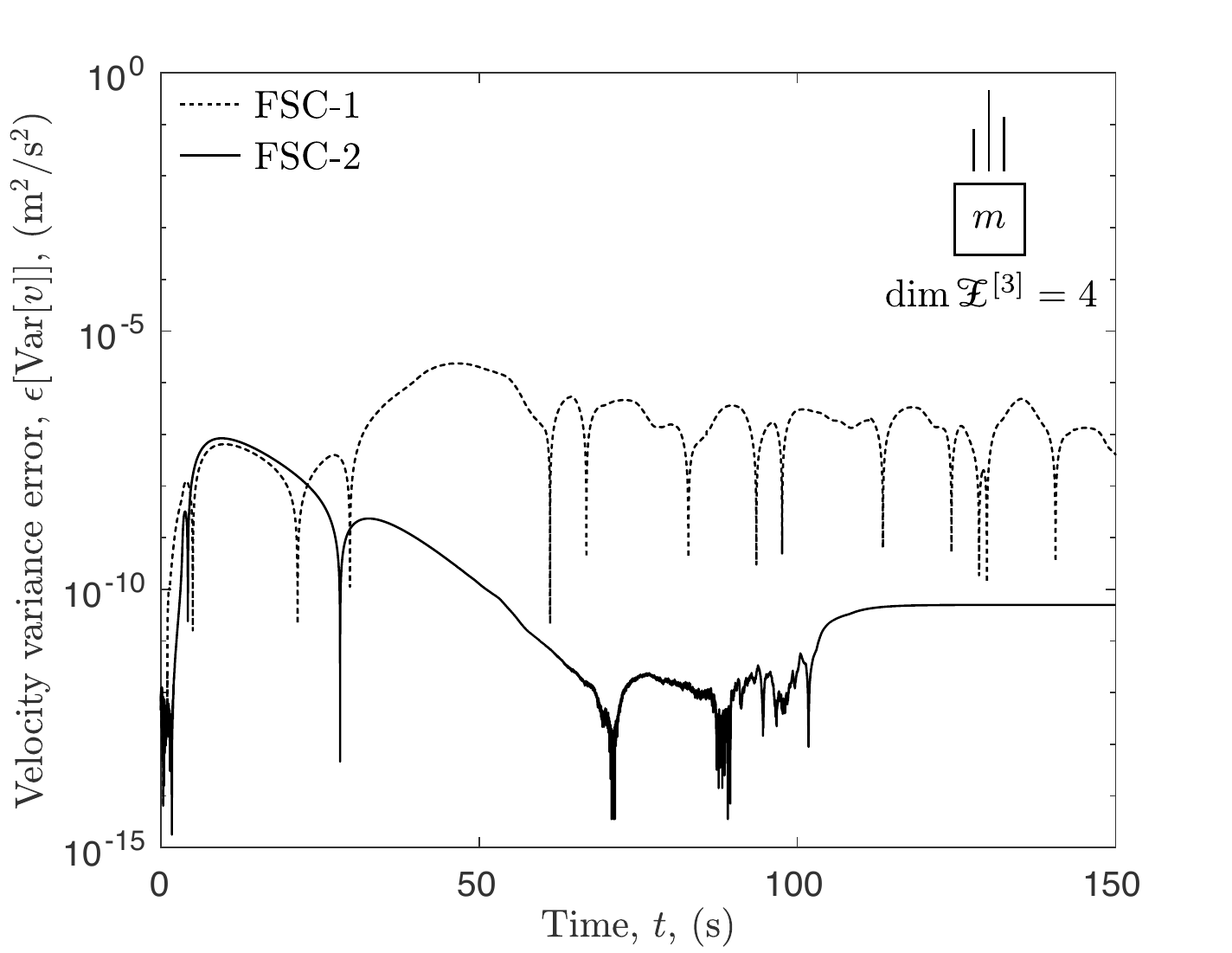}
\caption{Variance error for $\mathscr{Z}^{[3]}$}
\label{fig2System1_Uniform_FSC_Vel_Var_4_Error}
\end{subfigure}\quad
\begin{subfigure}[b]{0.495\textwidth}
\includegraphics[width=\textwidth]{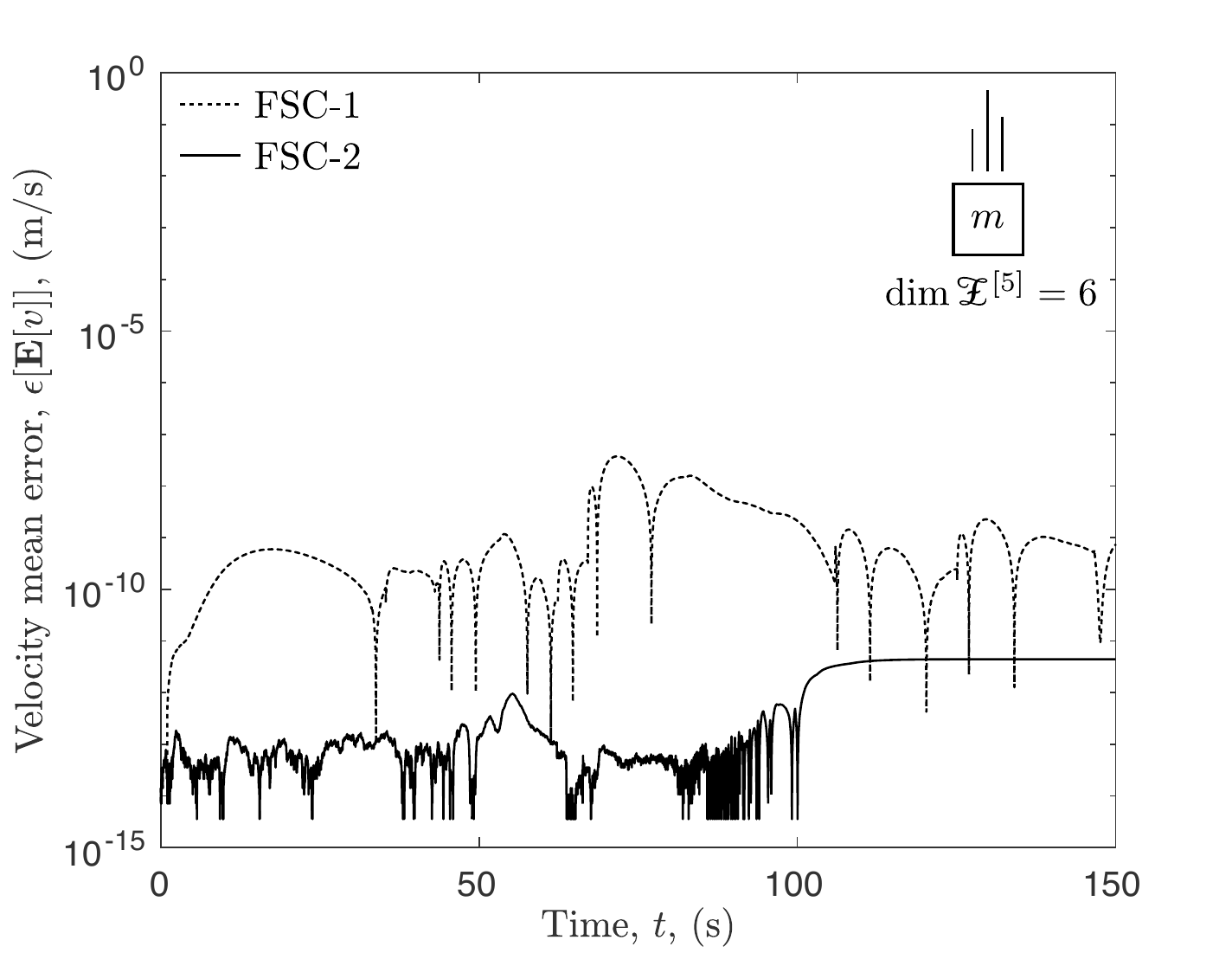}
\caption{Mean error for $\mathscr{Z}^{[5]}$}
\label{fig2System1_Uniform_FSC_Vel_Mean_6_Error}
\end{subfigure}\hfill
\begin{subfigure}[b]{0.495\textwidth}
\includegraphics[width=\textwidth]{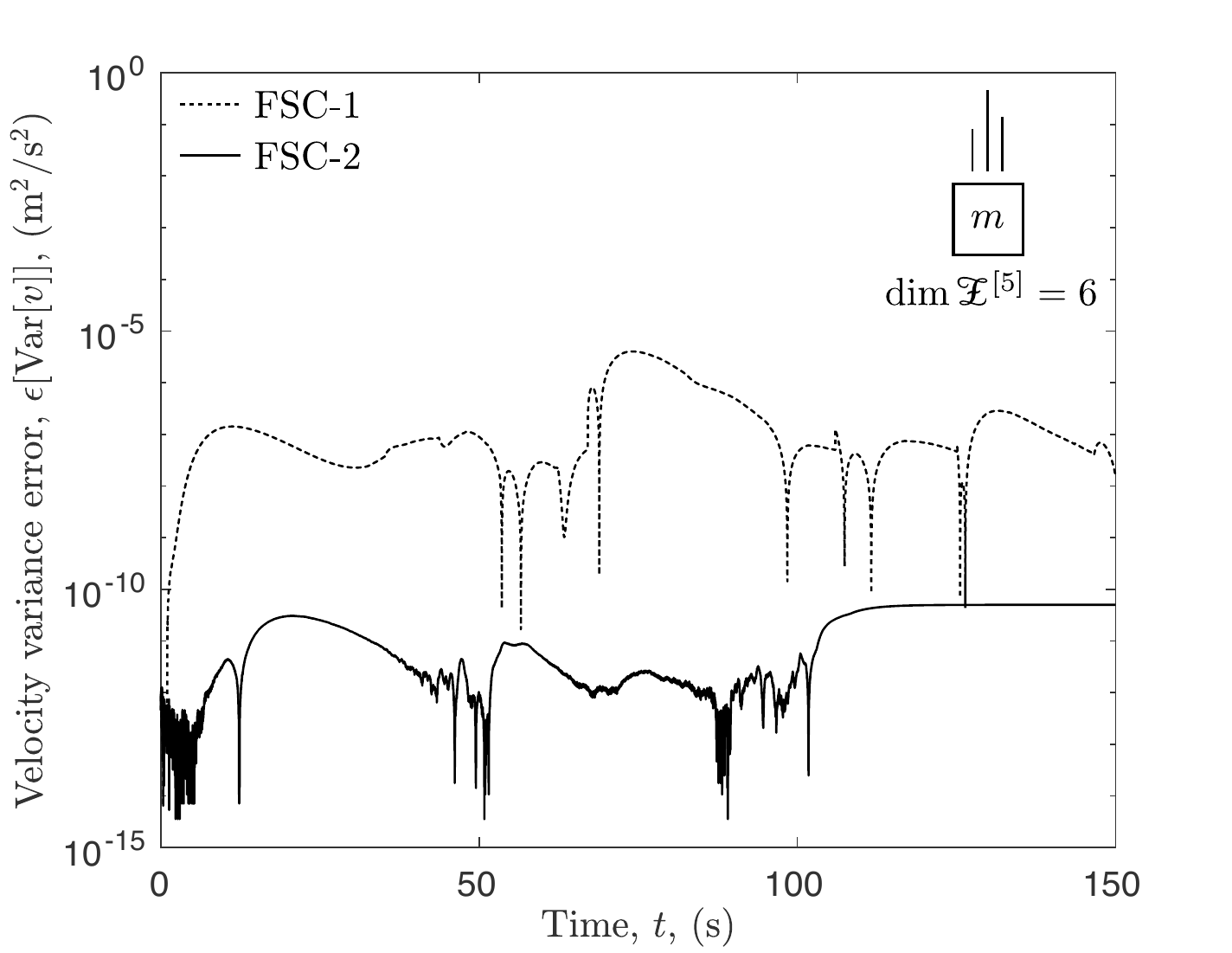}
\caption{Variance error for $\mathscr{Z}^{[5]}$}
\label{fig2System1_Uniform_FSC_Vel_Var_6_Error}
\end{subfigure}
\caption{\emph{Problem 1} --- Local error evolution of $\mathbf{E}[v]$ and $\mathrm{Var}[v]$ for different $p$-discretization levels of RFS and for $\mu\sim\mathrm{Uniform}$}
\label{fig2System1_Uniform_FSC_Vel_Error}
\end{figure}

\begin{figure}
\centering
\begin{subfigure}[b]{0.495\textwidth}
\includegraphics[width=\textwidth]{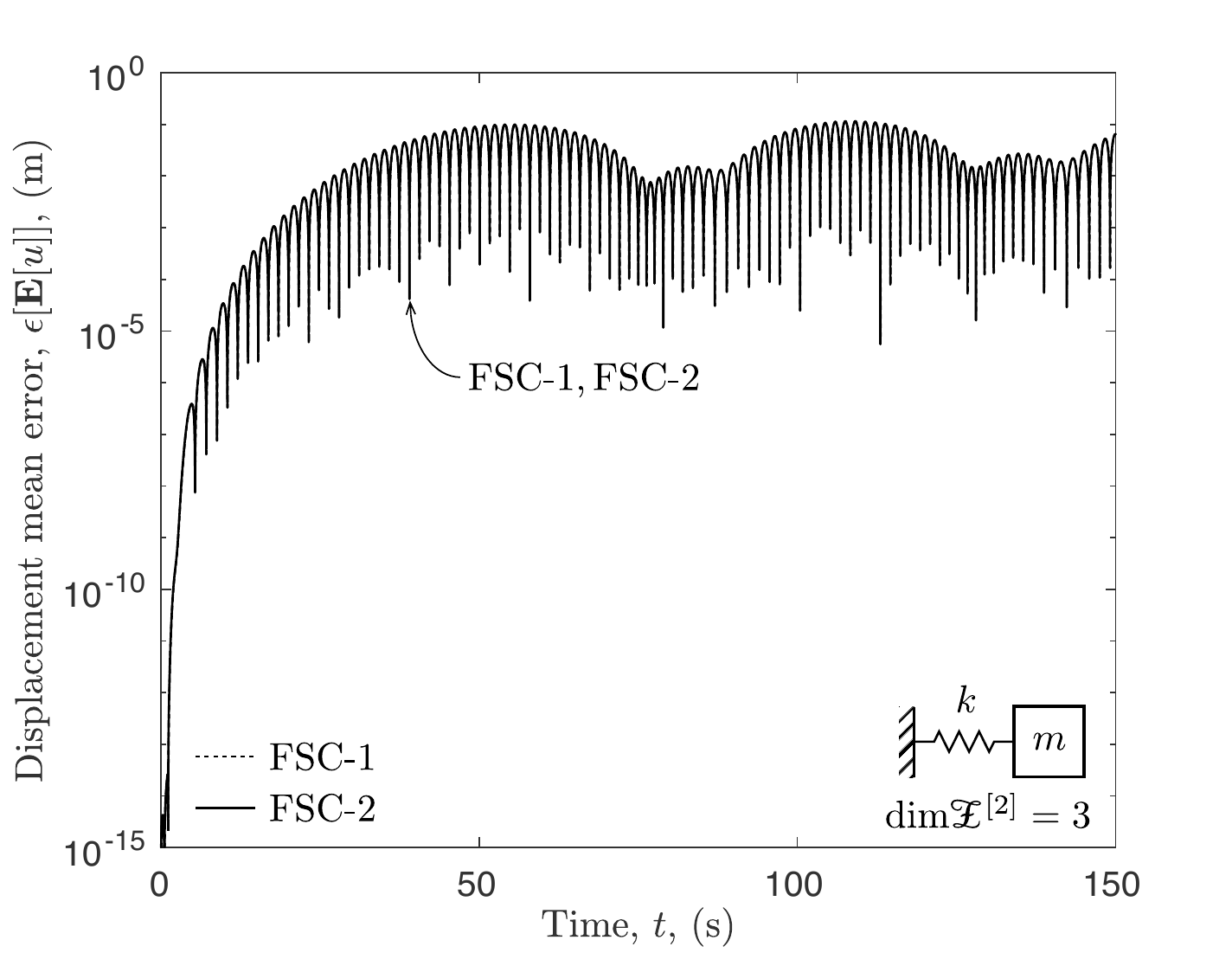}
\caption{Mean error for $\mathscr{Z}^{[2]}$}
\label{fig2System21_Uniform_FSC_Disp_Mean_3_Error}
\end{subfigure}\hfill
\begin{subfigure}[b]{0.495\textwidth}
\includegraphics[width=\textwidth]{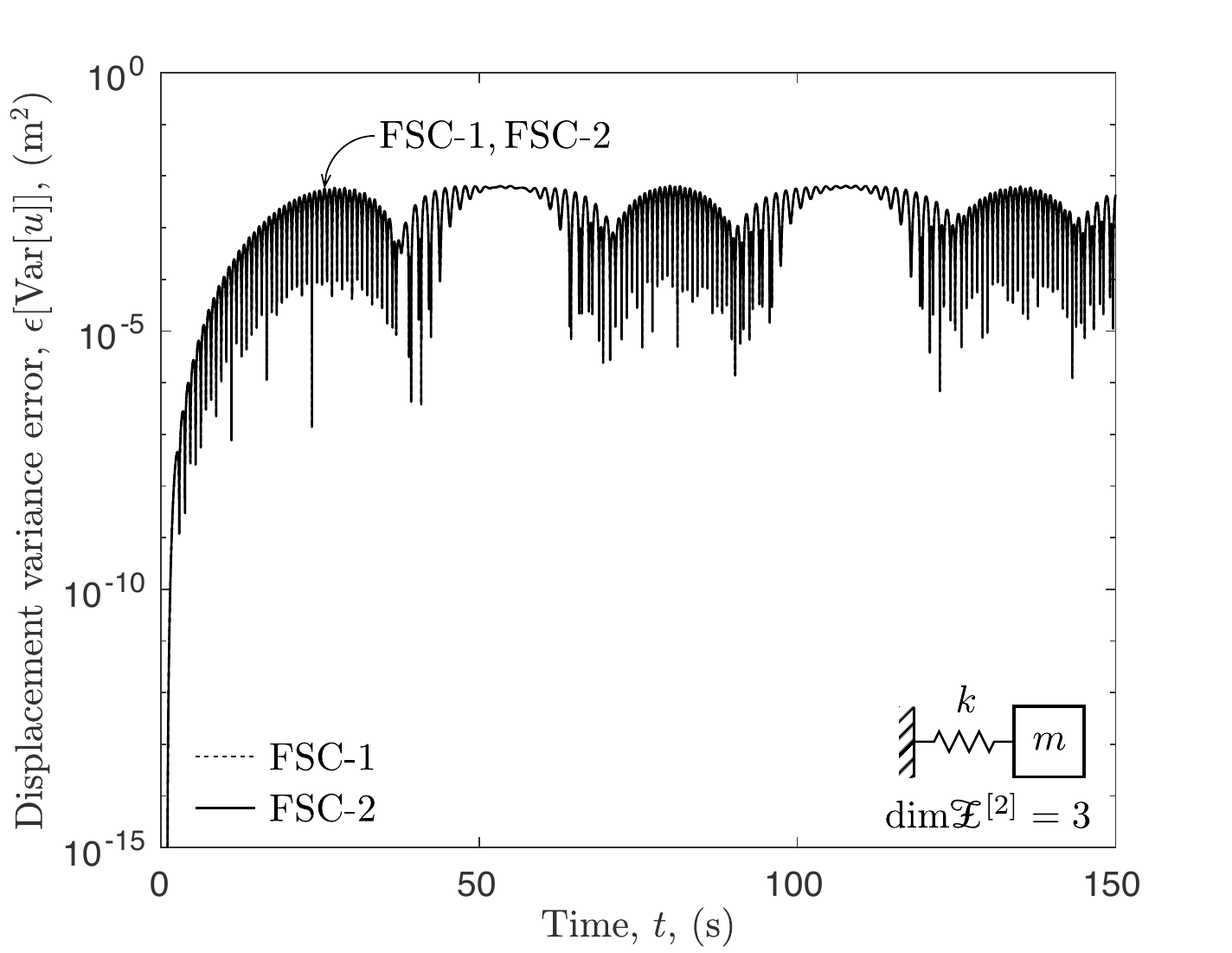}
\caption{Variance error for $\mathscr{Z}^{[2]}$}
\label{fig2System21_Uniform_FSC_Disp_Var_3_Error}
\end{subfigure}\quad
\begin{subfigure}[b]{0.495\textwidth}
\includegraphics[width=\textwidth]{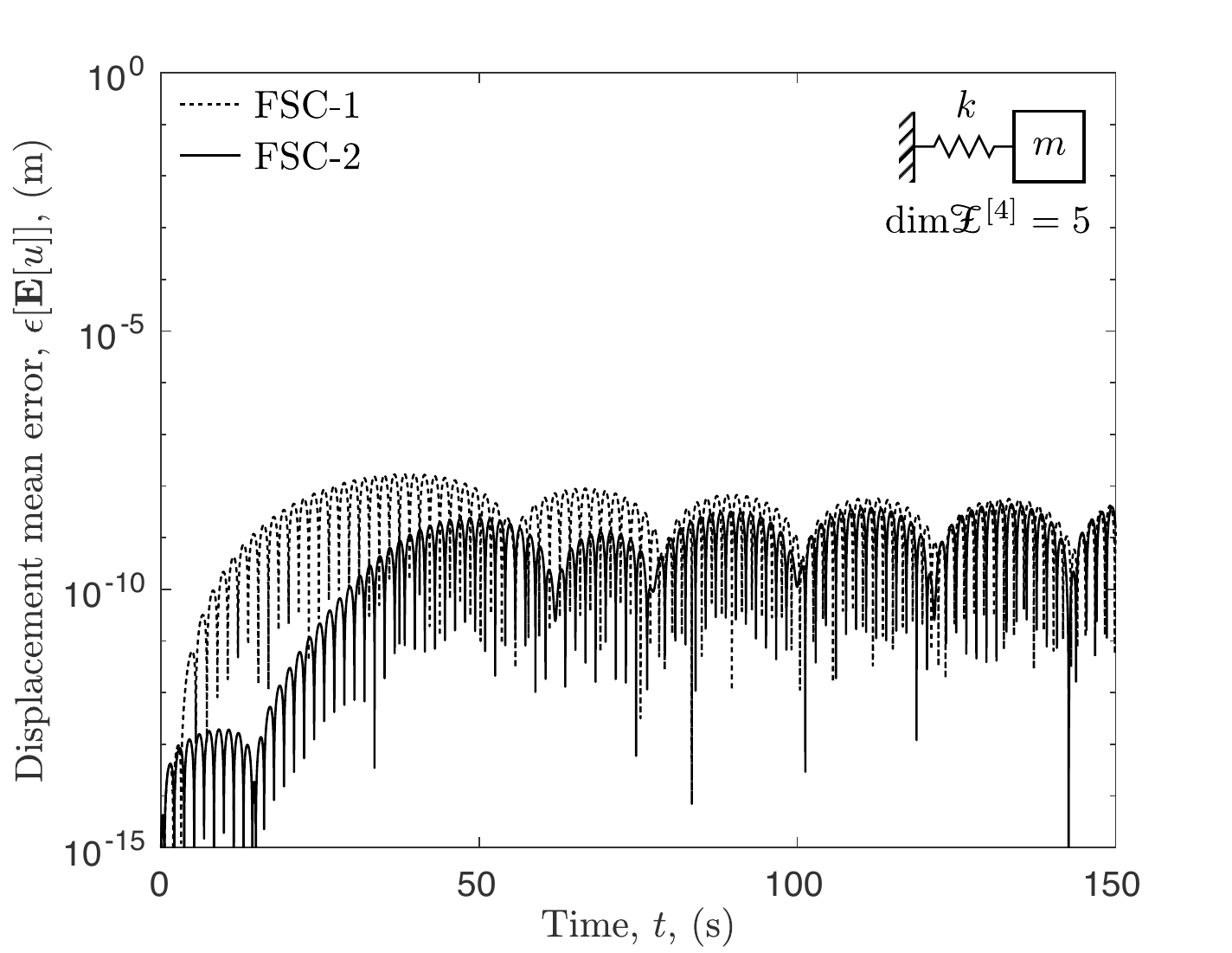}
\caption{Mean error for $\mathscr{Z}^{[4]}$}
\label{fig2System21_Uniform_FSC_Disp_Mean_5_Error}
\end{subfigure}\hfill
\begin{subfigure}[b]{0.495\textwidth}
\includegraphics[width=\textwidth]{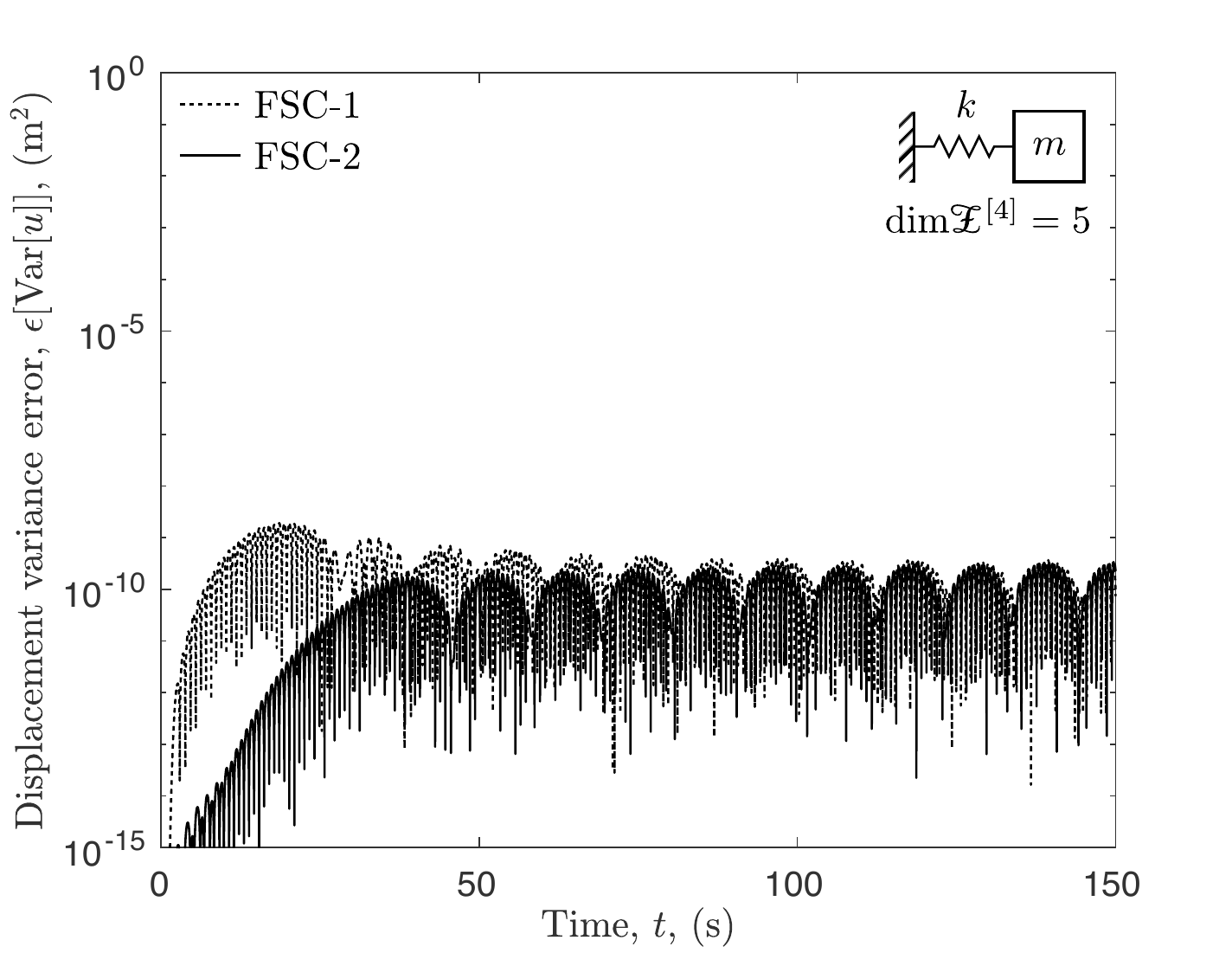}
\caption{Variance error for $\mathscr{Z}^{[4]}$}
\label{fig2System21_Uniform_FSC_Disp_Var_5_Error}
\end{subfigure}\quad
\begin{subfigure}[b]{0.495\textwidth}
\includegraphics[width=\textwidth]{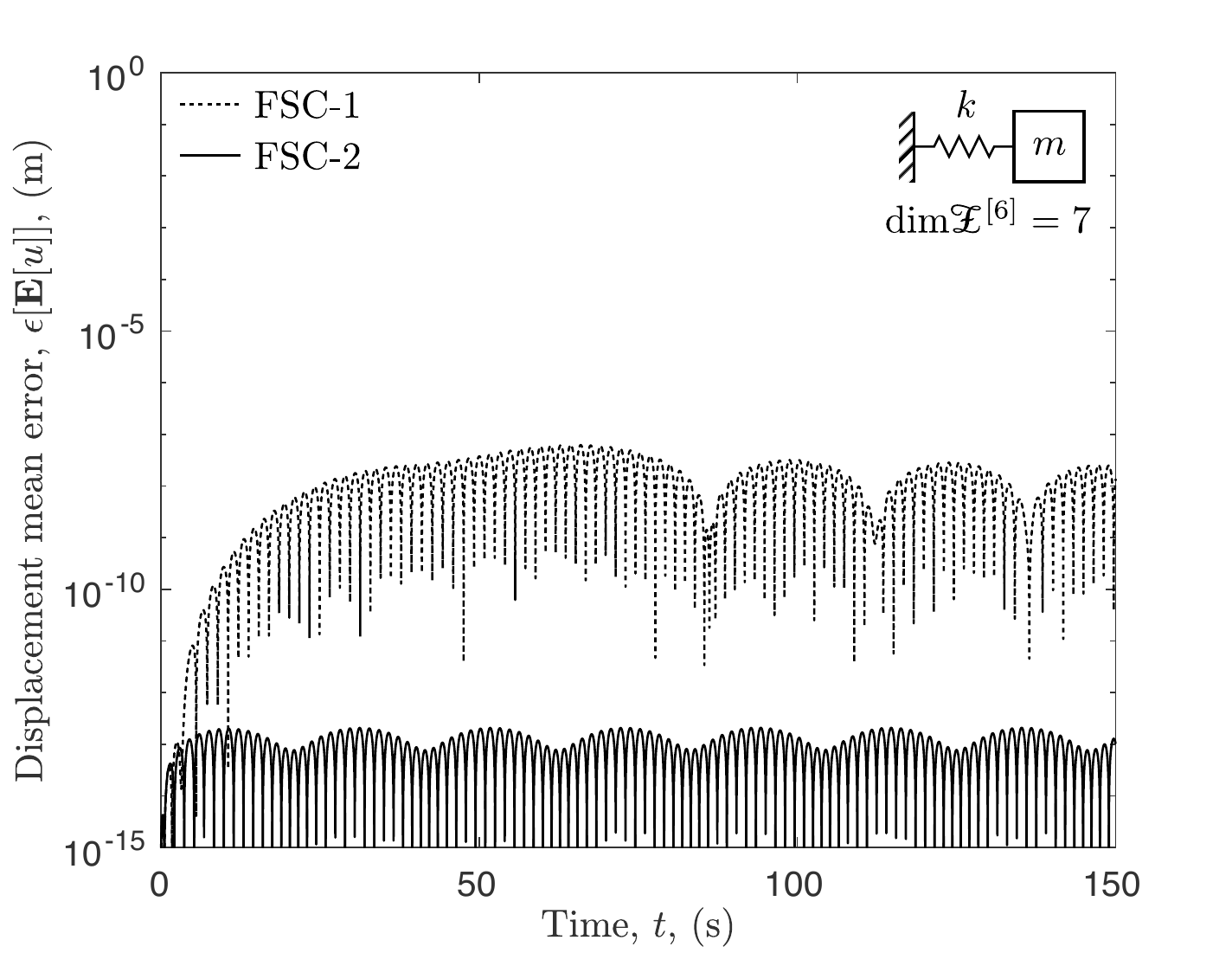}
\caption{Mean error for $\mathscr{Z}^{[6]}$}
\label{fig2System21_Uniform_FSC_Disp_Mean_7_Error}
\end{subfigure}\hfill
\begin{subfigure}[b]{0.495\textwidth}
\includegraphics[width=\textwidth]{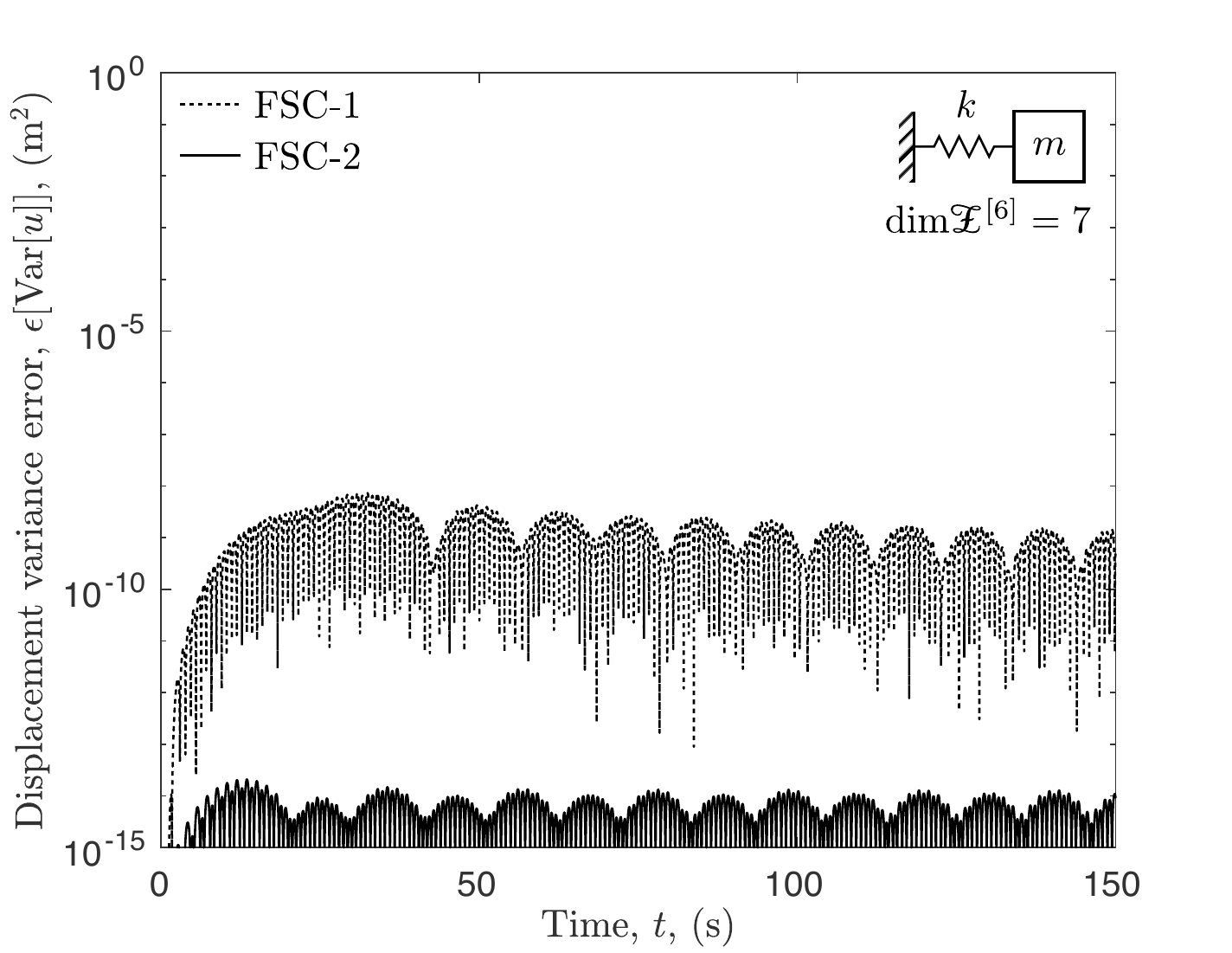}
\caption{Variance error for $\mathscr{Z}^{[6]}$}
\label{fig2System21_Uniform_FSC_Disp_Var_7_Error}
\end{subfigure}
\caption{\emph{Problem 2} --- Local error evolution of $\mathbf{E}[u]$ and $\mathrm{Var}[u]$ for different $p$-discretization levels of RFS and for $\mu\sim\mathrm{Uniform}$}
\label{fig2System21_Uniform_FSC_Disp_Error}
\end{figure}

\begin{figure}
\centering
\begin{subfigure}[b]{0.495\textwidth}
\includegraphics[width=\textwidth]{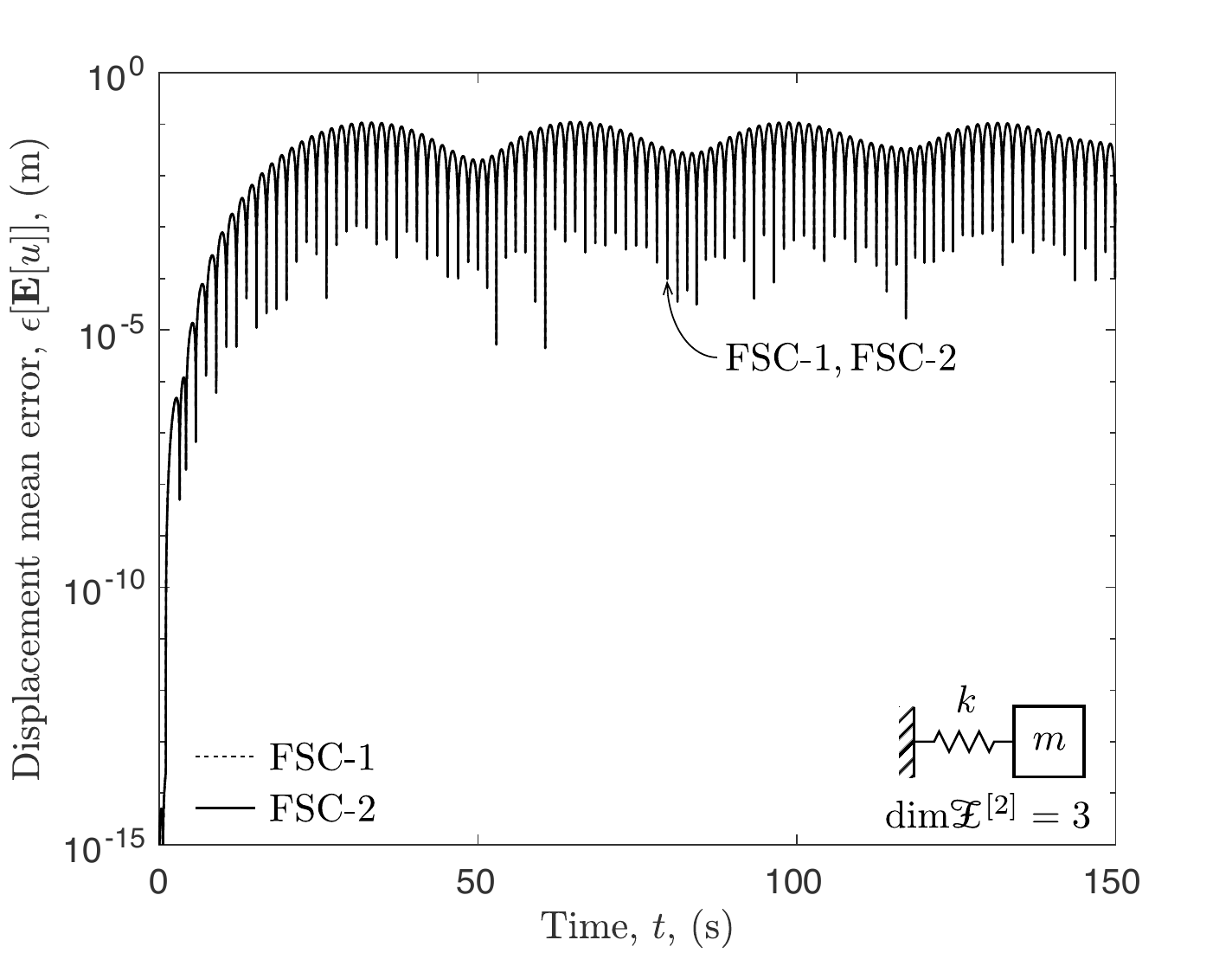}
\caption{Mean error for $\mathscr{Z}^{[2]}$}
\label{fig2System22_UniformUniform_FSC_Disp_Mean_3_Error}
\end{subfigure}\hfill
\begin{subfigure}[b]{0.495\textwidth}
\includegraphics[width=\textwidth]{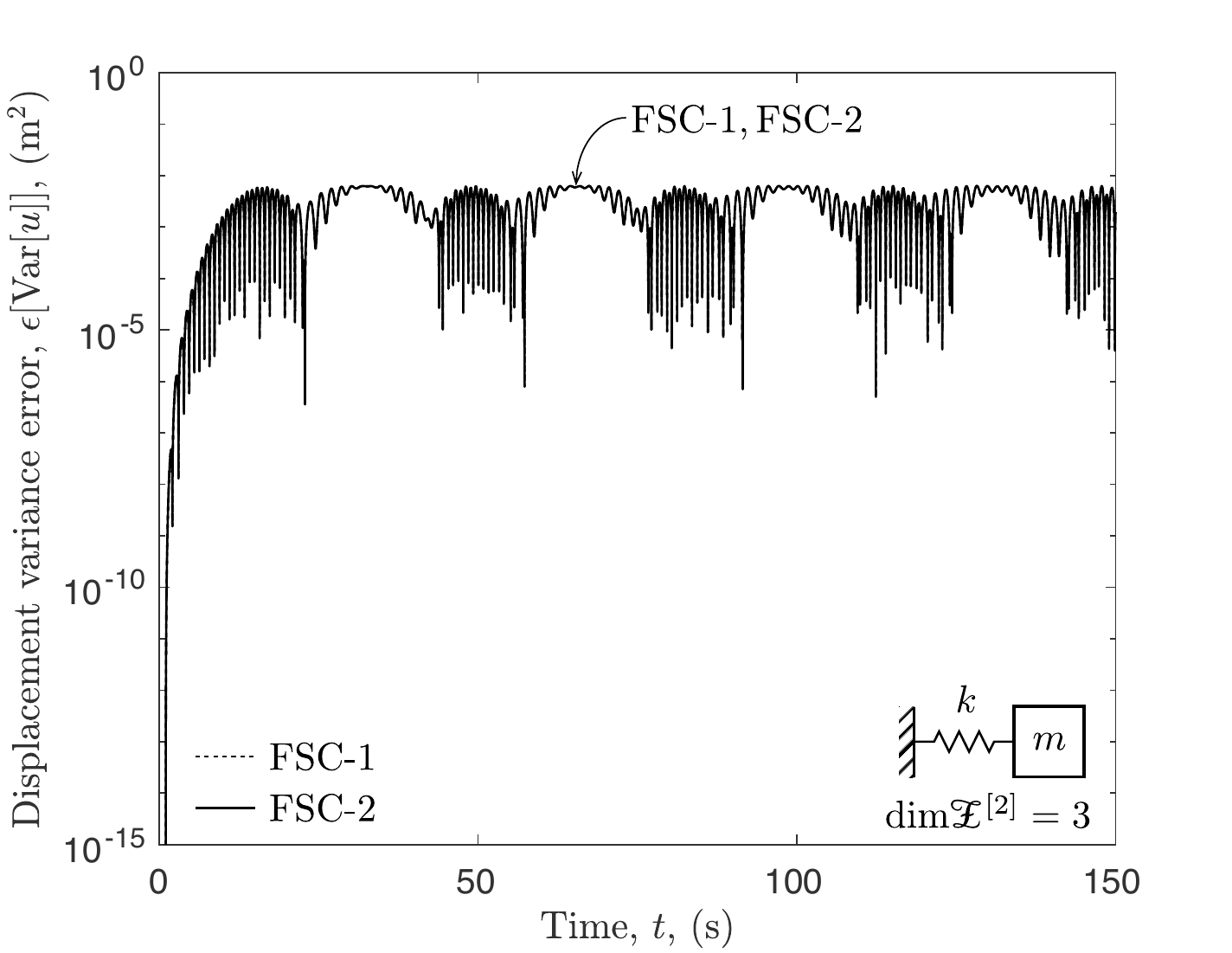}
\caption{Variance error for $\mathscr{Z}^{[2]}$}
\label{fig2System22_UniformUniform_FSC_Disp_Var_3_Error}
\end{subfigure}\quad
\begin{subfigure}[b]{0.495\textwidth}
\includegraphics[width=\textwidth]{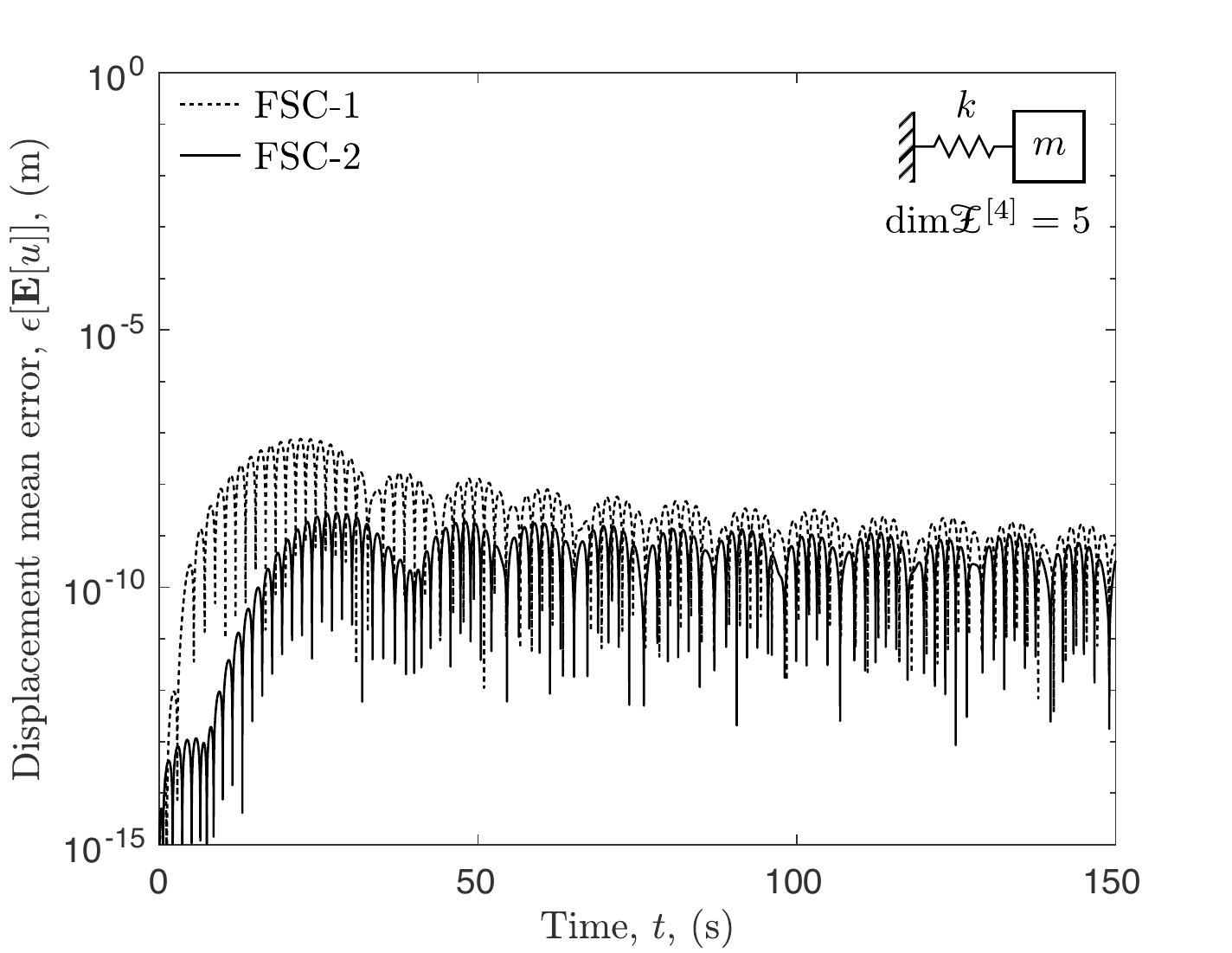}
\caption{Mean error for $\mathscr{Z}^{[4]}$}
\label{fig2System22_UniformUniform_FSC_Disp_Mean_5_Error}
\end{subfigure}\hfill
\begin{subfigure}[b]{0.495\textwidth}
\includegraphics[width=\textwidth]{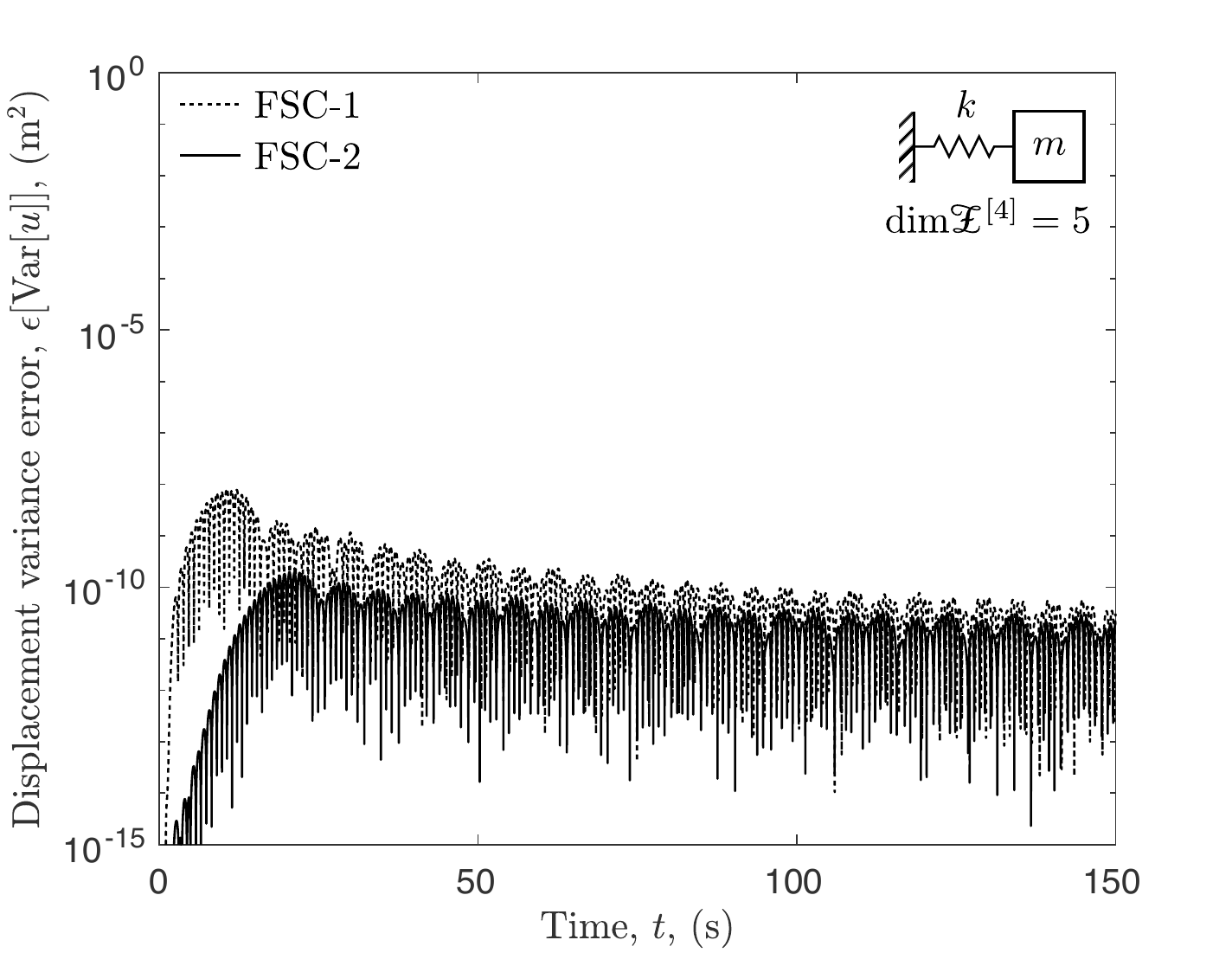}
\caption{Variance error for $\mathscr{Z}^{[4]}$}
\label{fig2System22_UniformUniform_FSC_Disp_Var_5_Error}
\end{subfigure}\quad
\begin{subfigure}[b]{0.495\textwidth}
\includegraphics[width=\textwidth]{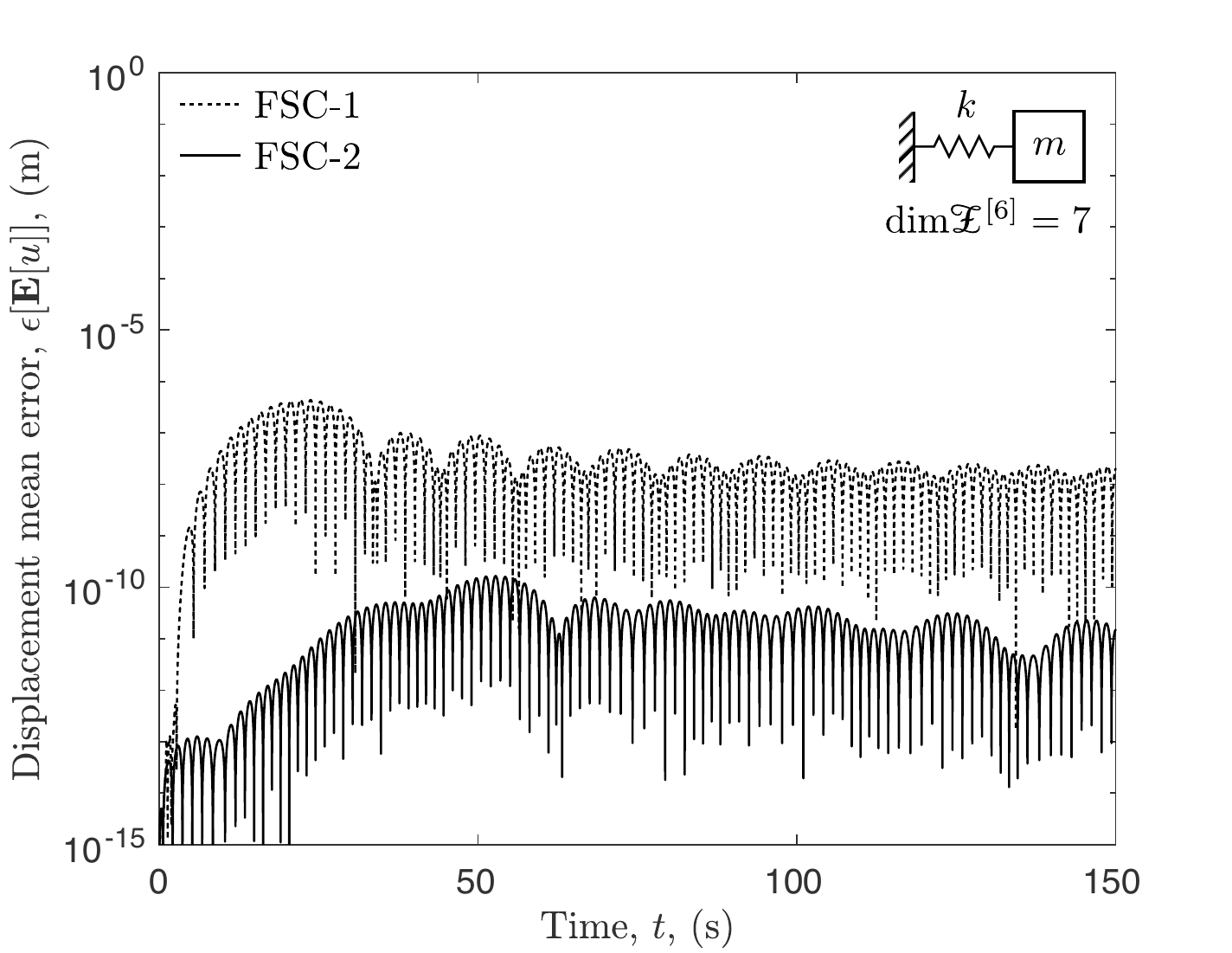}
\caption{Mean error for $\mathscr{Z}^{[6]}$}
\label{fig2System22_UniformUniform_FSC_Disp_Mean_7_Error}
\end{subfigure}\hfill
\begin{subfigure}[b]{0.495\textwidth}
\includegraphics[width=\textwidth]{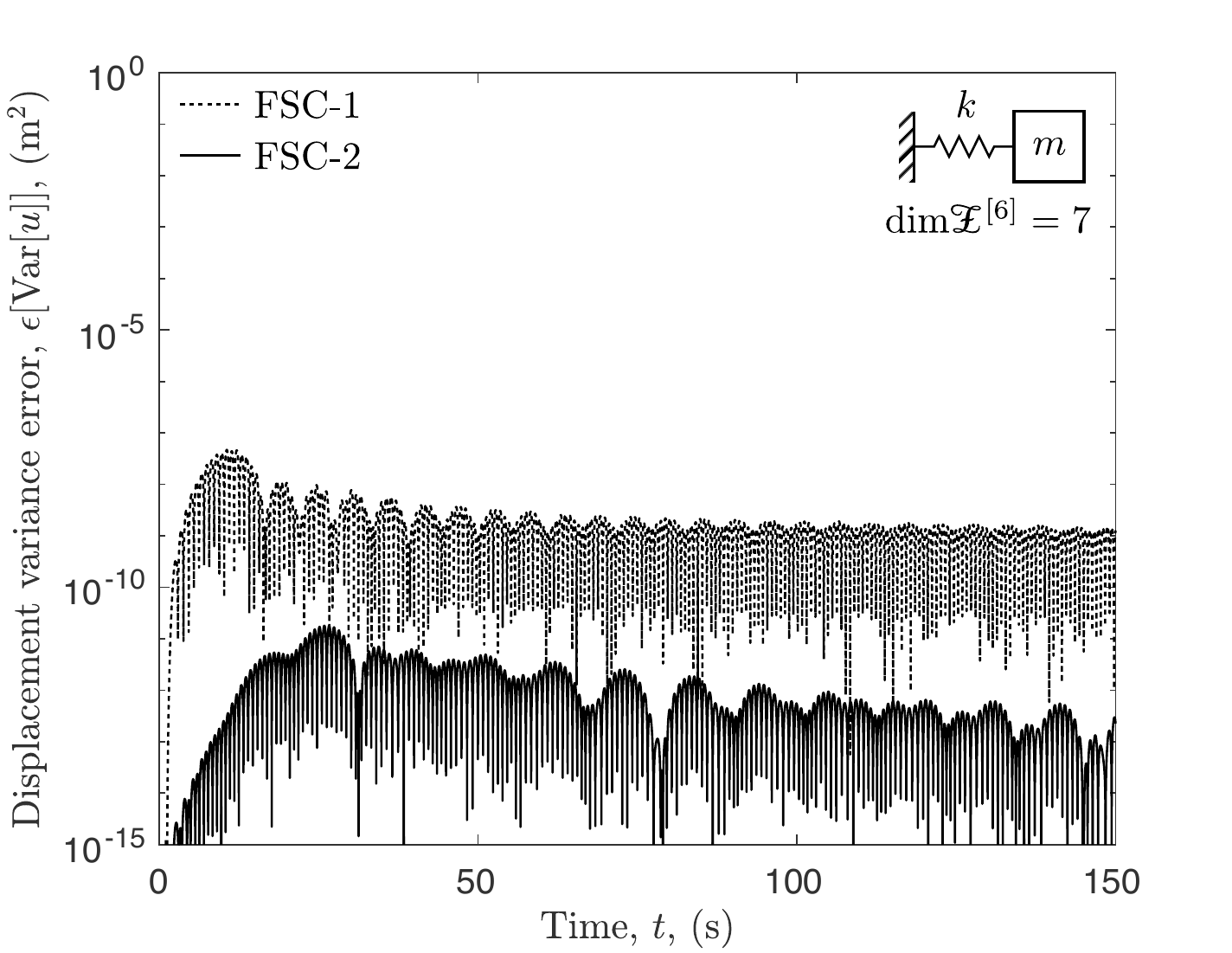}
\caption{Variance error for $\mathscr{Z}^{[6]}$}
\label{fig2System22_UniformUniform_FSC_Disp_Var_7_Error}
\end{subfigure}
\caption{\emph{Problem 3} --- Local error evolution of $\mathbf{E}[u]$ and $\mathrm{Var}[u]$ for different $p$-discretization levels of RFS and for $\mu\sim\mathrm{Uniform}\otimes\mathrm{Uniform}$}
\label{fig2System22_UniformUniform_FSC_Disp_Error}
\end{figure}

Figs.~\ref{fig2System1_FSC_Vel_GlobalError} to \ref{fig2System4_FSC_Jerk_GlobalError} present the convergence of global errors as a function of the number of basis vectors.
Included in these figures are the cases when the random parameters take different probability distributions as specified in Sections \ref{sec2NumExa10} to \ref{sec2NumExa50}.
In general, exponential convergence to the solution is achieved when $n+1$, $n+2$ and $n+3$ basis vectors are used.
However, in the case of using $n+4$ basis vectors, the accuracy of the results does not improve noticeably for FSC-1 but it does for FSC-2.
In fact, a significant difference between the two approaches can be discerned after using $n+4$ basis vectors.
Moreover, as it can be deduced from the results, that there is no gain in employing $n+5$ basis vectors in the simulations because it does not lead to an increase in the accuracy of the solution.
It is interesting to point out that in all figures, the results from FSC-1 and FSC-2 happen to be indistinguishable from each other whenever $n+1$ or $n+2$ basis vectors are used in the simulations.
This is in contrast to Fig.~\ref{fig2System1_FSC_Vel_GlobalError} which shows that exponential convergence to the solution cannot be attained if the number of basis vectors is increased from $n+2$ to $n+3$ (i.e.~from 3 to 4).
This peculiar result can be explained by noting that the system's response is non-oscillatory, which makes the track of the RFS deviate continuously as the simulation proceeds.
Figs.~\ref{fig2System21_FSC_Disp_GlobalError} and \ref{fig2System22_FSC_Disp_GlobalError} further show that when the probability space is one-dimensional (Problem 2), better results are obtained than when it is two-dimensional (Problem 3).
This is because more quadrature points are needed in Problem 3 to achieve the same level of accuracy.
We also note from Fig.~\ref{fig2System22_FSC_Disp_GlobalError} that albeit the probability space is two-dimensional, a few numbers of basis vectors are needed to obtain accurate results.
Remarkably, using only 5 or 6 basis vectors for FSC-2 already produces a global error of approximately $10^{-10}$.
Finally, Fig.~\ref{fig2System4_FSC_Jerk_GlobalError} indicates that when the response is unbounded (e.g.~when $\xi$ is assumed normally distributed), FSC-2 has the ability to control better the error propagation of the solution as the simulation proceeds.

\begin{figure}
\centering
\begin{subfigure}[b]{0.495\textwidth}
\includegraphics[width=\textwidth]{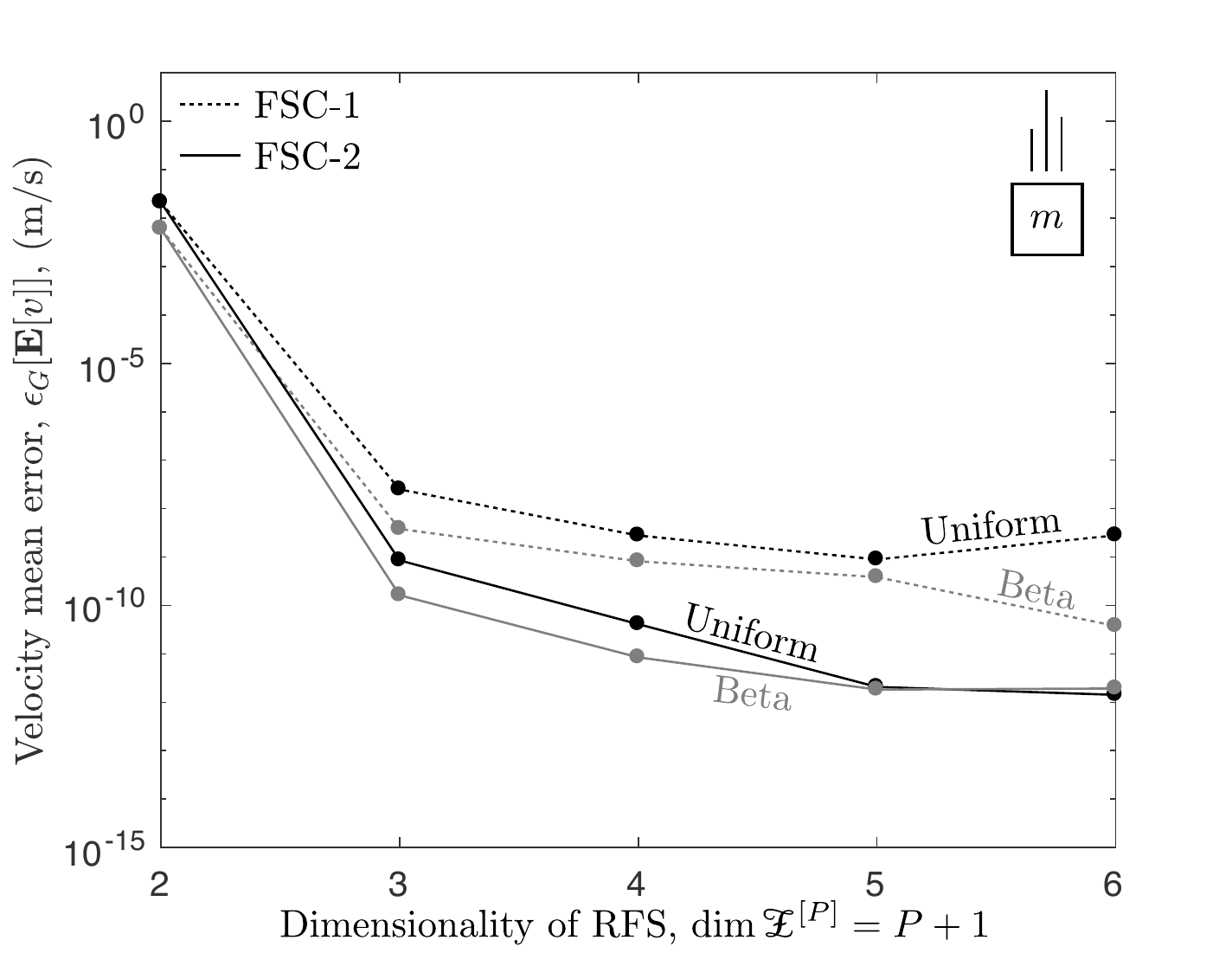}
\caption{Mean error}
\label{fig2System1_FSC_Vel_Mean_GlobalError}
\end{subfigure}\hfill
\begin{subfigure}[b]{0.495\textwidth}
\includegraphics[width=\textwidth]{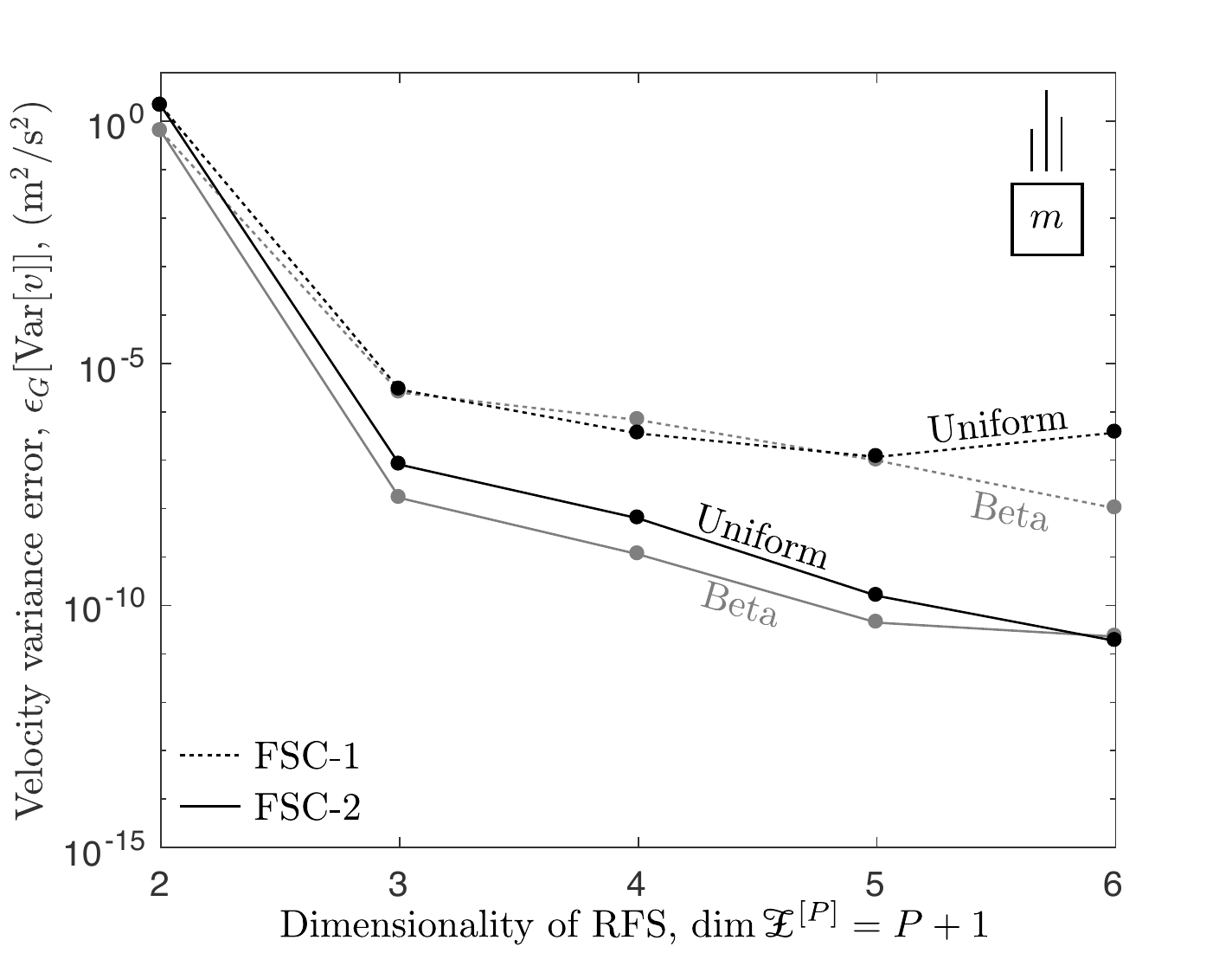}
\caption{Variance error}
\label{fig2System1_FSC_Vel_Var_GlobalError}
\end{subfigure}
\caption{\emph{Problem 1} --- Global error of $\mathbf{E}[v]$ and $\mathrm{Var}[v]$ for different $p$-discretization levels of RFS}
\label{fig2System1_FSC_Vel_GlobalError}
\end{figure}

\begin{figure}
\centering
\begin{subfigure}[b]{0.495\textwidth}
\includegraphics[width=\textwidth]{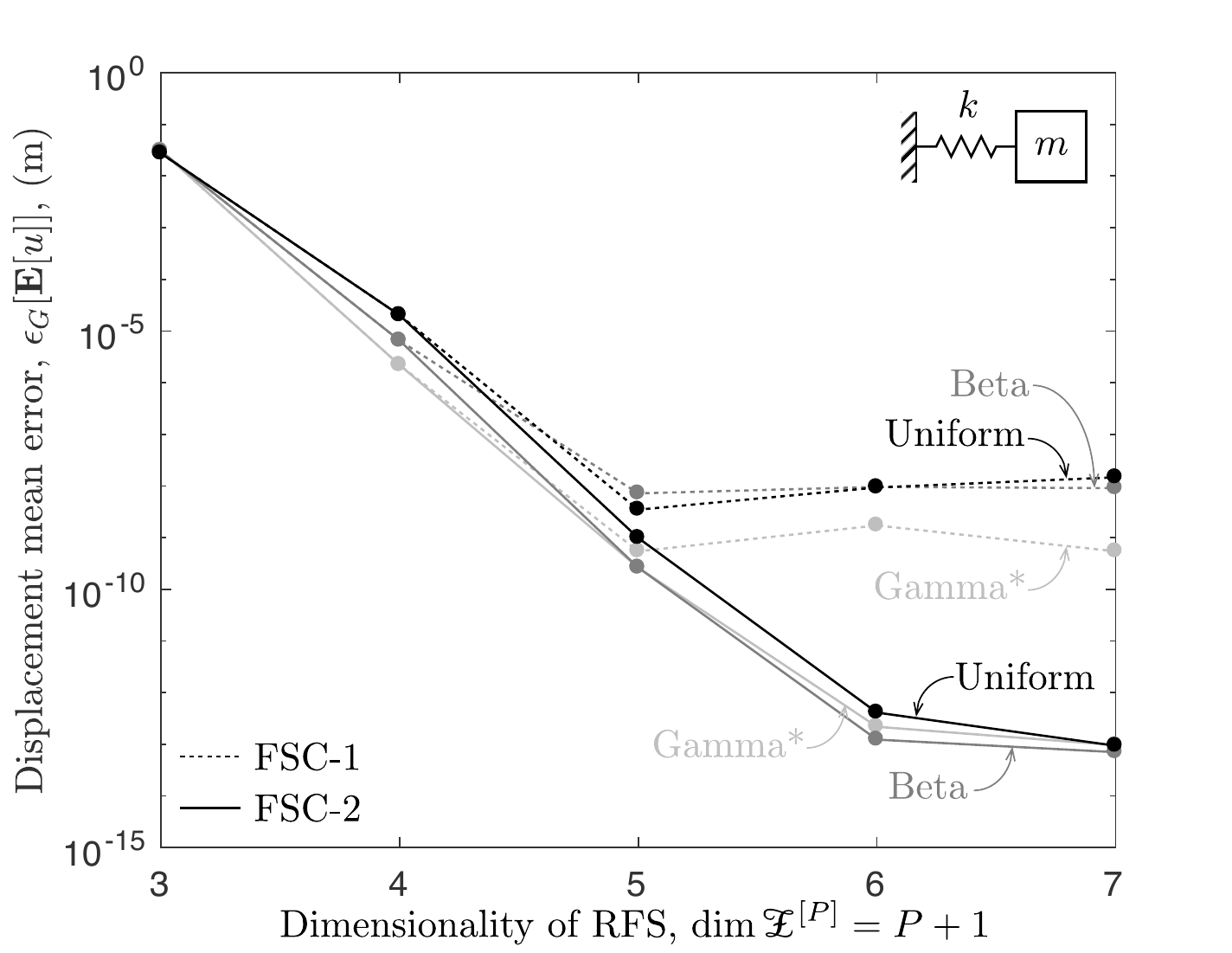}
\caption{Mean error}
\label{fig2System21_FSC_Disp_Mean_GlobalError}
\end{subfigure}\hfill
\begin{subfigure}[b]{0.495\textwidth}
\includegraphics[width=\textwidth]{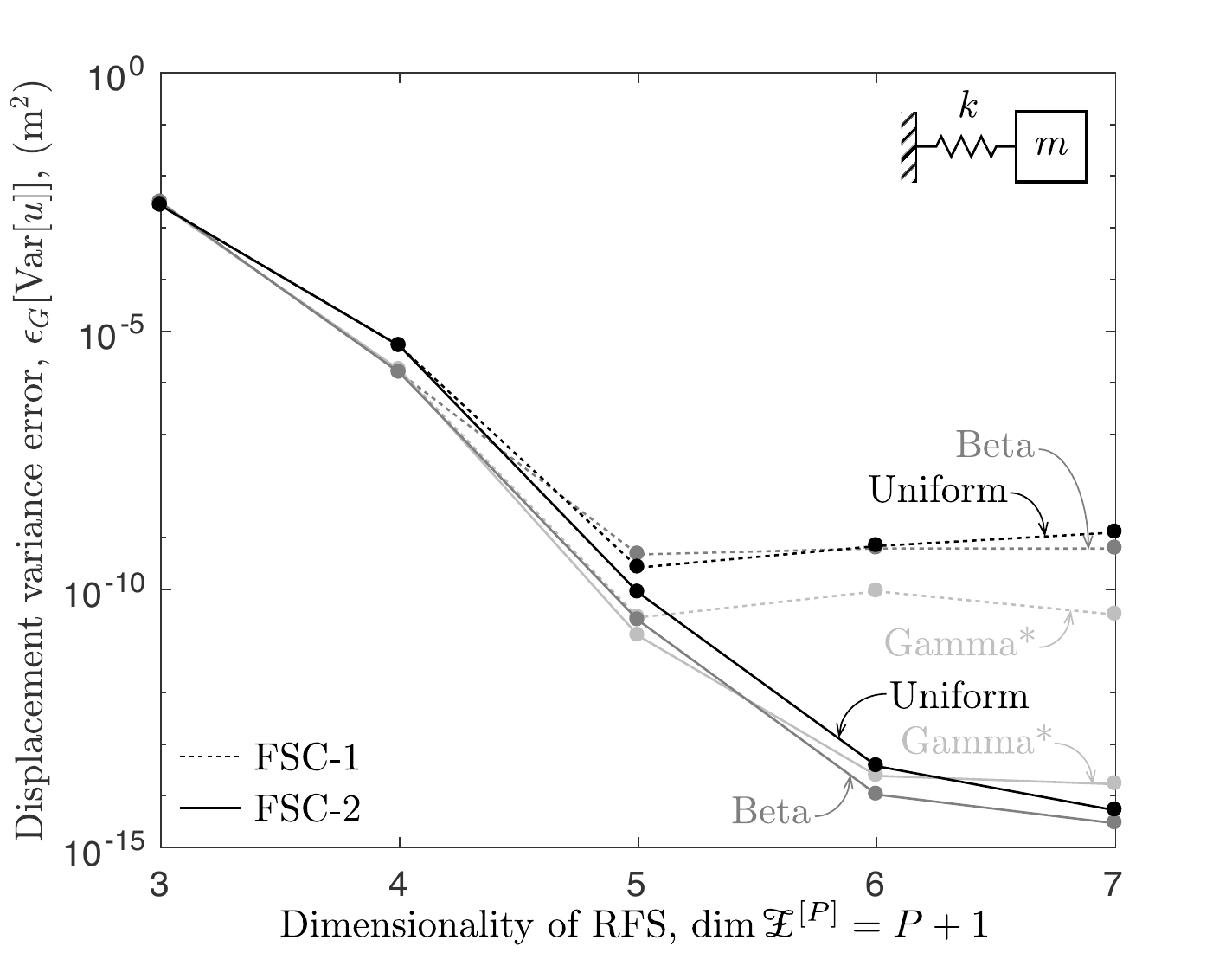}
\caption{Variance error}
\label{fig2System21_FSC_Disp_Var_GlobalError}
\end{subfigure}
\caption{\emph{Problem 2} --- Global error of $\mathbf{E}[u]$ and $\mathrm{Var}[u]$ for different $p$-discretization levels of RFS}
\label{fig2System21_FSC_Disp_GlobalError}
\end{figure}

\begin{figure}
\centering
\begin{subfigure}[b]{0.495\textwidth}
\includegraphics[width=\textwidth]{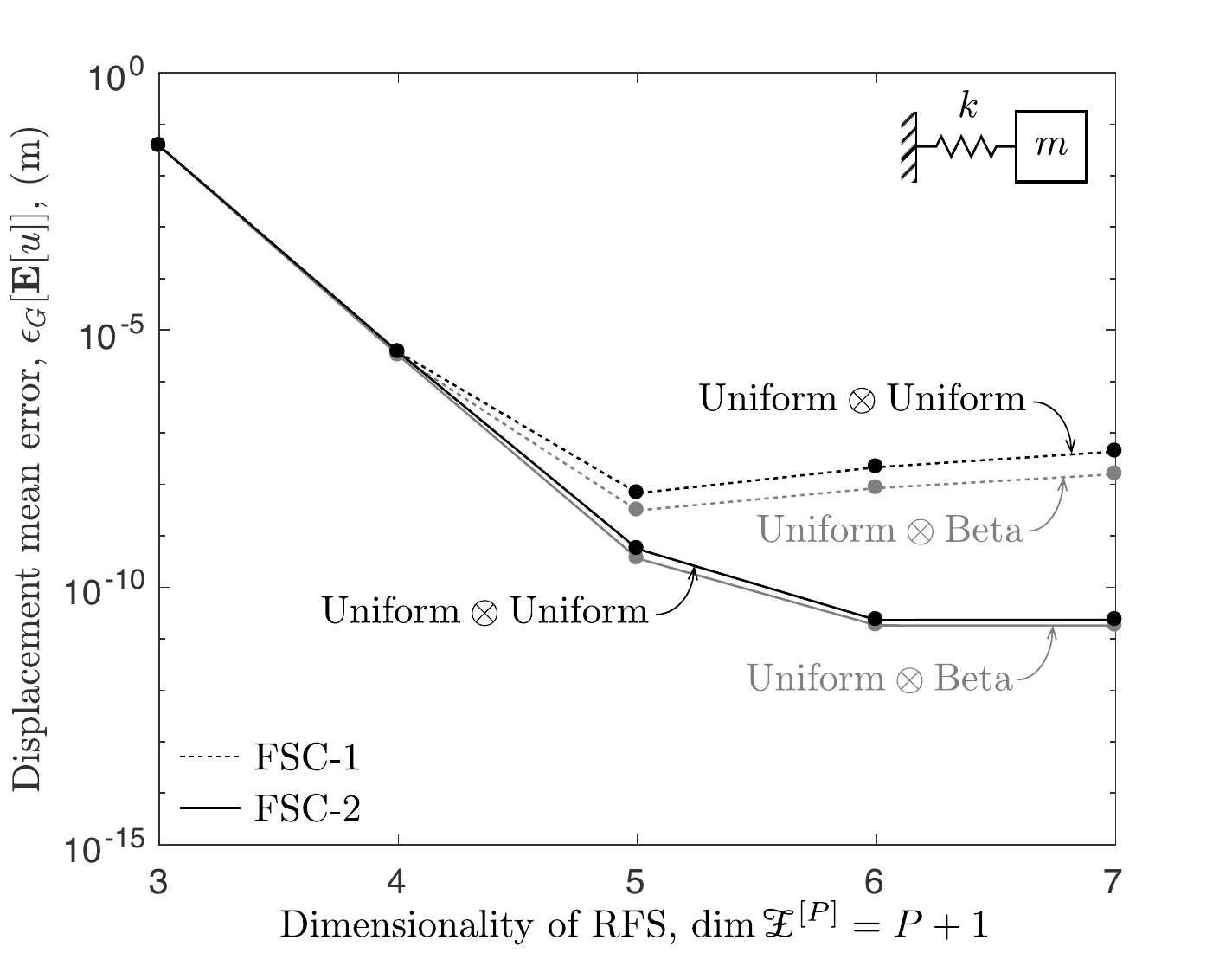}
\caption{Mean error}
\label{fig2System22_FSC_Disp_Mean_GlobalError}
\end{subfigure}\hfill
\begin{subfigure}[b]{0.495\textwidth}
\includegraphics[width=\textwidth]{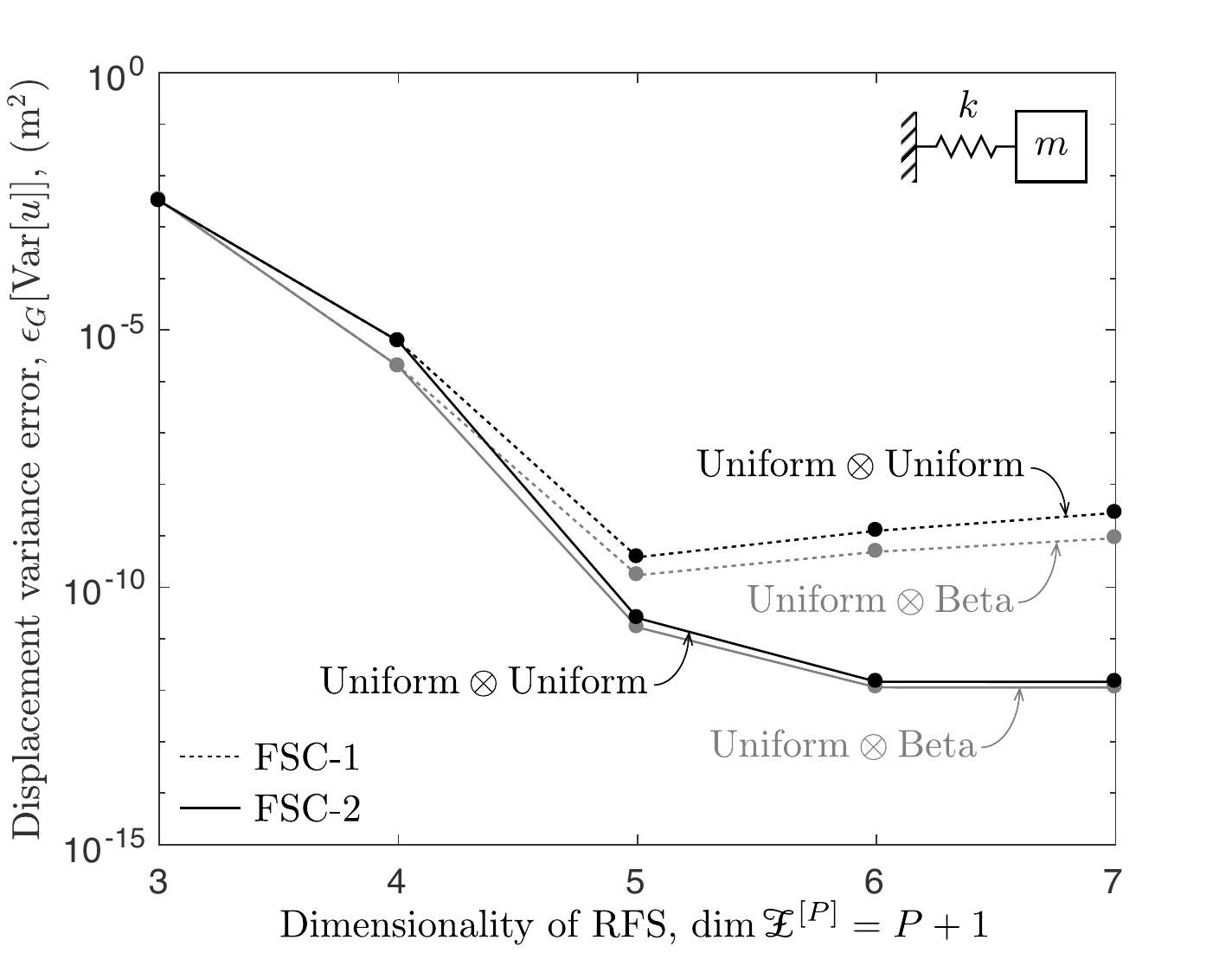}
\caption{Variance error}
\label{fig2System22_FSC_Disp_Var_GlobalError}
\end{subfigure}
\caption{\emph{Problem 3} --- Global error of $\mathbf{E}[u]$ and $\mathrm{Var}[u]$ for different $p$-discretization levels of RFS}
\label{fig2System22_FSC_Disp_GlobalError}
\end{figure}

\begin{figure}
\centering
\begin{subfigure}[b]{0.495\textwidth}
\includegraphics[width=\textwidth]{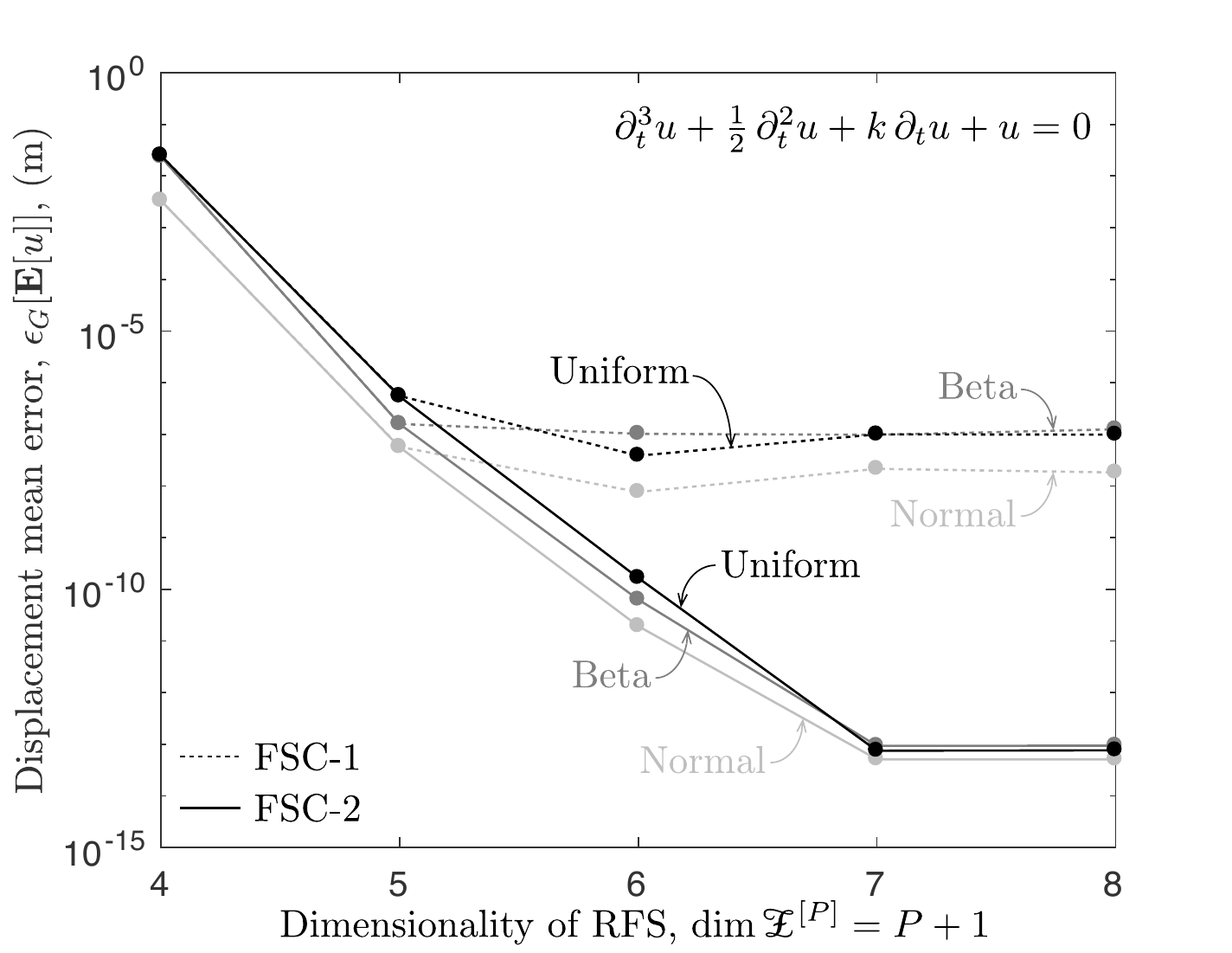}
\caption{Mean error}
\label{fig2System3_FSC_Disp_Mean_GlobalError}
\end{subfigure}\hfill
\begin{subfigure}[b]{0.495\textwidth}
\includegraphics[width=\textwidth]{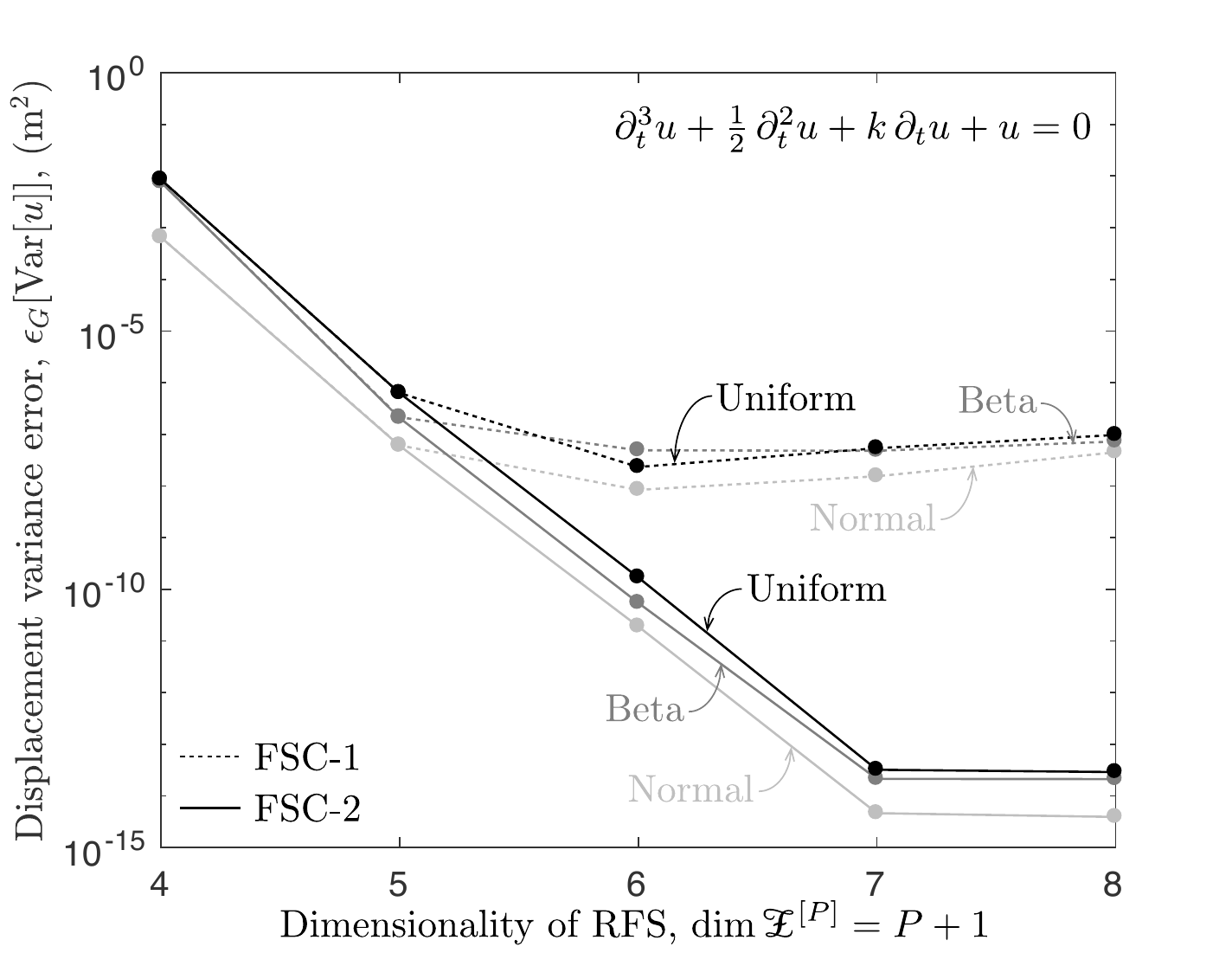}
\caption{Variance error}
\label{fig2System3_FSC_Disp_Var_GlobalError}
\end{subfigure}
\caption{\emph{Problem 4} --- Global error of $\mathbf{E}[u]$ and $\mathrm{Var}[u]$ for different $p$-discretization levels of RFS}
\label{fig2System3_FSC_Disp_GlobalError}
\end{figure}

\begin{figure}
\centering
\begin{subfigure}[b]{0.495\textwidth}
\includegraphics[width=\textwidth]{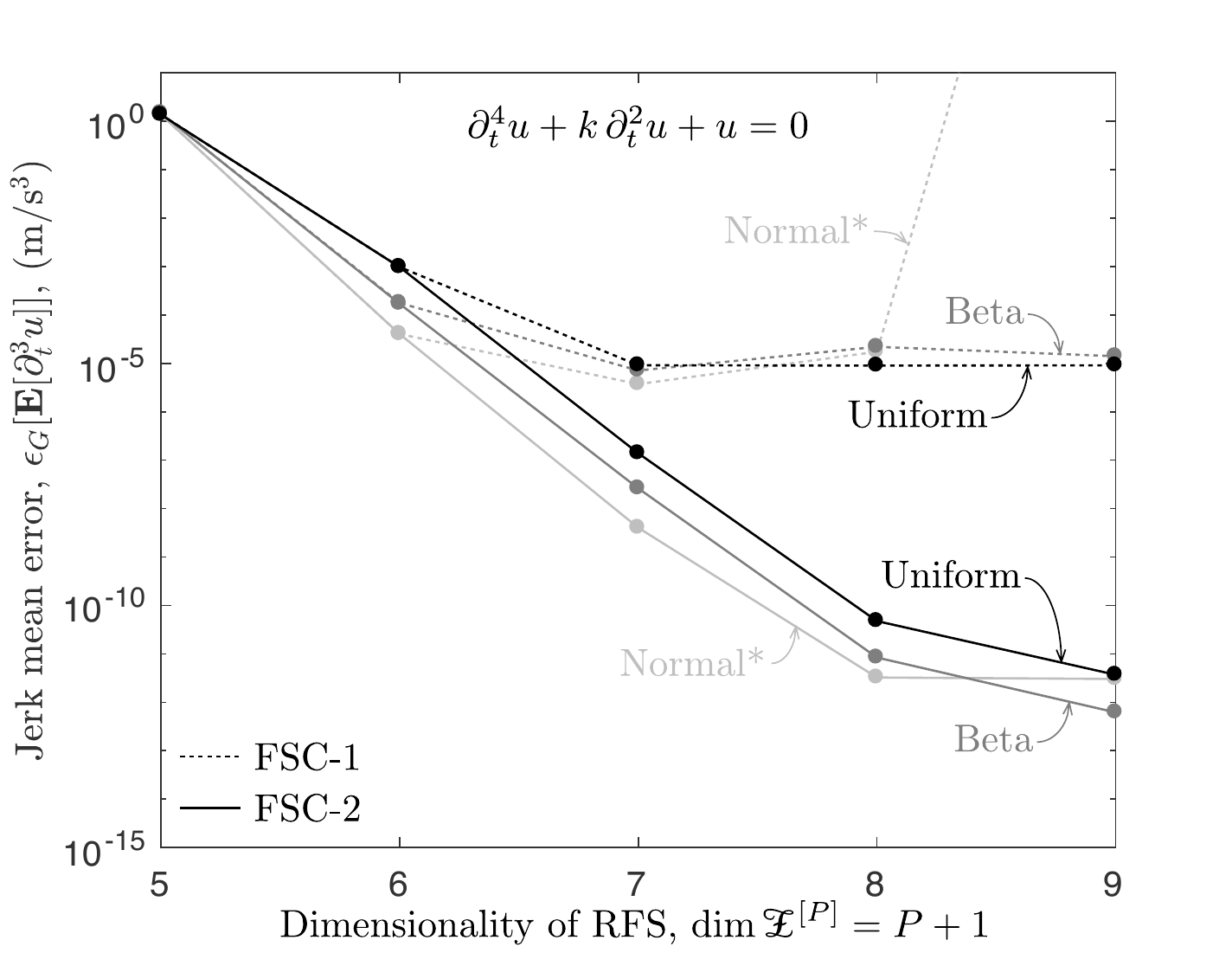}
\caption{Mean error}
\label{fig2System4_FSC_Jerk_Mean_GlobalError}
\end{subfigure}\hfill
\begin{subfigure}[b]{0.495\textwidth}
\includegraphics[width=\textwidth]{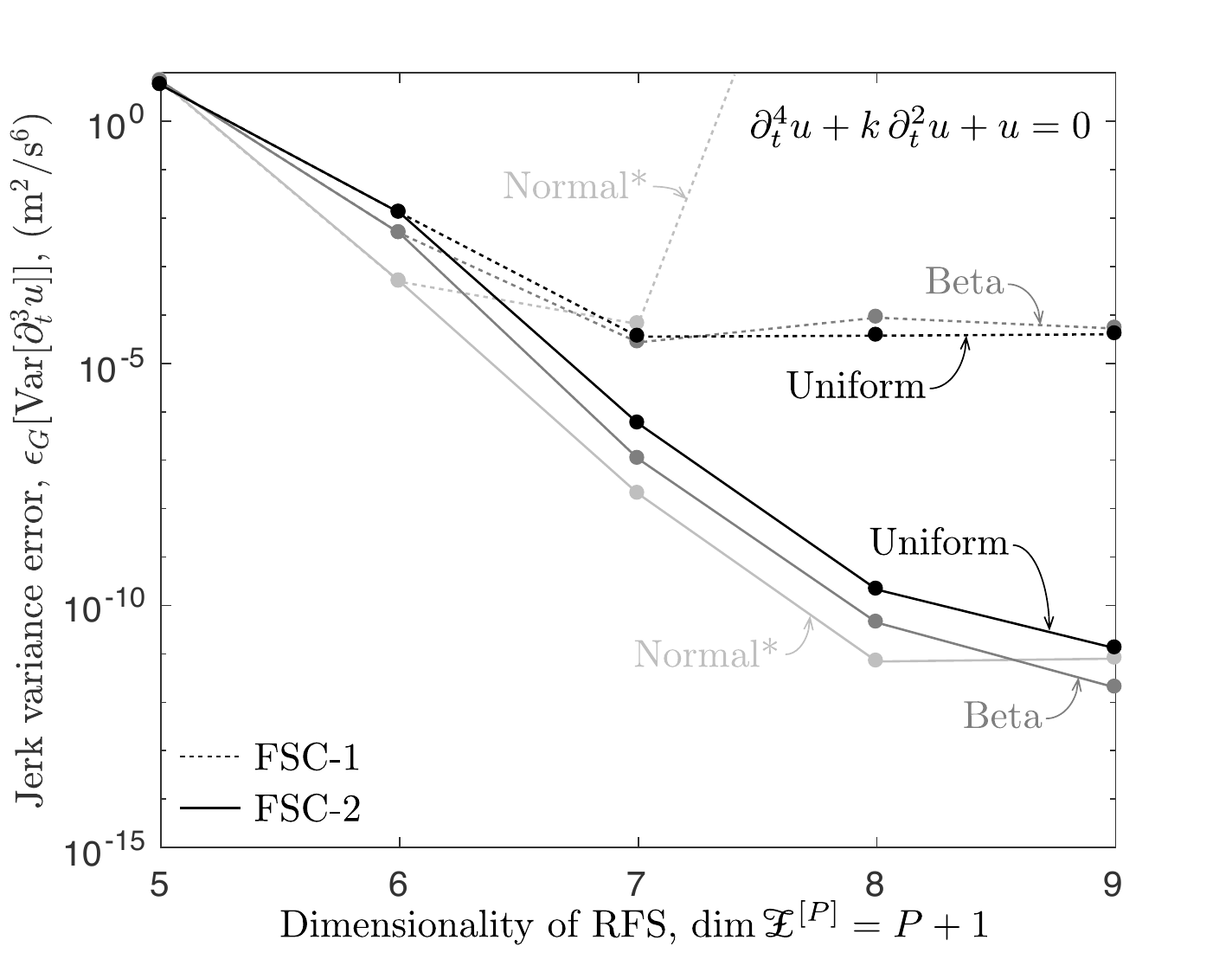}
\caption{Variance error}
\label{fig2System4_FSC_Jerk_Var_GlobalError}
\end{subfigure}
\caption{\emph{Problem 5} --- Global error of $\mathbf{E}[\partial_t^3u]$ and $\mathrm{Var}[\partial_t^3u]$ for different $p$-discretization levels of RFS}
\label{fig2System4_FSC_Jerk_GlobalError}
\end{figure}

Fig.~\ref{fig2System21_Uniform_FSC_Disp_GlobalError} presents the convergence of global errors for Problem 2 as a function of the number of basis vectors and for different time-step sizes.
The goal of this figure is to show the implications of increasing the time-step size used by default ($\Delta t=0.001$ s) in regard to the accuracy of the results.
Though here we only depict the case when $\xi$ is uniformly distributed, similar trends are obtained when other distributions are used.
In particular, this figure shows that the discretization of the temporal function space plays an important role when it comes to the FSC-2 approach but not when it comes to the FSC-1 approach.
The reason for this is that in FSC-1 the errors coming from the RFS discretization are substantially larger than those coming from the discretization of the temporal function space.
This leads to the perception that decreasing the time-step size in FSC-1 does not have a direct effect on improving the accuracy of the results derived from the RFS discretization.
Furthermore, we observe that more accurate results are obtained with FSC-2 if the time-step size used for the simulation is progressively decreased.

\begin{figure}
\centering
\begin{subfigure}[b]{0.495\textwidth}
\includegraphics[width=\textwidth]{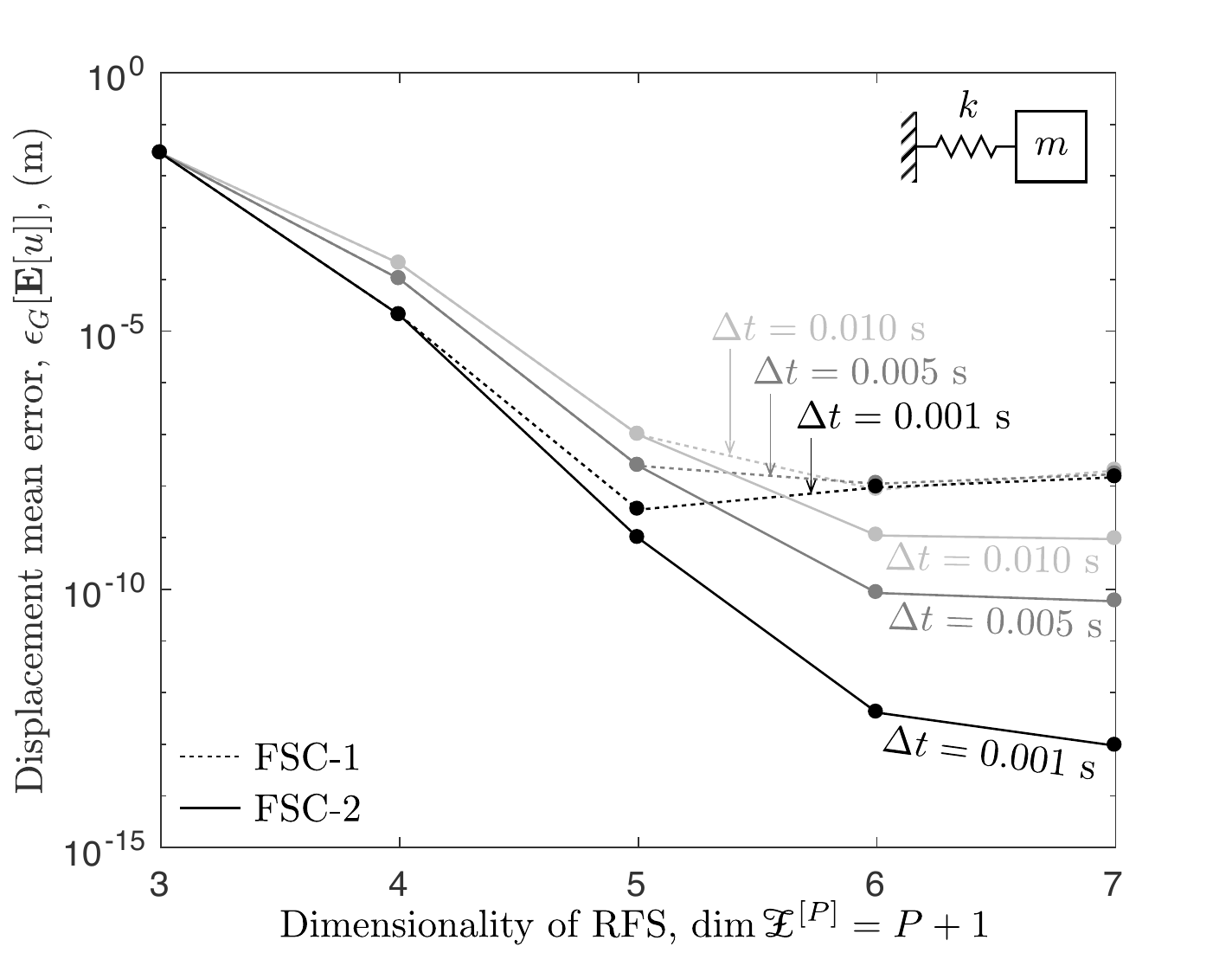}
\caption{Mean error}
\label{fig2System21_Uniform_FSC_Disp_Mean_GlobalError}
\end{subfigure}\hfill
\begin{subfigure}[b]{0.495\textwidth}
\includegraphics[width=\textwidth]{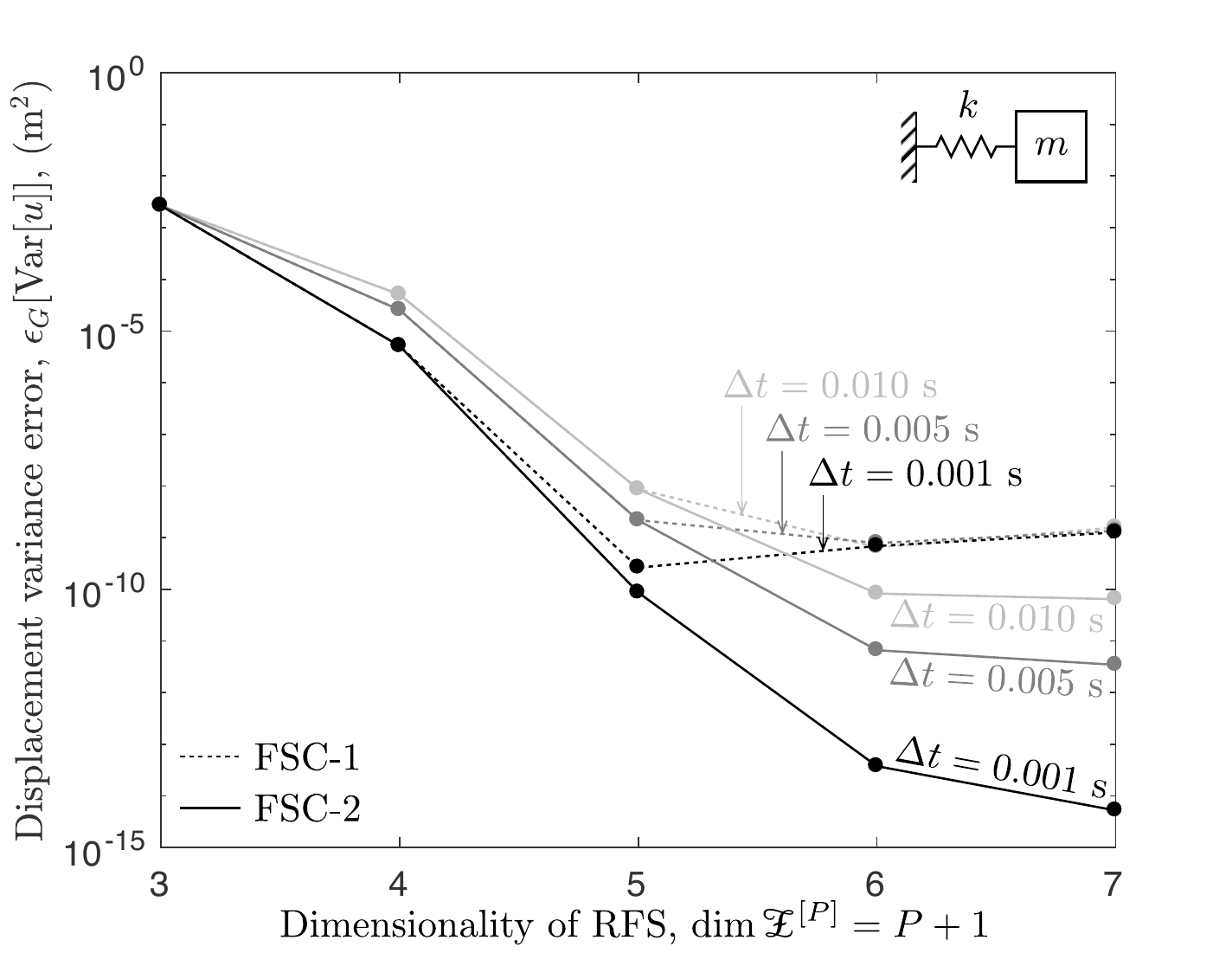}
\caption{Variance error}
\label{fig2System21_Uniform_FSC_Disp_Var_GlobalError}
\end{subfigure}
\caption{\emph{Problem 2} --- Global error of $\mathbf{E}[u]$ and $\mathrm{Var}[u]$ for different $p$-discretization levels of RFS and time-step sizes ($\mu\sim\mathrm{Uniform}$)}
\label{fig2System21_Uniform_FSC_Disp_GlobalError}
\end{figure}

Fig.~\ref{fig2System21_Uniform_FSC_Disp_GlobalError_QuadPoints} depicts the convergence of global errors for Problem 2 as a function of the number of quadrature points utilized to estimate the inner products that appear when using the spectral approach.
We see that when 6 basis vectors are used to perform the simulation, FSC-2 produces a more accurate result than FSC-1 if the number of quadrature points is sufficiently large.
However, no discernible differences between the two FSC approaches are observed if 4 basis vectors are used.
In fact, when 10 quadrature points (or even 20 quadrature points for the case of the variance) are utilized, the global error does not improve as a function of the number of basis vectors used or the FSC approach chosen.
This is because the global error is in this case dominated by the error coming from the Gaussian quadrature rule rather than by the errors coming from the discretization of the random function space and the transfer of the probability information.

\begin{figure}
\centering
\begin{subfigure}[b]{0.495\textwidth}
\includegraphics[width=\textwidth]{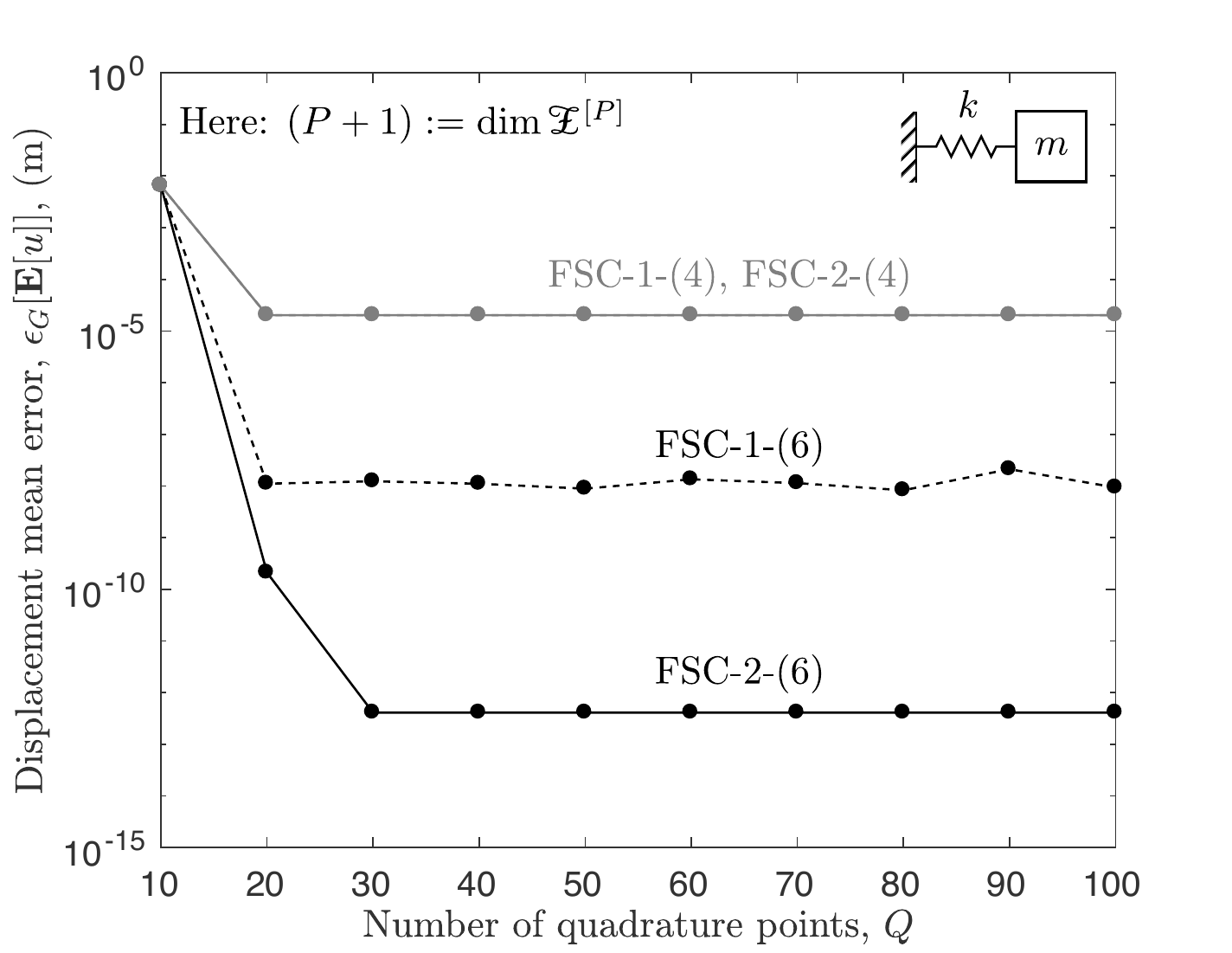}
\caption{Mean error}
\label{fig2System21_Uniform_FSC_Disp_Mean_GlobalError_QuadPoints}
\end{subfigure}\hfill
\begin{subfigure}[b]{0.495\textwidth}
\includegraphics[width=\textwidth]{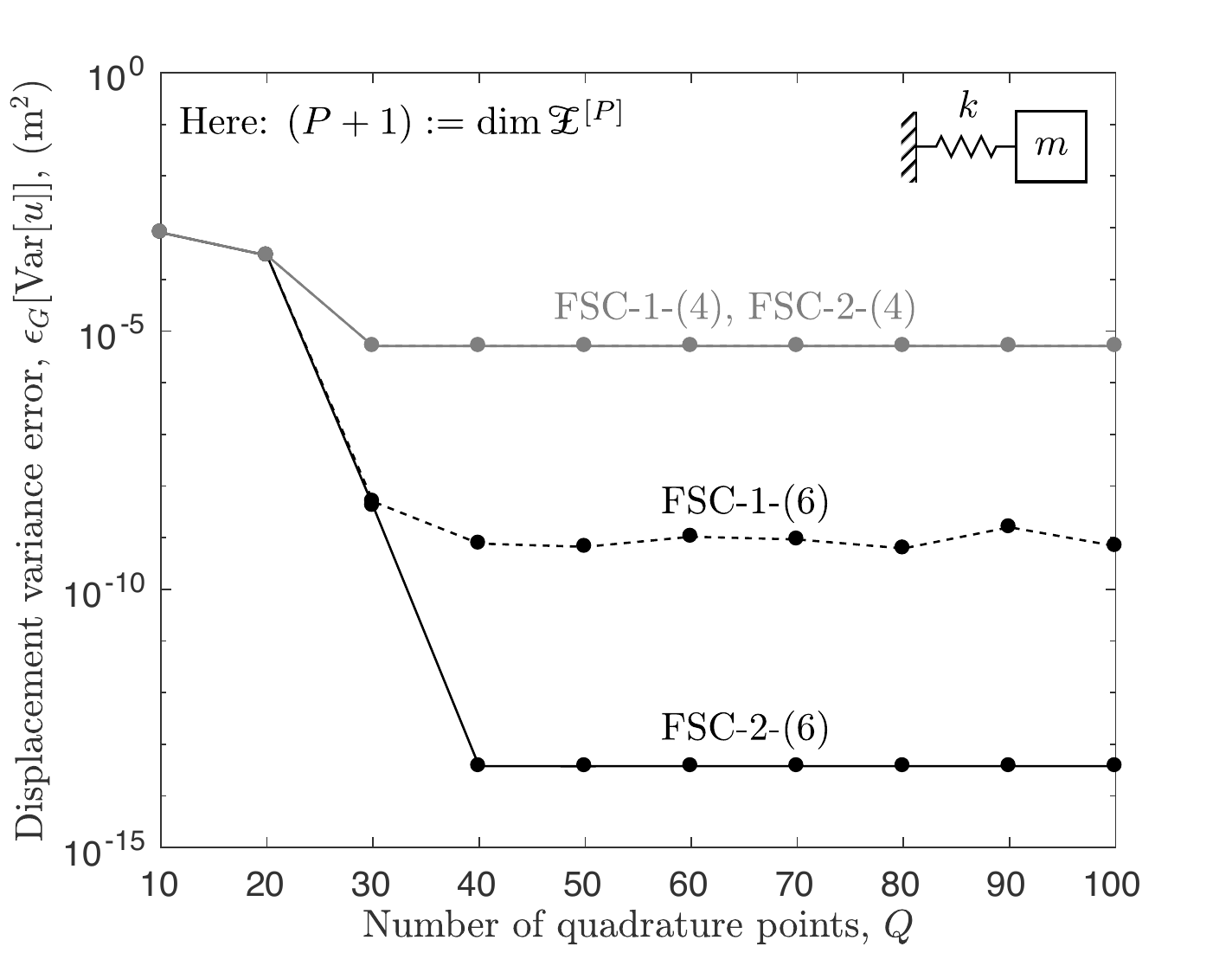}
\caption{Variance error}
\label{fig2System21_Uniform_FSC_Disp_Var_GlobalError_QuadPoints}
\end{subfigure}
\caption{\emph{Problem 2} --- Global error of $\mathbf{E}[u]$ and $\mathrm{Var}[u]$ as a function of the number of quadrature points employed to estimate the inner products ($\mu\sim\mathrm{Uniform}$)}
\label{fig2System21_Uniform_FSC_Disp_GlobalError_QuadPoints}
\end{figure}

Finally, Fig.~\ref{fig2System21_Uniform_ComputationalCost_Disp} plots the global errors as a function of the computational cost associated with FSC-1, FSC-2 and mTD-gPC.
This cost is expressed here in terms of the wall-clock time taken to complete the simulation.
Labels $P2Q0$, $P2Q1$ and $P2Q2$ are defined in Ref.~\cite{heuveline2014hybrid} (Pg.~45), and they correspond to the cases when 6, 12 and 18 basis vectors are employed in the simulations, respectively.
Some conclusions can be deduced from this figure.
First, with respect to achieving a similar level of error, both FSC approaches run much faster than mTD-gPC.
In particular, if a global error of approximately $10^{-8}$ is desired for the simulation, FSC runs about 6 times faster than mTD-gPC.
This outcome can be explained by noticing that mTD-gPC requires more than twice the number of basis vectors than FSC.
Second, for the same number of basis vectors, the FSC method is able to produce results that are at least 6 orders of magnitude more accurate than mTD-gPC.
For example, if the FSC-2 approach is employed with 6 basis vectors, the results are 11 orders of magnitude more accurate than the mTD-gPC counterpart.
Therefore, the FSC methods are not only superior in terms of computational efficiency than mTD-gPC, but they also have the ability to encode the probability information a lot better as the simulation proceeds.
However, it is important to note that by increasing the number of basis vectors from 5 to 7 for FSC-1 and 12 to 18 for mTD-gPC, the results worsen noticeably by an order of magnitude or so.
This is primarily due to the limited precision of the machine and the fact that the probability information is being transferred in the mean-square sense.
Moreover, Fig.~\ref{fig2System21_Uniform_ComputationalCost_Disp} also reveals that in general FSC-2 runs slightly faster than FSC-1.
In fact, the more basis vectors we use, the higher this difference in speed is.
Finally, we also observe that when 7 basis vectors are used, FSC-2 is 5 orders of magnitude more accurate than FSC-1.

\begin{figure}
\centering
\begin{subfigure}[b]{0.495\textwidth}
\includegraphics[width=\textwidth]{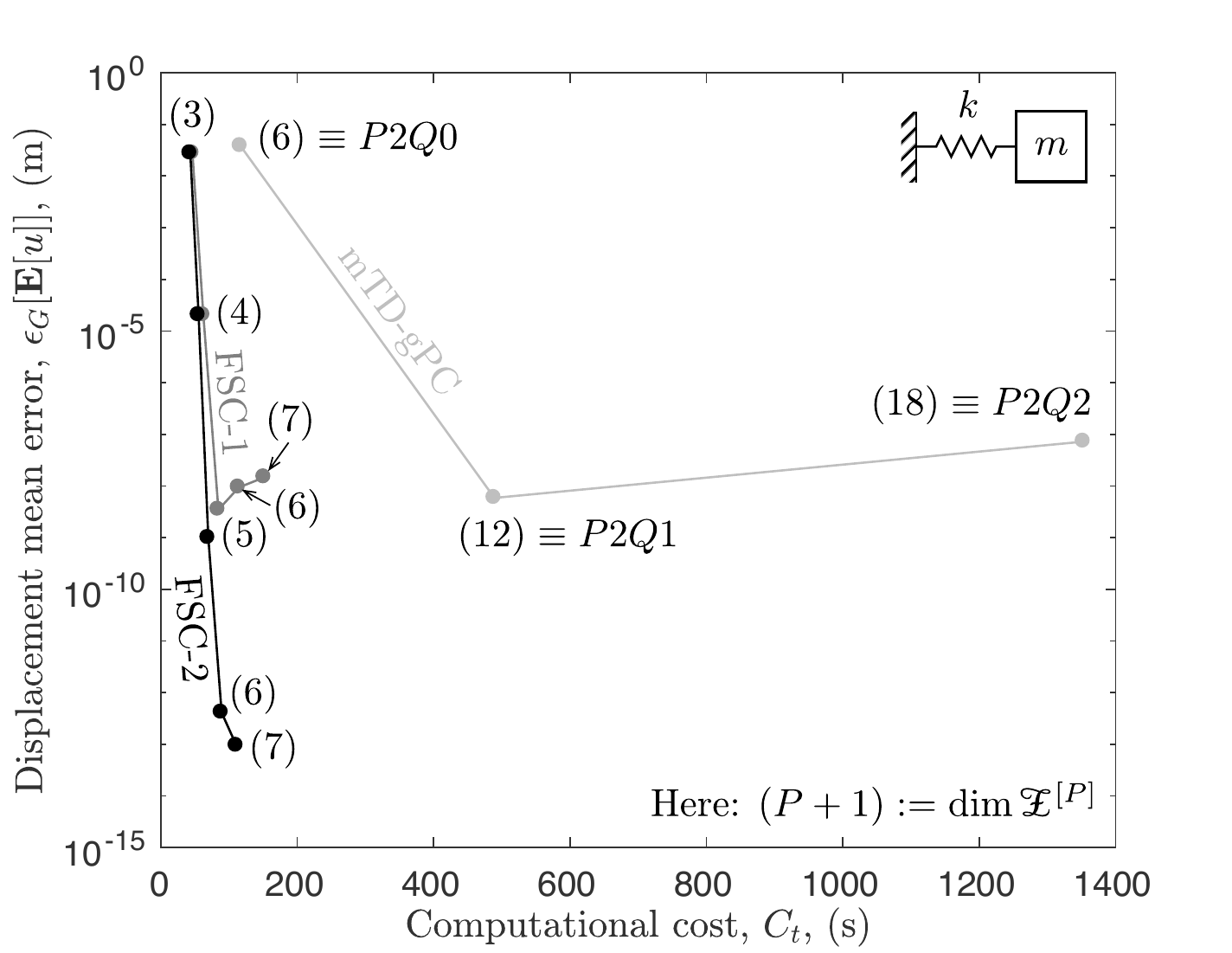}
\caption{Mean error}
\label{fig2System21_Uniform_ComputationalCost_Disp_Mean}
\end{subfigure}\hfill
\begin{subfigure}[b]{0.495\textwidth}
\includegraphics[width=\textwidth]{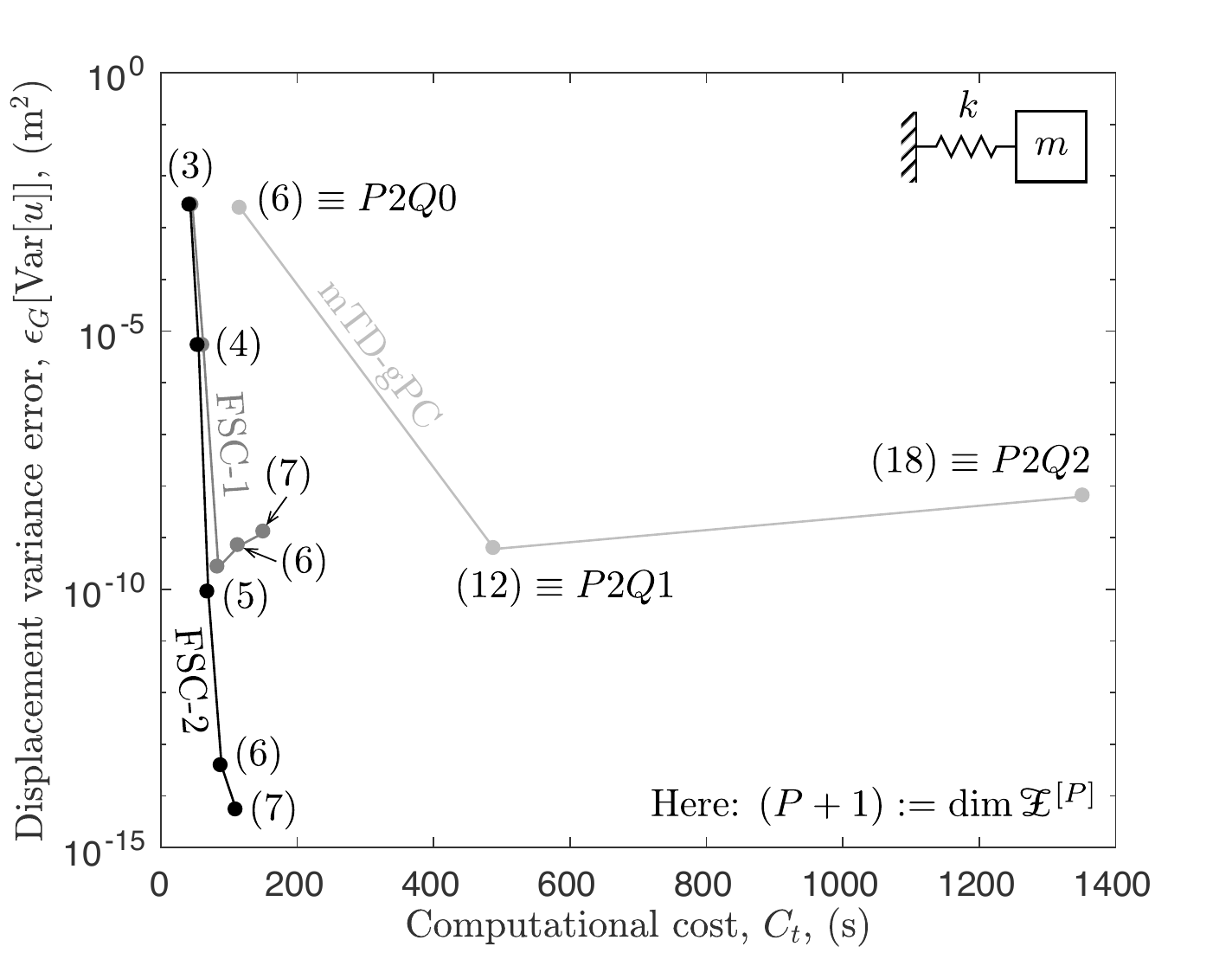}
\caption{Variance error}
\label{fig2System21_Uniform_ComputationalCost_Disp_Var}
\end{subfigure}
\caption{\emph{Problem 2} --- Global error of $\mathbf{E}[u]$ and $\mathrm{Var}[u]$ versus computational cost for different $p$-discretization levels of RFS and for $\mu\sim\mathrm{Uniform}$}
\label{fig2System21_Uniform_ComputationalCost_Disp}
\end{figure}

\subsection{Numerical results for the nonlinear system}

\begin{figure}
\centering
\begin{subfigure}[b]{0.495\textwidth}
\includegraphics[width=\textwidth]{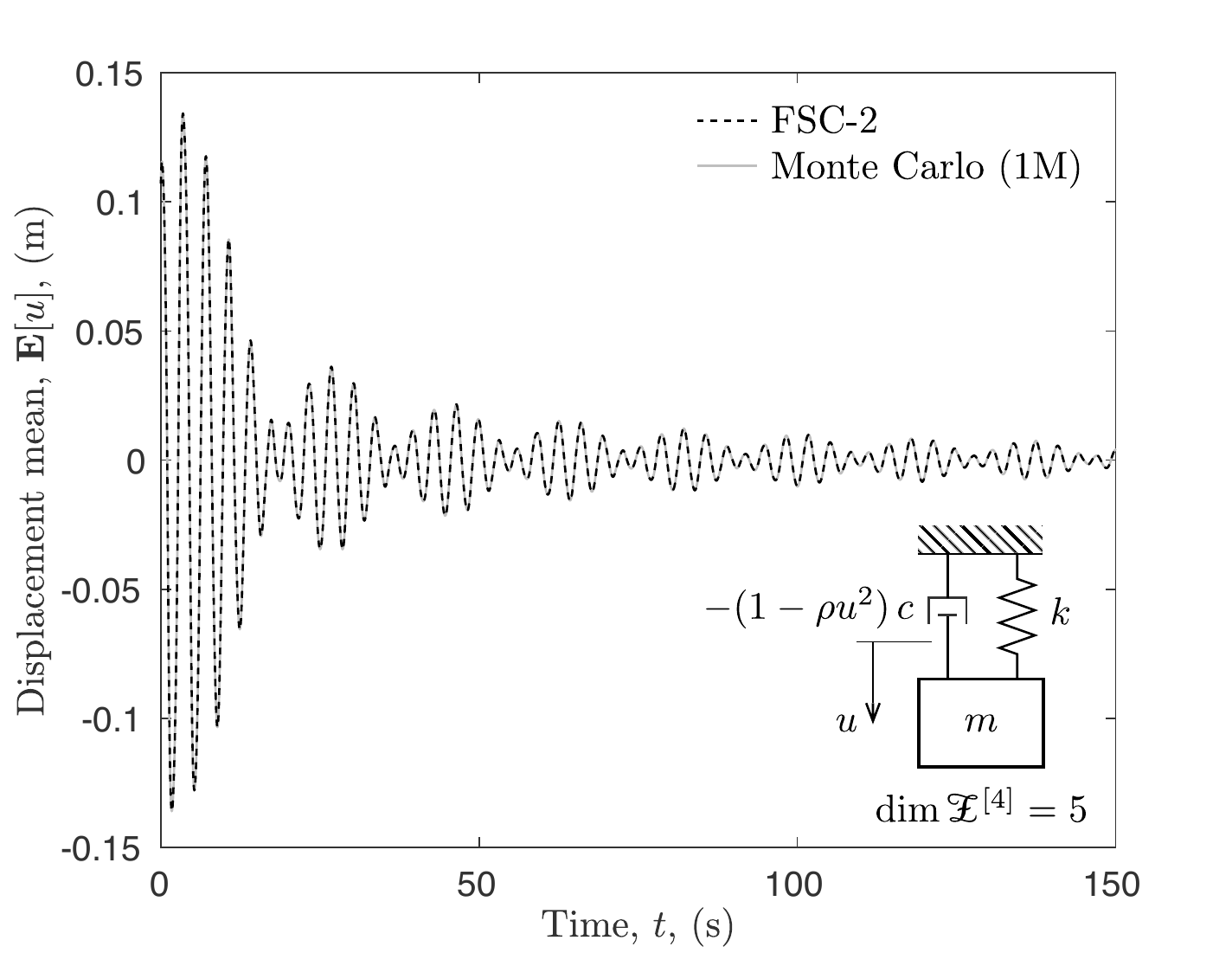}
\caption{Mean}
\label{fig2System5_UniformBeta_FSC_Disp_Mean_5}
\end{subfigure}\hfill
\begin{subfigure}[b]{0.495\textwidth}
\includegraphics[width=\textwidth]{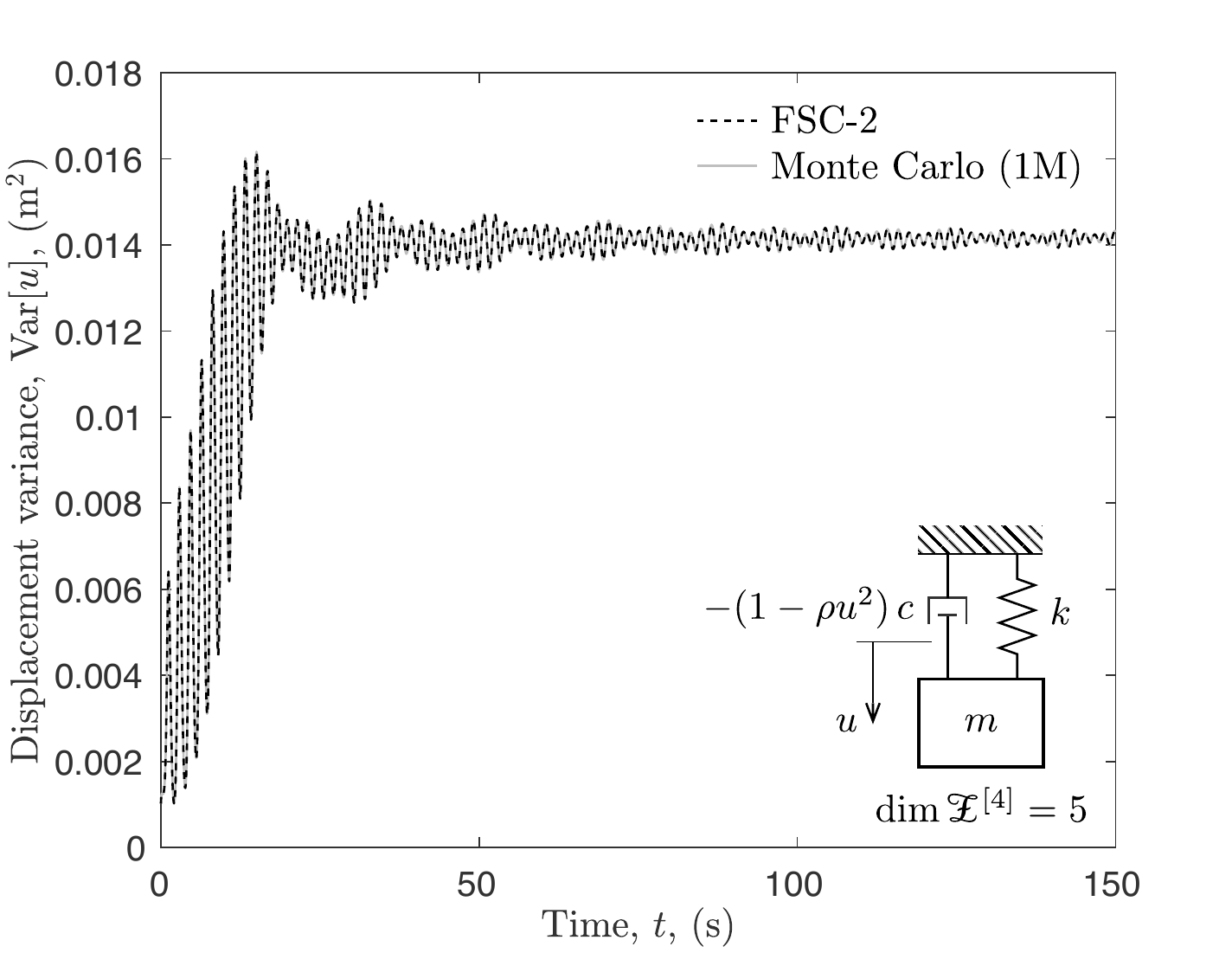}
\caption{Variance}
\label{fig2System5_UniformBeta_FSC_Disp_Var_5}
\end{subfigure}
\caption{\emph{Problem 6 (the Van-der-Pol oscillator)} --- Evolution of $\mathbf{E}[u]$ and $\mathrm{Var}[u]$ for the case when the $p$-discretization level of RFS is $\mathscr{Z}^{[4]}$ and $\mu\sim\mathrm{Uniform}\otimes\mathrm{Beta}$}
\label{fig2System5_UniformBeta_FSC_Disp_5}
\end{figure}

\begin{figure}
\centering
\begin{subfigure}[b]{0.495\textwidth}
\includegraphics[width=\textwidth]{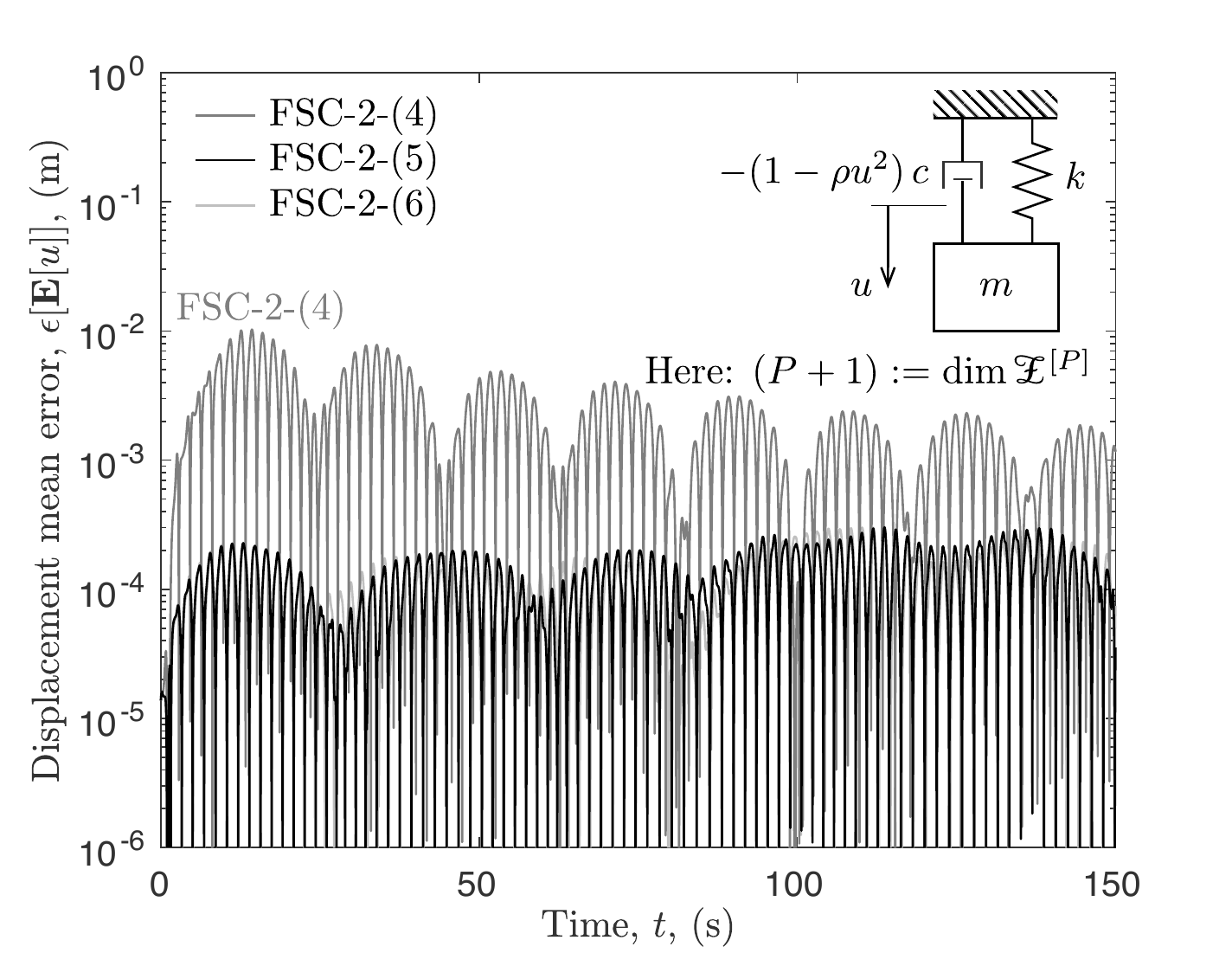}
\caption{Mean error}
\label{fig2System5_UniformBeta_FSC_Disp_Mean_Error}
\end{subfigure}\hfill
\begin{subfigure}[b]{0.495\textwidth}
\includegraphics[width=\textwidth]{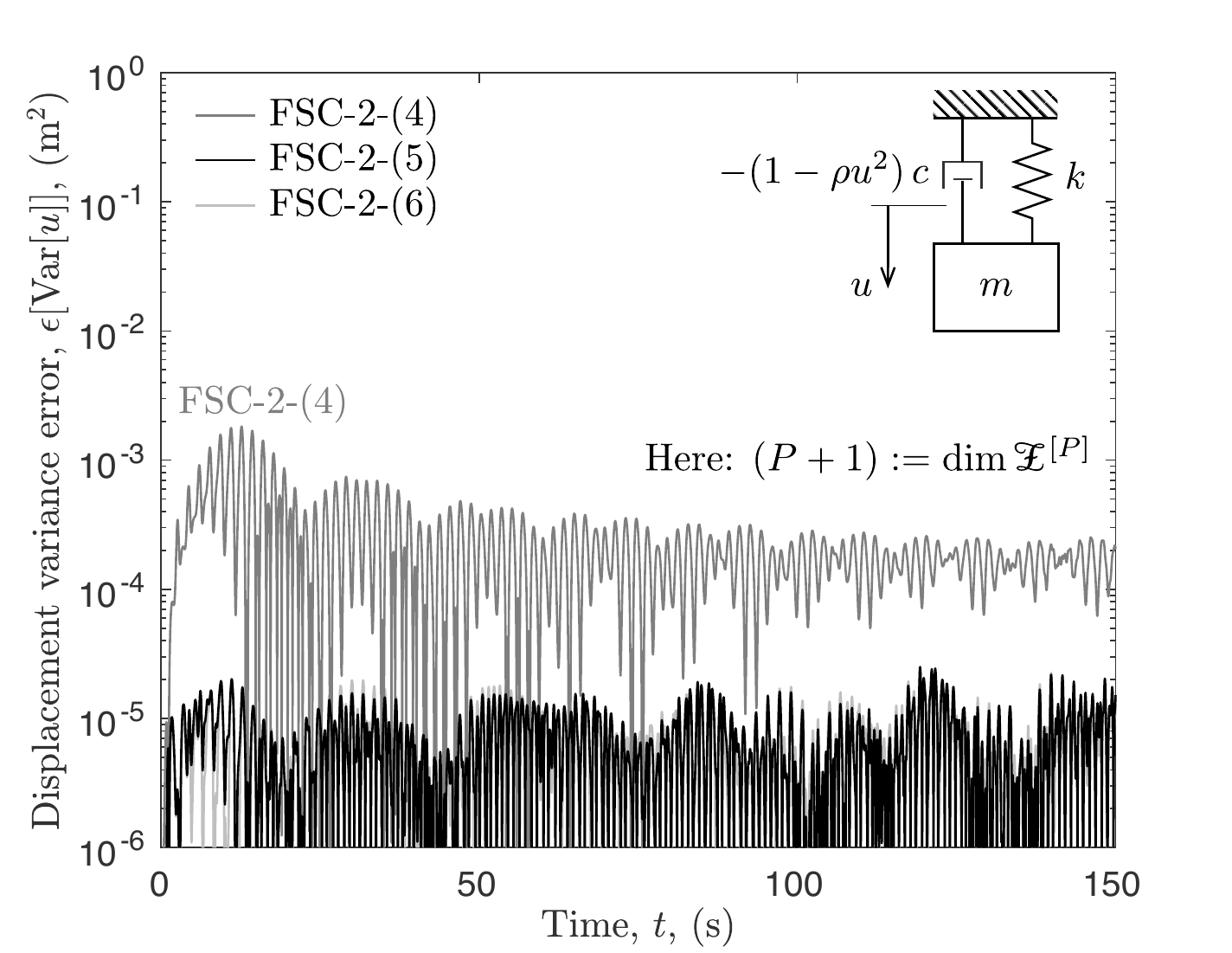}
\caption{Variance error}
\label{fig2System5_UniformBeta_FSC_Disp_Var_Error}
\end{subfigure}
\caption{\emph{Problem 6 (the Van-der-Pol oscillator)} --- Local error evolution of $\mathbf{E}[u]$ and $\mathrm{Var}[u]$ for different $p$-discretization levels of RFS with respect to the 1-million Monte Carlo simulation ($\mu\sim\mathrm{Uniform}\otimes\mathrm{Beta}$)}
\label{fig2System5_UniformBeta_FSC_Disp_Error}
\end{figure}

For this section we use the RK4 method with $\Delta t=0.005$ s to speed up the simulations.
Fig.~\ref{fig2System5_UniformBeta_FSC_Disp_5} depicts the evolution of the mean and variance of the oscillator's displacement using FSC-2 and a Monte Carlo simulation with one million realizations.
This figure shows that by using only 5 basis vectors, FSC-2 has the ability to reproduce the Monte Carlo response with high fidelity.
This is quite a remarkable result, since not only the problem is highly nonlinear over the temporal-random space but it also features a two-dimensional probability space. 

Fig.~\ref{fig2System5_UniformBeta_FSC_Disp_Error} plots the local errors in mean and variance of the oscillator's displacement.
We observe that by using 4 basis vectors, an error of $10^{-2}$ and $10^{-3}$ can be achieved for the mean and variance, respectively.
If more basis vectors are used, say 5 or 6, the corresponding errors cut down two orders of magnitude.
This suggests that 5 basis vectors are sufficient to reproduce a Monte Carlo simulation with one million realizations to a comparable level of accuracy.

\section{A parametric, high-dimensional stochastic problem}\label{sec2StoPro5Dim}

In this section, we show that the FSC method does not suffer from the curse of dimensionality at the random-function-space level by solving a parametric, high-dimensional problem.
Further, we show that by using the FSC method in conjunction with Monte Carlo integration to compute the inner products, one can overcome the curse of dimensionality at the random-space level as well---effectively eliminating the curse of dimensionality altogether.
For this, we investigate the following problem.

\subsection{Problem statement}

Find the displacement of the system $u:\mathfrak{T}\times\Xi\to\mathbb{R}$ in $\mathscr{U}$, such that ($\mu$-a.e.):
\begin{subequations}\label{appeq2StoPro5Dim100}
\begin{align}
\ddot{u}+ku=f&\qquad\text{on $\mathfrak{T}\times\Xi$}\label{appeq2StoPro5Dim100a}\\
\big\{u(0,\cdot\,)=\mathscr{u},\,\dot{u}(0,\cdot\,)=\mathscr{v}\big\}&\qquad\text{on $\{0\}\times\Xi$},\label{appeq2StoPro5Dim100b}
\end{align}
\end{subequations}
where $k,f:\mathfrak{T}\times\Xi\to\mathbb{R}$ and $\mathscr{u},\mathscr{v}:\Xi\to\mathbb{R}$ are random variables given by:
\begin{gather*}
k(t,\xi)=\tfrac{1}{2400}\big(\xi^1+40\big)\big(\xi^6+\xi^7+40\big)\big(3-\exp(-\tfrac{1}{77}(\xi^2+7)t)\big),\\
f(t,\xi)=\tfrac{1}{1200}\big(\xi^3+7\big)\big(\xi^8+40\big)\big((\xi^9+\xi^{10})\sin(\tfrac{1}{7}\pi t)+3\big),\\
\mathscr{u}(\xi)=\tfrac{1}{7}(\xi^4+8)\quad\text{and}\quad\mathscr{v}(\xi)=\tfrac{1}{8}(\xi^5+7).
\end{gather*}

In this problem we assume the following probability distribution: $\mathrm{Beta}^{\otimes 3}\otimes\mathrm{Uniform}^{\otimes d-3}\sim\xi\in\Xi=[-1,1]^d$.
For the beta distribution, we take $(\alpha,\beta)=(2,5)$ as the parameters of the distribution.
Since the random space is a parametric $d$-dimensional space, here we study specifically the cases when $d=5,7,10$.
In particular, when $d=5$, we take $\xi^6=\xi^7=\xi^8=\xi^9=\xi^{10}=0$, and when $d=7$, we put $\xi^8=\xi^9=\xi^{10}=0$.
We recall that $\xi=(\xi^1,\ldots,\xi^d)$ for a $d$-dimensional space.
The spectral discretization of this problem is derived in detail in Appendix \ref{appsec2StoPro5Dim}.

\subsection{Discussion on numerical results}

To run all the simulations, we employ the RK4 method with a time-step size of 0.01 s.
In Table \ref{apptab2StoPro5Dim1000} we have provided the non-orthogonalized version of the random basis for use within the FSC scheme.
Because the random space is in this case high-dimensional, the inner products are computed with a Monte Carlo integration using $10^5$ quadrature points sampled from the random domain.
For reference, each FSC simulation is compared below against a Monte Carlo simulation with one million realizations to determine if convergence was achieved for the FSC-2 scheme.

In Fig.~\ref{fig2System6_Beta3UniformX_FSC4_MC1e5_Disp} we plot the evolution of the mean and variance of the system's displacement for the three stochastic problems considered in this section.
Two observations are in order.
First, for the smallest RFS that one can construct using the FSC method (i.e.~$\mathscr{Z}^{[3]}$ for this problem), the FSC solution is capable of reproducing reasonably well the Monte Carlo solution.
Second, as the dimensionality of the random space is increased, we see that there is no need to increase the dimensionality of the RFS to achieve good results.
This is different from other spectral methods such as gPC or TD-gPC, which would have required a lot more basis vectors to span the RFS.
For example, if a total-order tensor product were utilized to solve the 10-dimensional problem, expression \eqref{eq2SetNot1030A} would have indicated the need of using 66 basis vectors to enrich the random basis with quadratic functions.
However, as we see in Fig.~\ref{fig2System6_Beta3UniformX_FSC4_MC1e5_Disp},
only 4 basis vectors are needed in the simulations with FSC to achieve solutions that are nearly indistinguishable from the Monte Carlo counterpart.
This proves---at least from a numerical standpoint---our assertion that the FSC method does not suffer from the curse of dimensionality at the random-function-space level.

\begin{figure}
\centering
\begin{subfigure}[b]{0.495\textwidth}
\includegraphics[width=\textwidth]{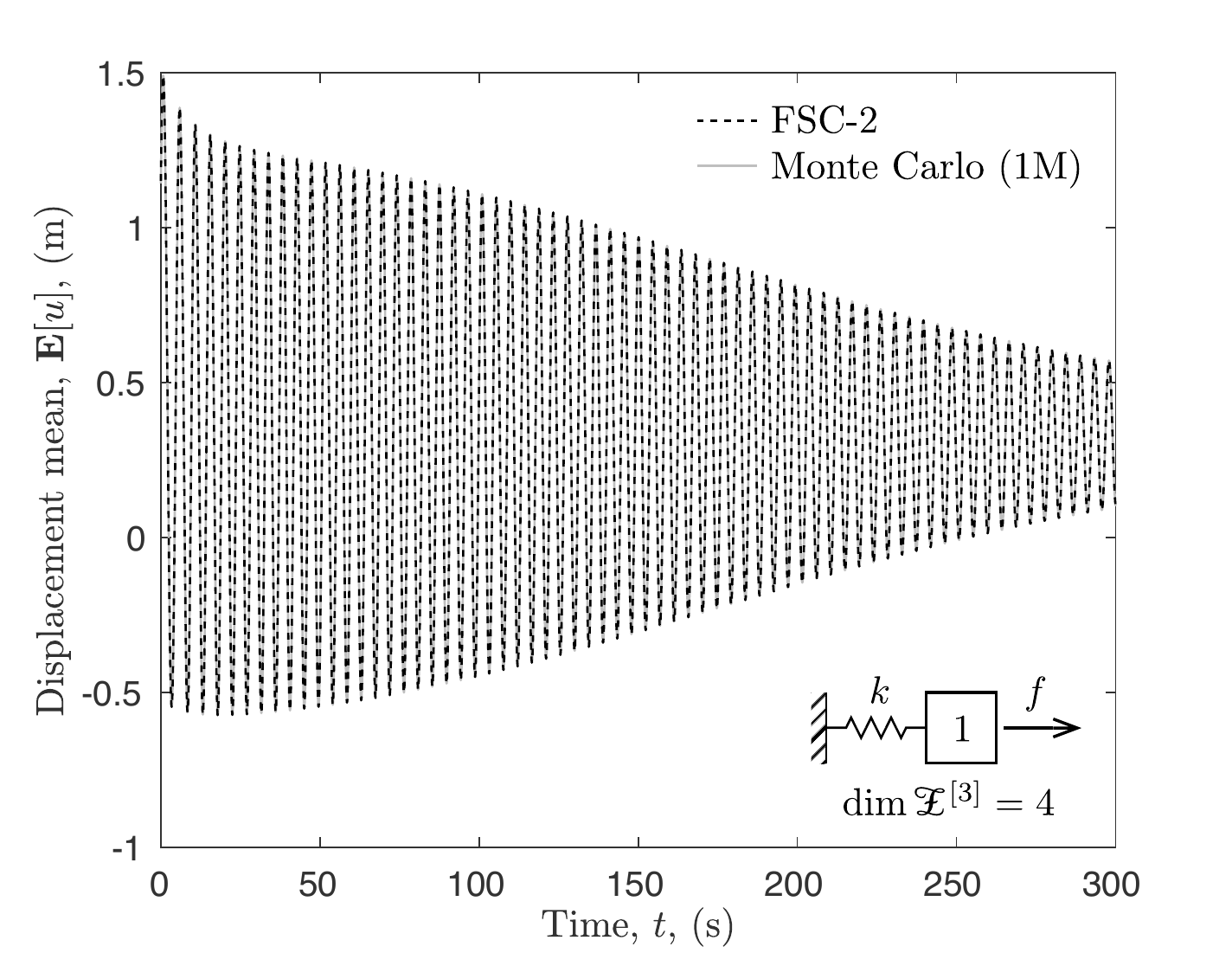}
\caption{Mean for $d=5$}
\label{fig2System6_Beta3Uniform2_FSC4_MC1e5_Disp_Mean}
\end{subfigure}\hfill
\begin{subfigure}[b]{0.495\textwidth}
\includegraphics[width=\textwidth]{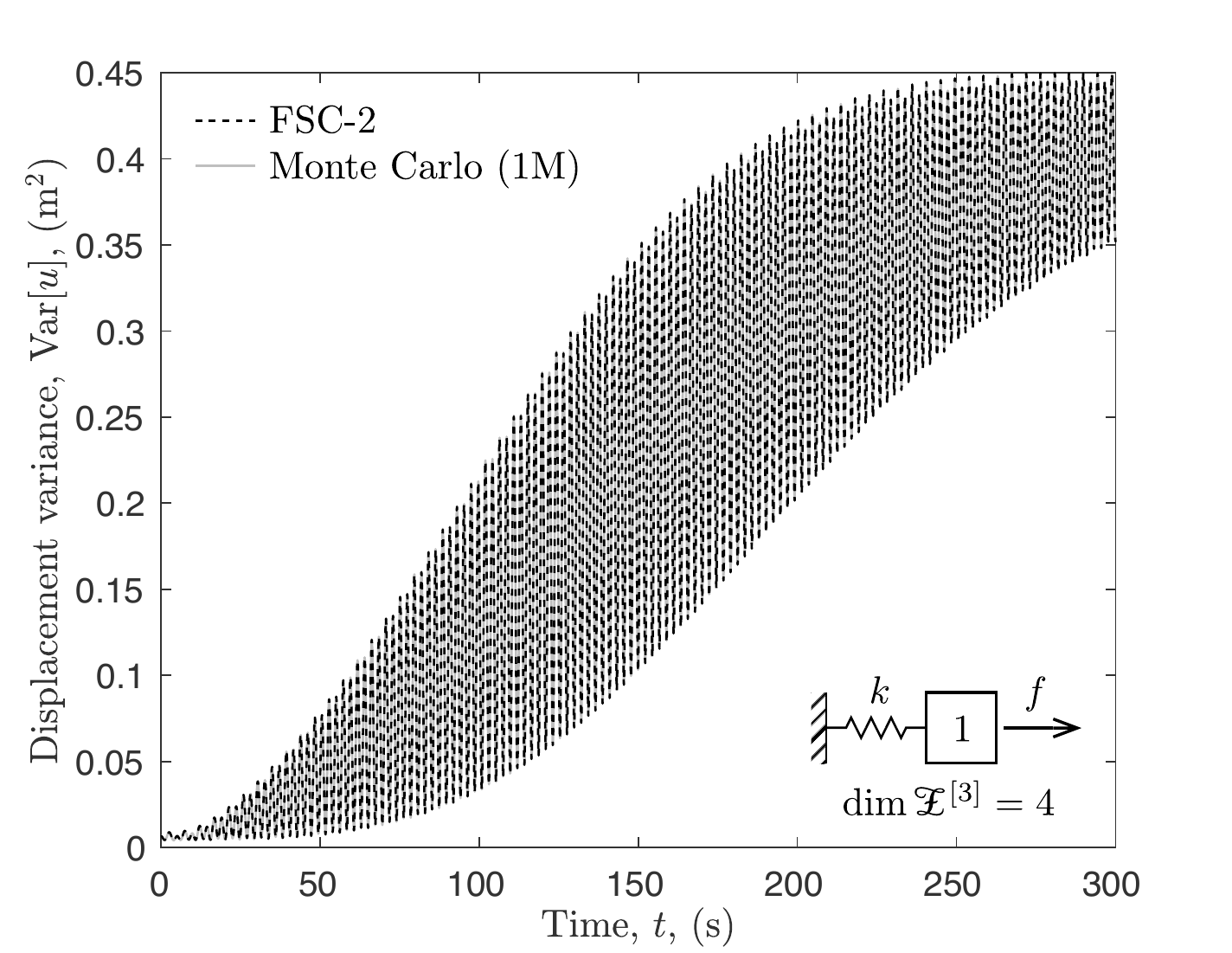}
\caption{Variance for $d=5$}
\label{fig2System6_Beta3Uniform2_FSC4_MC1e5_Disp_Var}
\end{subfigure}\quad
\begin{subfigure}[b]{0.495\textwidth}
\includegraphics[width=\textwidth]{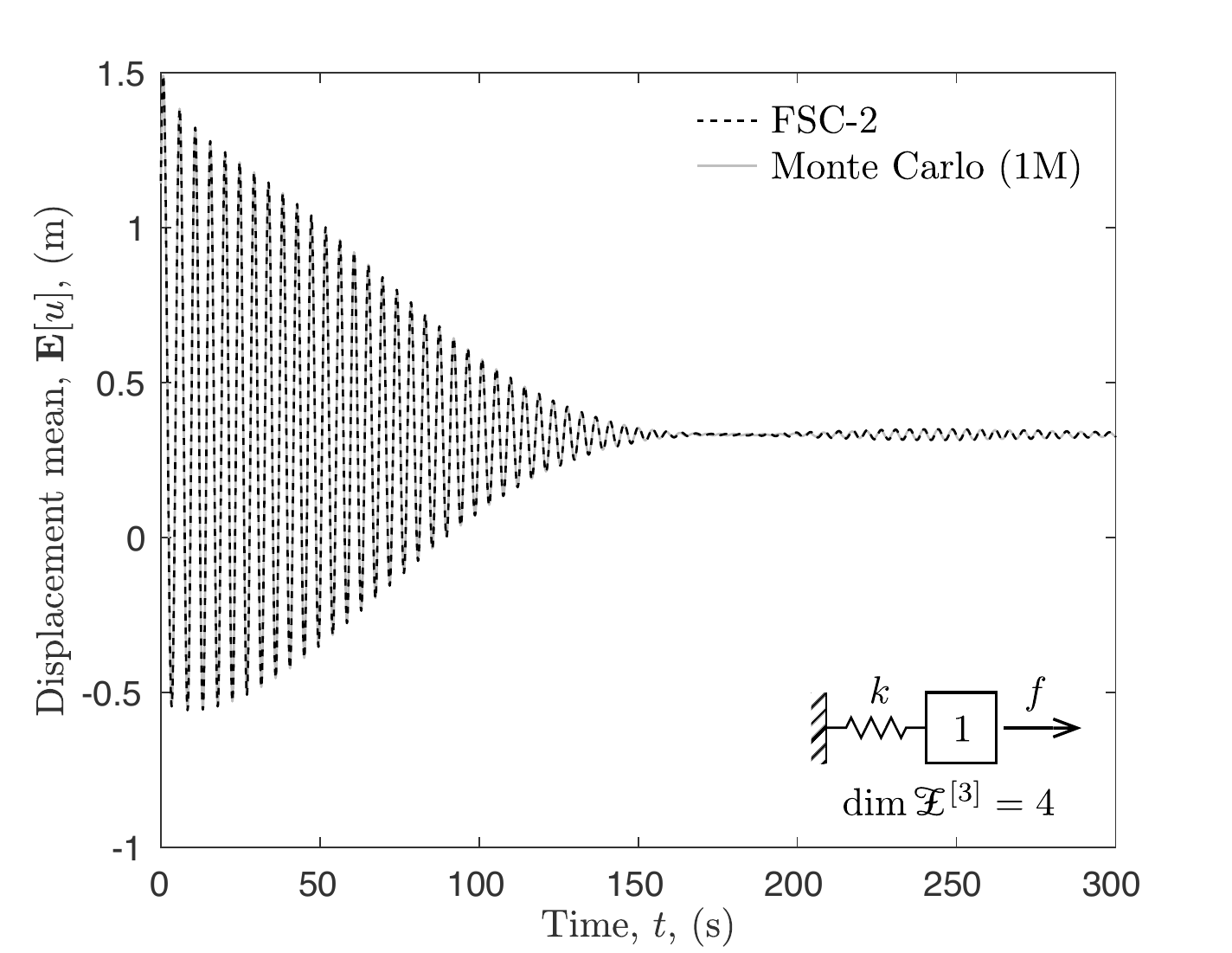}
\caption{Mean for $d=7$}
\label{fig2System6_Beta3Uniform4_FSC4_MC1e5_Disp_Mean}
\end{subfigure}\hfill
\begin{subfigure}[b]{0.495\textwidth}
\includegraphics[width=\textwidth]{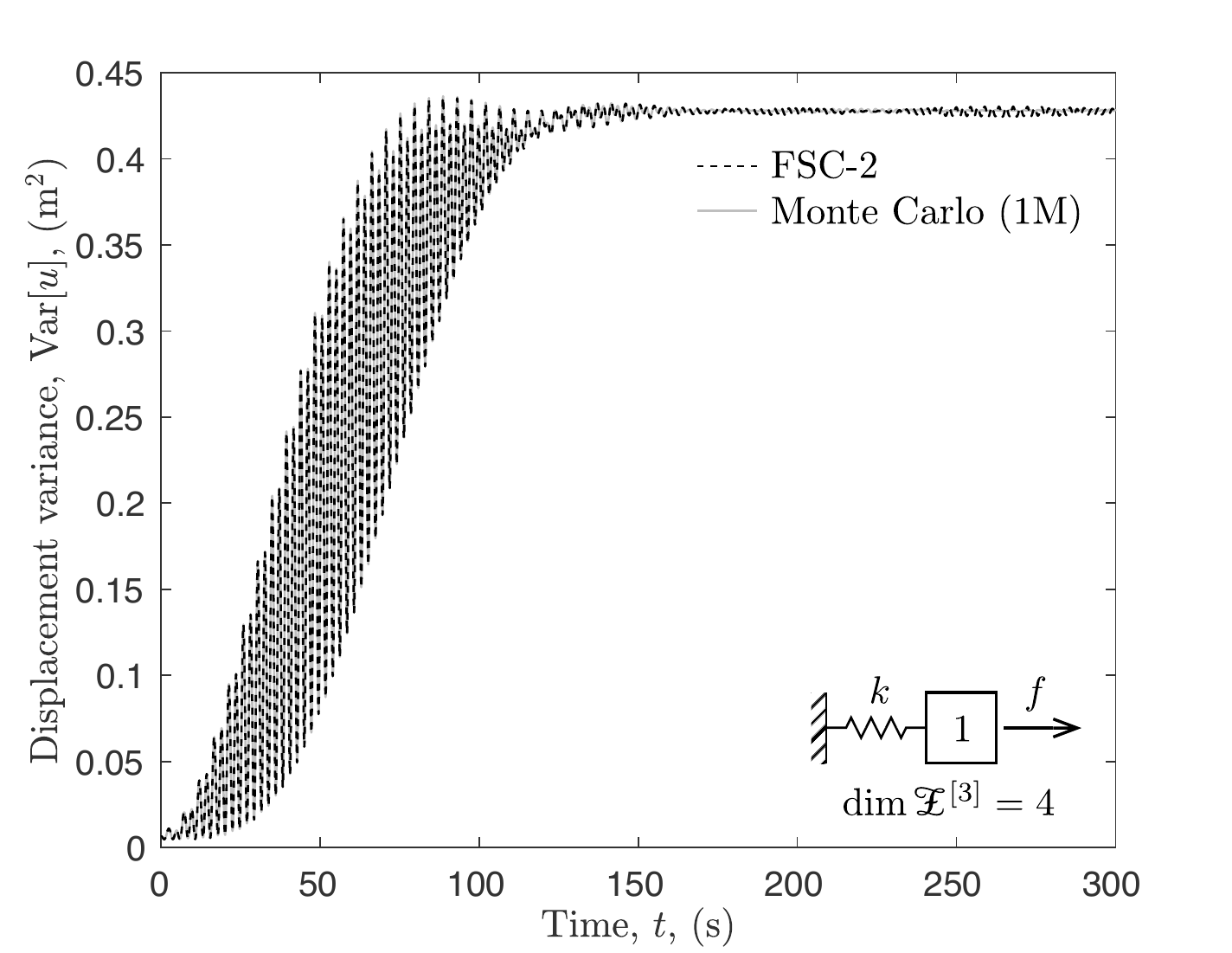}
\caption{Variance for $d=7$}
\label{fig2System6_Beta3Uniform4_FSC4_MC1e5_Disp_Var}
\end{subfigure}\quad
\begin{subfigure}[b]{0.495\textwidth}
\includegraphics[width=\textwidth]{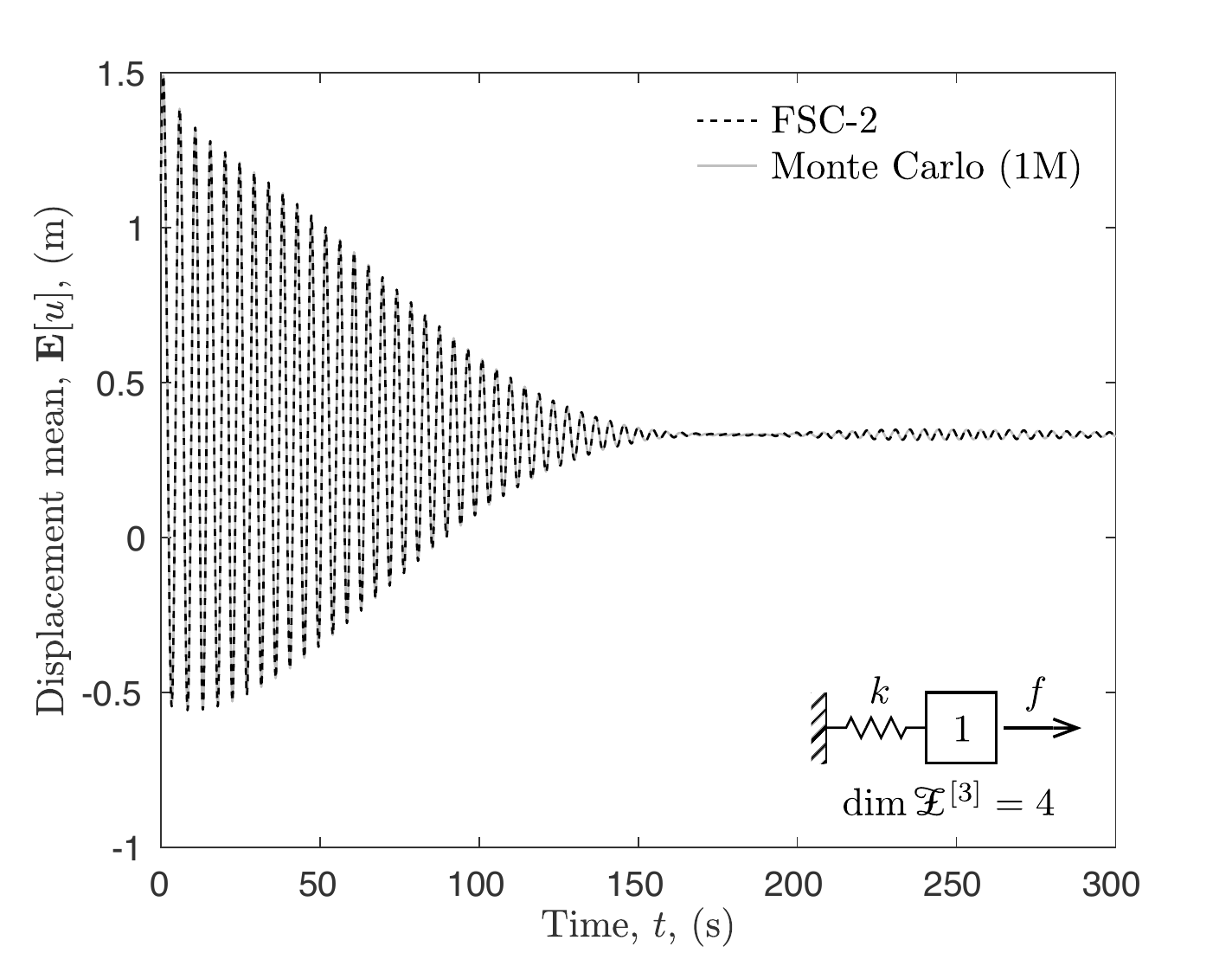}
\caption{Mean for $d=10$}
\label{fig2System6_Beta3Uniform7_FSC4_MC1e5_Disp_Mean}
\end{subfigure}\hfill
\begin{subfigure}[b]{0.495\textwidth}
\includegraphics[width=\textwidth]{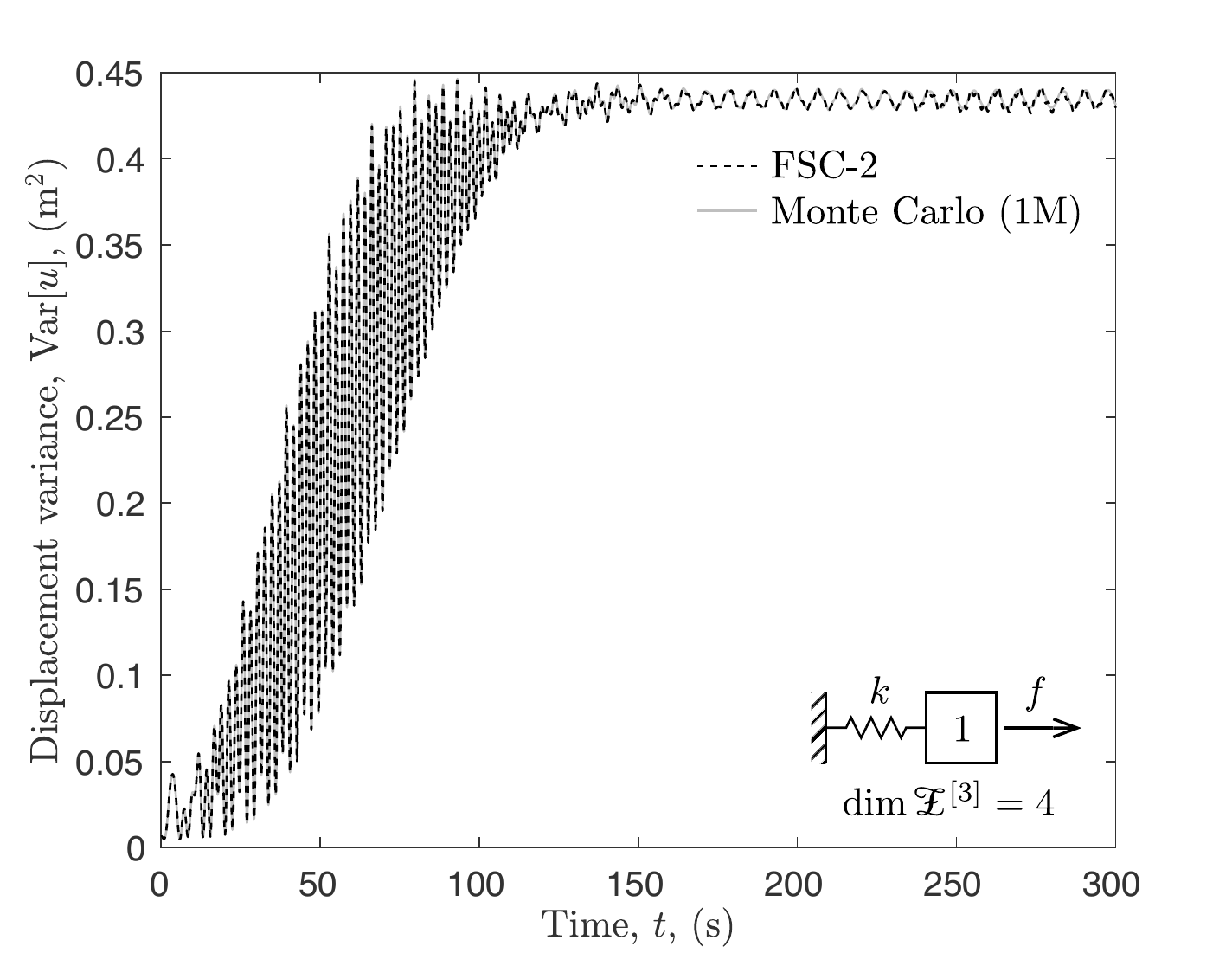}
\caption{Variance for $d=10$}
\label{fig2System6_Beta3Uniform7_FSC4_MC1e5_Disp_Var}
\end{subfigure}
\caption{Evolution of $\mathbf{E}[u]$ and $\mathrm{Var}[u]$ for the case when the $p$-discretization level of RFS is $\mathscr{Z}^{[3]}$ and $\mu\sim\mathrm{Beta}^{\otimes 3}\otimes\mathrm{Uniform}^{\otimes d-3}$}
\label{fig2System6_Beta3UniformX_FSC4_MC1e5_Disp}
\end{figure}

To study now how the number of quadrature points affects the estimation of the inner products with a Monte Carlo integration---and consequently, the FSC results---, we run the FSC-2 simulations using five realizations of $10^2$, $10^3$, $10^4$ and $10^5$ quadrature points.
The resulting solutions are then compared against a reference solution obtained from performing a Monte Carlo simulation with one million realizations.
In Fig.~\ref{fig2System6_FSC4_Disp} we present this study for the three stochastic problems considered in this section by computing the global errors obtained from the mean and variance of the system's displacement.
The first observation that we can make is that the variability of the five realizations is in general higher when using fewer quadrature points.
This is not surprising since in Monte Carlo integration is well recognized that the fewer the number of quadrature points, the more variation in the results is expected from different realizations of quadrature points.
A second observation is that the accuracy of the results improves as the number of quadrature points increases.
This is once again expected for it is a direct consequence of the law of large numbers in probability theory.
This parametric, high-dimensional stochastic problem has therefore shown that when the dimensionality of the probability space is high, the inner products can be computed with a Monte Carlo integration technique without producing adverse effects, such as non-convergence issues, in the FSC scheme.

From this section it follows that if we combine the FSC method with a Monte Carlo integration technique, the curse of dimensionality can be fully annihilated at both random levels, namely: at the random-space level and at the random-function-space level.
This makes the FSC method more suitable for solving higher-dimensional stochastic problems with the spectral approach.

\begin{figure}
\centering
\begin{subfigure}[b]{0.495\textwidth}
\includegraphics[width=\textwidth]{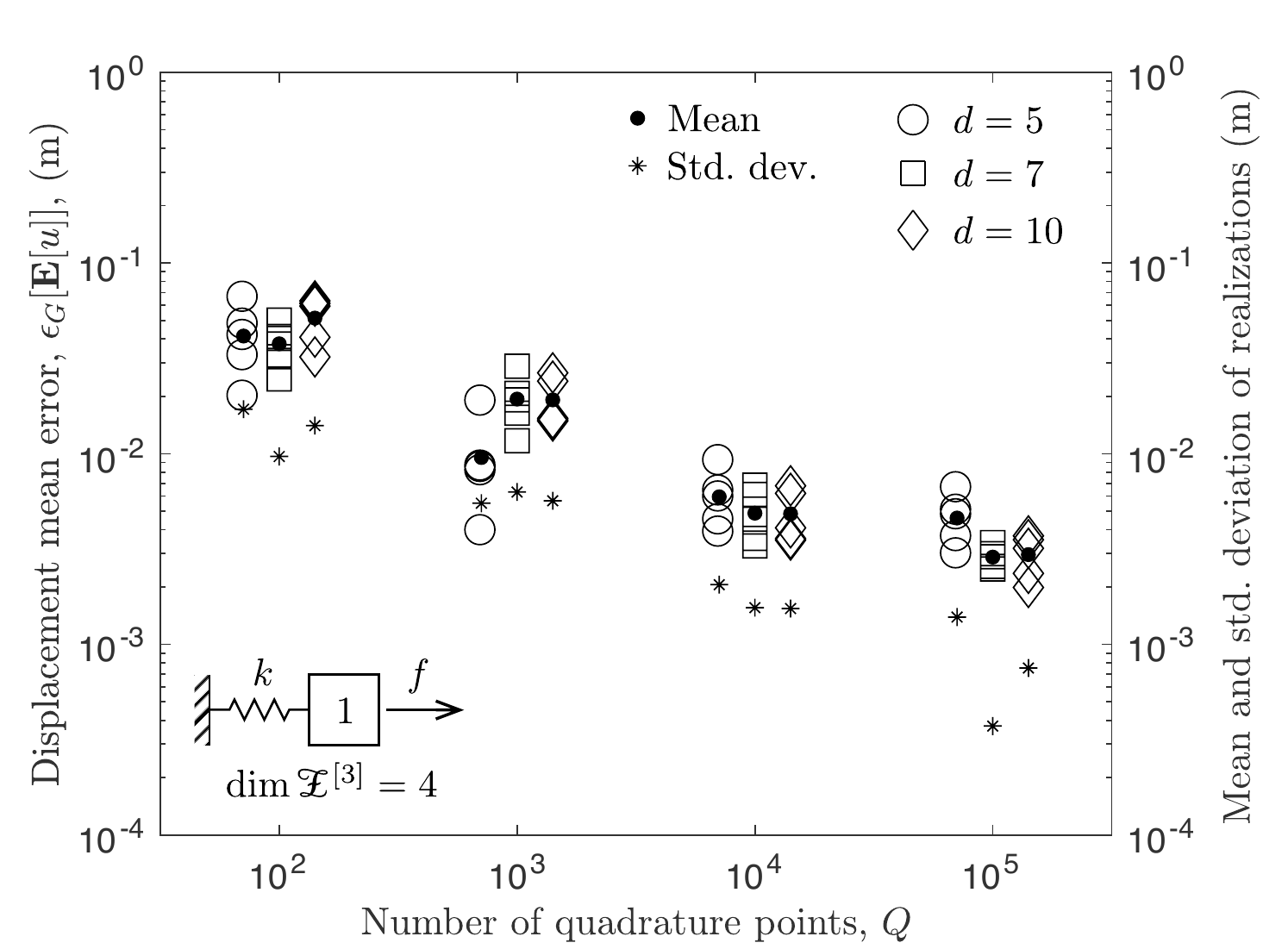}
\caption{Mean error}
\label{fig2System6_FSC4_Disp_Mean}
\end{subfigure}\hfill
\begin{subfigure}[b]{0.495\textwidth}
\includegraphics[width=\textwidth]{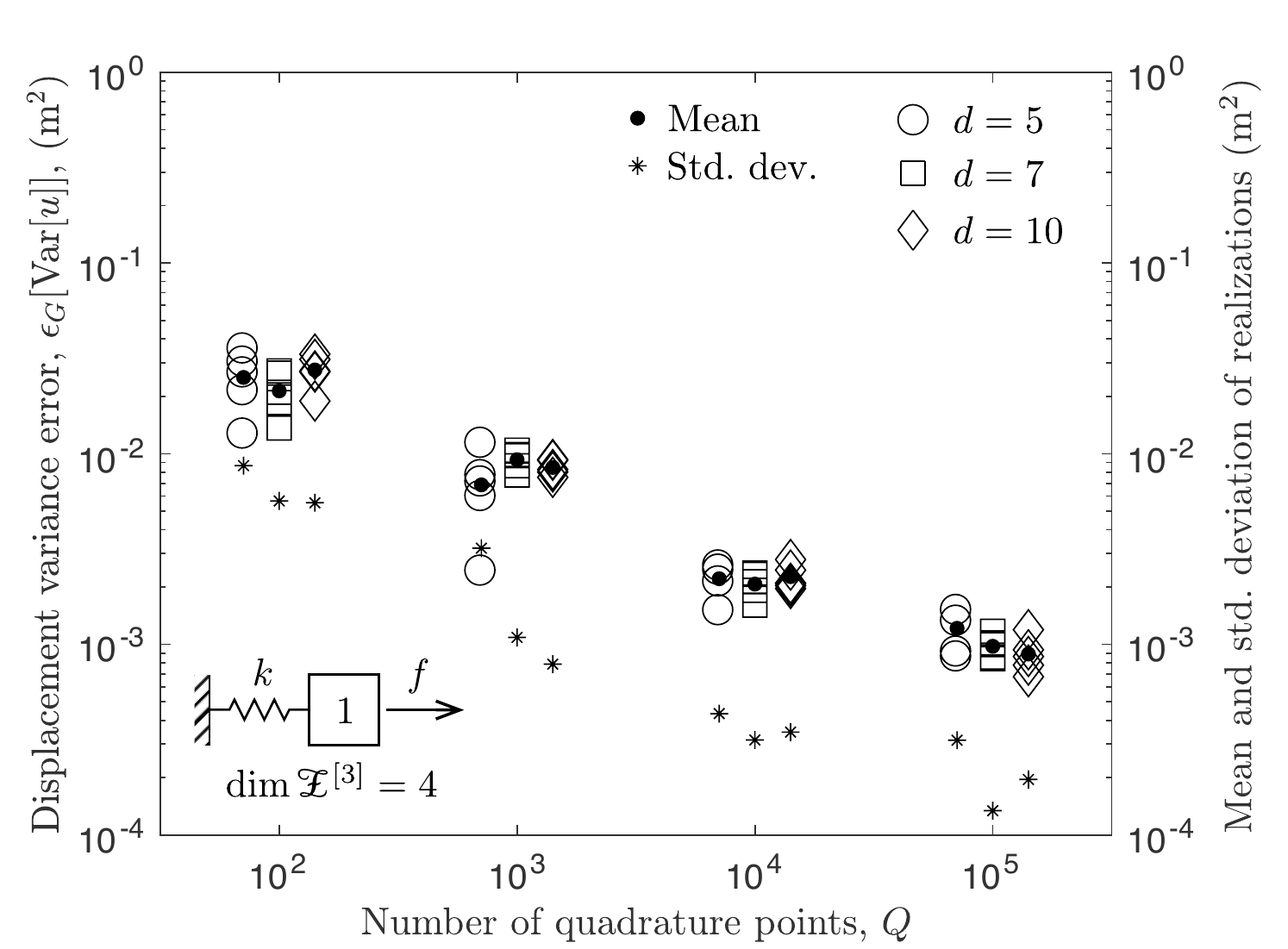}
\caption{Variance error}
\label{fig2System6_FSC4_Disp_Var}
\end{subfigure}
\caption{Global error of $\mathbf{E}[u]$ and $\mathrm{Var}[u]$ as a function of the number of quadrature points employed to estimate the inner products with a Monte Carlo integration ($\mu\sim\mathrm{Beta}^{\otimes 3}\otimes\mathrm{Uniform}^{\otimes d-3}$)}
\label{fig2System6_FSC4_Disp}
\end{figure}

\section{Conclusion}\label{sec2con}

A novel numerical method called the \emph{flow-driven spectral chaos} (FSC) is proposed in this paper to quantify uncertainties in the long-time response of stochastic dynamical systems.
In the FSC method, we use the concept of enriched stochastic flow maps to track the evolution of the stochastic part of the solution space efficiently in time.
This track is motivated by the fact that the solution of a stochastic ODE and its probability distribution change significantly as a function of time.
(It is well-known that when the state of a system is pushed one-time step forward in the simulation, the random basis loses unavoidably its optimality.)
Thus, to resolve this long-time integration issue, we span the stochastic part of the solution space via the time derivatives of the solution at the current time for use within the time step of the simulation.
This new way of approaching the problem follows upon noting that a Taylor-series expansion can decompose a stochastic process into an infinite series of functions in the form of a product of a temporal function and a random function (this is of course provided that the process is assumed analytic on the temporal domain).
The random functions thus generated are subsequently orthogonalized to serve as a random basis for the solution space.
To transfer the probability information at any instant of time, two approaches are developed herein, FSC-1 and FSC-2.
The first approach (FSC-1) enforces the probability information to be transferred in the mean-square sense, whereas the second approach (FSC-2) ensures that the probability information is transferred exactly.
As discussed in Section \ref{sec2DisNumRes}, the FSC-2 approach has not only the ability to produce results that are more accurate than FSC-1, but also the ability to transfer the probability information faster than FSC-1.
This is especially true if the order of the ODE is low (provided we keep the dimensionality of the random function space fixed).
Therefore, we suggest using the FSC-2 approach when referring to the FSC method.

We have shown that the FSC method is insensitive to the curse of dimensionality at the random-function-space level.
This is because in practice the stochastic flow map of a system is chosen to be finite-accurate, which allows the aforementioned Taylor-series expansion to be truncated up to a specific order.
This is in contrast to other methods such as gPC \cite{xiu2002wiener} and TD-gPC \cite{gerritsma2010time}, which use the concept of polynomial chaos and tensor products to find a suitable random basis for the solution space, and which are known to suffer from the curse of dimensionality to some extent.
This curse of dimensionality has been regarded as a fundamental issue in methods based on the spectral approach since the introduction of the PC method back in 1938.
This paper has addressed this fundamental issue at the random-function-space level.

The FSC method has been applied to six representative problems in this paper.
The first five deal with systems governed by a linear stochastic ODE, while the last one is a system governed by a nonlinear stochastic ODE (the Van-der-Pol oscillator).
These ODEs were selected to range from first to fourth order so that the performance of the FSC method could be investigated more thoroughly.
Based on our findings, we can conclude that FSC outperforms TD-gPC in both accuracy and computational efficiency for solving stochastic dynamical systems with complex physics.
Furthermore, in Section \ref{sec2StoPro5Dim} we solved three high-dimensional stochastic problems to demonstrate that the curse of dimensionality can be overcome at both, the random-function-space level and the random-space level, by using the FSC method together with Monte Carlo integration to compute the inner products.

\begin{appendices}
\section{The Van-der-Pol oscillator}\label{appsec2DerVanderPolOsc}

\subsection{Discretization of the random function space}

The random function space for the Van-der-Pol oscillator described in Section \ref{sec2NumExa60} is discretized here for sake of reference.
Let $\mathscr{Z}^{[P]}$ be a finite subspace of $\mathscr{Z}$.
Then, it is clear that the oscillator's displacement can be represented in $\mathscr{Z}^{[P]}$ with:
\begin{equation}\label{appeq2VanderPol100}
u(t,\xi)\approx u^{[P]}(t,\xi)=\sum_{j=0}^P u^j(t)\,\Psi_j(\xi)\equiv u^j(t)\,\Psi_j(\xi).
\end{equation}

Similarly, let $\tilde{\mathscr{Z}}$ be a finite subspace of $\mathscr{Z}$ to represent the stochastic-input space of the system \emph{exactly}.
This subspace is defined as $\tilde{\mathscr{Z}}=\mathrm{span}\{\tilde{\Psi}_m\}_{m=0}^2$, where $\tilde{\Psi}_0\equiv1$, $\tilde{\Psi}_1=c-\mathbf{E}[c]$ and $\tilde{\Psi}_2=\mathscr{u}-\mathbf{E}[\mathscr{u}]$.
Hence, we have:
\begin{equation}\label{appeq2VanderPol200}
c(\xi)=\sum_{m=0}^2 c^m\,\tilde{\Psi}_m(\xi)\equiv c^m\,\tilde{\Psi}_m(\xi)\quad\text{and}\quad
\mathscr{u}(\xi)=\sum_{m=0}^2\mathscr{u}^{\!m}\,\tilde{\Psi}_m(\xi)\equiv\mathscr{u}^{\!m}\,\tilde{\Psi}_m(\xi),
\end{equation}
whence $c^0=\mathbf{E}[c]$, $c^1=1$, $c^2=0$, $\mathscr{u}^{\!0}=\mathbf{E}[\mathscr{u}]$, $\mathscr{u}^{\!1}=0$ and $\mathscr{u}^{\!2}=1$ are the coefficients of the finite series.

\begin{remark}
In this problem, $\mathrm{dim}\,\tilde{\mathscr{Z}}=3$ because $c$ and $\mathscr{u}$ are the only independent random variables in the model's input.
In a more general setting, $\mathrm{dim}\,\tilde{\mathscr{Z}}=\tilde{N}+1$, where $\tilde{N}$ represents the number of independent random variables specified in input $x$ of mathematical model $\boldsymbol{\mathcal{M}}[u]$.
\end{remark}

Replacing \eqref{appeq2VanderPol100} and \eqref{appeq2VanderPol200} into \eqref{eq2NumExa500}, and then projecting onto $\mathscr{Z}^{[P]}$, yields a system of $P+1$ second-order ordinary differential equations in the variable $t$ (with initial conditions):
\begin{subequations}\label{appeq2VanderPol300}
\begin{align}
m\,\mathbb{M}\indices{^i_{j000}}\ddot{u}^j-\mathbb{M}\indices{^i_{j00m}}\dot{u}^jc^m+\rho\,\mathbb{M}\indices{^i_{jklm}}\dot{u}^ju^ku^lc^m+k\,\mathbb{M}\indices{^i_{j000}}u^j=0&\qquad\text{on $\mathfrak{T}$}\label{appeq2VanderPol300a}\\
\big\{\mathbb{M}\indices{^i_{j000}}\,u^j(0)=\mathbb{M}\indices{^i_{000m}}\mathscr{u}^{\!m},\,\mathbb{M}\indices{^i_{j000}}\,\dot{u}^j(0)=2\,\mathbb{M}\indices{^i_{000m}}\mathscr{u}^{\!m}-0.10\,\mathbb{M}\indices{^i_{0000}}\big\}&\qquad\text{on $\{0\}$,}\label{appeq2VanderPol300b}
\end{align}
\end{subequations}
where $\mathbb{M}=\mathbb{M}\indices{^i_{jklm}}\Psi_i\otimes\Psi^j\otimes\Psi^k\otimes\Psi^l\otimes\tilde{\Psi}^m:{\mathscr{Z}^{[P]}}'\times\mathscr{Z}^{[P]}\times\mathscr{Z}^{[P]}\times\mathscr{Z}^{[P]}\times\tilde{\mathscr{Z}}\to\mathbb{R}$ is the discretized \emph{random tensor}\footnote{This tensor is also known as \emph{multiplication tensor} in the literature, see e.g.~\cite{le2010spectral,sullivan2015introduction}.} for the whole dynamical system (assumed to exist) given by (a summation sign is implied over every repeated index)
\begin{equation*}
\mathbb{M}[\alpha,w,x,y,z]=\mathbb{M}\indices{^i_{jklm}}\,\alpha_i\,w^j\,x^k\,y^l\,z^m,
\end{equation*}
whence $i,j,k,l\in\{0,1,\ldots,P\}$, $m\in\{0,1,2\}$, and
\begin{equation*}
\mathbb{M}\indices{^i_{jklm}}=\frac{\langle\Psi_i,\Psi_j\Psi_k\Psi_l\tilde{\Psi}_m\rangle}{\langle\Psi_i,\Psi_i\rangle}.
\end{equation*}
This $\mathbb{M}$ is a tensor of type $(1,4)$ and symmetric in the indices $j$, $k$ and $l$.

However, since $\mathbb{M}\indices{^i_{j000}}$ is nothing but $\delta\indices{^i_j}$ (the Kronecker delta), system \eqref{appeq2VanderPol300} can be restated as follows:
\begin{align*}
m\ddot{u}^i-\mathbb{M}\indices{^i_{j00m}}\dot{u}^jc^m+\rho\,\mathbb{M}\indices{^i_{jklm}}\dot{u}^ju^ku^lc^m+ku^i=0&\qquad\text{on $\mathfrak{T}$}\\
\big\{u^i(0)=\mathbb{M}\indices{^i_{000m}}\mathscr{u}^{\!m},\,\dot{u}^i(0)=2\,\mathbb{M}\indices{^i_{000m}}\mathscr{u}^{\!m}-0.10\,\delta\indices{^i_0}\big\}&\qquad\text{on $\{0\}$.}
\end{align*}

\subsection{Random basis for use in the FSC scheme}

Table \ref{apptab2RanBasTimInt1000} presents the non-orthogonalized version of the random basis that we use in the simulations with FSC over the region $\mathfrak{R}_i=\mathrm{cl}(\mathfrak{T}_i)\times\Xi$.

\makeatletter
\def\tabcaption{Definition of $\tau$-functions for a single-degree-of-freedom system subjected to free vibration with $k\sim\mathrm{Uniform}$ in $[k_a,k_b]$}
\def\tablabel{apptab1SDOF1000}
\@ifundefined{puformat}%
{
\begin{table}
\centering\small
\caption{Non-orthogonalized version of the random basis used for the Van-der-Pol oscillator (Section \ref{sec2NumExa60})}
\label{apptab2RanBasTimInt1000}
\begin{tabular}{@{}p{\textwidth}@{}}
\toprule\\[-5ex]
\begin{equation*}
\Phi_{0.i}(\xi):=1
\end{equation*}\\[-3.25ex]
\midrule\\[-5ex]
\begin{equation*}
\Phi_{1.i}(\xi):=\hat{\varphi}^1(M)(0,\hat{s}(t_i,\xi))= u_{.i}(t_i,\xi)=u_{.i-1}(t_i,\xi)
\end{equation*}\\[-3.25ex]
\midrule\\[-5ex]
\begin{equation*}
\Phi_{2.i}(\xi):=\hat{\varphi}^2(M)(0,\hat{s}(t_i,\xi))= \dot{u}_{.i}(t_i,\xi)=\dot{u}_{.i-1}(t_i,\xi)
\end{equation*}\\[-3.25ex]
\midrule\\[-4.5ex]
\begin{equation*}
\Phi_{3.i}(\xi):=\hat{\varphi}^3(M)(0,\hat{s}(t_i,\xi))=\partial_t^2u_{.i}(t_i,\xi)=\frac{1}{m}\big((1-\rho\,u{^2}_{.i}(t_i,\xi))\,c(\xi)\,\dot{u}_{.i}(t_i,\xi)-k\,u_{.i}(t_i,\xi)\big)
\end{equation*}\\[-2ex]
\midrule\\[-7.5ex]
\begin{multline*}
\Phi_{4.i}(\xi):=\hat{\varphi}^4(M)(0,\hat{s}(t_i,\xi))=\partial_t^3u_{.i}(t_i,\xi)=\\
\frac{1}{m}\big({-2\rho}\,c(\xi)\,u_{.i}(t_i,\xi)\,\dot{u}{^2}_{.i}(t_i,\xi)+(1-\rho\,u{^2}_{.i}(t_i,\xi))\,c(\xi)\,\partial_t^2{u}_{.i}(t_i,\xi)-k\,\dot{u}_{.i}(t_i,\xi)\big)
\end{multline*}\\[-3.25ex]
\midrule\\[-7.5ex]
\begin{multline*}
\Phi_{5.i}(\xi):=\hat{\varphi}^5(M)(0,\hat{s}(t_i,\xi))=\partial_t^4u_{.i}(t_i,\xi)=\frac{1}{m}\big({-2\rho}\,c(\xi)\,\dot{u}{^3}_{.i}(t_i,\xi)\\
-6\rho\,c(\xi)\,u_{.i}(t_i,\xi)\,\dot{u}_{.i}(t_i,\xi)\,\partial_t^2u_{.i}(t_i,\xi)+(1-\rho\,u{^2}_{.i}(t_i,\xi))\,c(\xi)\,\partial_t^3{u}_{.i}(t_i,\xi)-k\,\partial_t^2u_{.i}(t_i,\xi)\big)
\end{multline*}\\[-3ex]
\bottomrule
\end{tabular}
\end{table}
}%
{
\begin{table}[p!]
\centering\small
\caption{Non-orthogonalized version of the random basis used for the Van-der-Pol oscillator (Section \ref{sec2NumExa60})}
\label{apptab2RanBasTimInt1000}
\begin{tabular}{@{}p{\textwidth}@{}}
\toprule\\[-8.5ex]
\begin{equation*}
\Phi_{0.i}(\xi):=1
\end{equation*}\\[-5ex]
\midrule\\[-8.5ex]
\begin{equation*}
\Phi_{1.i}(\xi):=\hat{\varphi}^1(M)(0,\hat{s}(t_i,\xi))= u_{.i}(t_i,\xi)=u_{.i-1}(t_i,\xi)
\end{equation*}\\[-5ex]
\midrule\\[-8.5ex]
\begin{equation*}
\Phi_{2.i}(\xi):=\hat{\varphi}^2(M)(0,\hat{s}(t_i,\xi))= \dot{u}_{.i}(t_i,\xi)=\dot{u}_{.i-1}(t_i,\xi)
\end{equation*}\\[-5ex]
\midrule\\[-9.5ex]
\begin{equation*}
\Phi_{3.i}(\xi):=\hat{\varphi}^3(M)(0,\hat{s}(t_i,\xi))=\partial_t^2u_{.i}(t_i,\xi)=\frac{1}{m}\big((1-\rho\,u{^2}_{.i}(t_i,\xi))\,c(\xi)\,\dot{u}_{.i}(t_i,\xi)-k\,u_{.i}(t_i,\xi)\big)
\end{equation*}\\[-5ex]
\midrule\\[-11ex]
\begin{multline*}
\Phi_{4.i}(\xi):=\hat{\varphi}^4(M)(0,\hat{s}(t_i,\xi))=\partial_t^3u_{.i}(t_i,\xi)=\\
\frac{1}{m}\big({-2\rho}\,c(\xi)\,u_{.i}(t_i,\xi)\,\dot{u}{^2}_{.i}(t_i,\xi)+(1-\rho\,u{^2}_{.i}(t_i,\xi))\,c(\xi)\,\partial_t^2{u}_{.i}(t_i,\xi)-k\,\dot{u}_{.i}(t_i,\xi)\big)
\end{multline*}\\[-5ex]
\midrule\\[-10ex]
\begin{multline*}
\Phi_{5.i}(\xi):=\hat{\varphi}^5(M)(0,\hat{s}(t_i,\xi))=\partial_t^4u_{.i}(t_i,\xi)=\frac{1}{m}\big({-2\rho}\,c(\xi)\,\dot{u}{^3}_{.i}(t_i,\xi)\\
-6\rho\,c(\xi)\,u_{.i}(t_i,\xi)\,\dot{u}_{.i}(t_i,\xi)\,\partial_t^2u_{.i}(t_i,\xi)+(1-\rho\,u{^2}_{.i}(t_i,\xi))\,c(\xi)\,\partial_t^3{u}_{.i}(t_i,\xi)-k\,\partial_t^2u_{.i}(t_i,\xi)\big)
\end{multline*}\\[-5ex]
\bottomrule
\end{tabular}
\end{table}
}
\makeatother

\section{A parametric, high-dimensional stochastic problem}\label{appsec2StoPro5Dim}

\subsection{Discretization of the random function space}

The random function space for the parametric, high-dimensional problem described in Section \ref{sec2StoPro5Dim} is discretized here.
This time, however, because there is at least one random variable enclosed in the argument of $\exp(\,\cdot\,)$, deriving a single random tensor (as we did in Appendix \ref{appsec2DerVanderPolOsc}) to represent the entire stochasticity of the system at hand is not feasible.
Instead, we derive below a collection of random tensors for each of the random variables that appear in \eqref{appeq2StoPro5Dim100} separately, namely: $\boldsymbol{k}(t)$, $\boldsymbol{f}(t)$, $\boldsymbol{\mathscr{u}}$ and $\boldsymbol{\mathscr{v}}$.

As in Appendix \ref{appsec2DerVanderPolOsc}, we take $\mathscr{Z}^{[P]}$ to be a finite subspace of $\mathscr{Z}$, and the solution representation for \eqref{appeq2StoPro5Dim100} as that stipulated by \eqref{appeq2VanderPol100}.
Therefore, replacing \eqref{appeq2VanderPol100} into \eqref{appeq2StoPro5Dim100}, and then projecting onto $\mathscr{Z}^{[P]}$ gives
\begin{subequations}\label{appeq2StoPro5Dim300}
\begin{align}
\ddot{u}^i+k\indices{^i_j} u^j=f^i&\qquad\text{on $\mathfrak{T}$}\label{appeq2StoPro5Dim300a}\\
\big\{u^i(0)=\mathscr{u}^i,\,\dot{u}^i(0)=\mathscr{v}^i\big\}&\qquad\text{on $\{0\}$,}\label{appeq2StoPro5Dim300b}
\end{align}
\end{subequations}
where $i,j\in\{0,\ldots,P\}$ (summation sign implied over repeated index $j$), and
\begin{gather*}
k\indices{^i_j}(t)=\langle\Psi_i,k(t,\cdot\,)\,\Psi_j\rangle/\langle\Psi_i,\Psi_i\rangle,\quad
f^i(t)=\langle\Psi_i,f(t,\cdot\,)\rangle/\langle\Psi_i,\Psi_i\rangle,\\
\mathscr{u}^i=\langle\Psi_i,\mathscr{u}\rangle/\langle\Psi_i,\Psi_i\rangle\quad\text{and}\quad
\mathscr{v}^i=\langle\Psi_i,\mathscr{v}\rangle/\langle\Psi_i,\Psi_i\rangle.
\end{gather*}

\begin{remark}
Note that the discretized random tensor associated with $k$ is of type $(1,1)$, i.e.~$\boldsymbol{k}(t)=k\indices{^i_j}(t)\,\Psi_i\otimes\Psi^j:{\mathscr{Z}^{[P]}}'\times\mathscr{Z}^{[P]}\to\mathbb{R}$ (assumed to exist), whereas those associated with $f$, $\mathscr{u}$ and $\mathscr{v}$ are of type $(1,0)$, i.e.~$\boldsymbol{f}(t)=f^i(t)\,\Psi_i,\boldsymbol{\mathscr{u}}=\mathscr{u}^i\,\Psi_i,\boldsymbol{\mathscr{v}}=\mathscr{v}^i\,\Psi_i:{\mathscr{Z}^{[P]}}'\to\mathbb{R}$.
Because of this, the last three random tensors can also be regarded as random vectors in $\mathscr{Z}^{[P]}$, and so we have the following identification: $\boldsymbol{f}(t)\mapsto f(t,\cdot\,)$, $\boldsymbol{\mathscr{u}}\mapsto\mathscr{u}$ and $\boldsymbol{\mathscr{v}}\mapsto\mathscr{v}$.
\end{remark}

\subsection{Random basis for use in the FSC scheme}

Table \ref{apptab2StoPro5Dim1000} presents the non-orthogonalized version of the random basis that we use in the simulations with FSC over the region $\mathfrak{R}_i=\mathrm{cl}(\mathfrak{T}_i)\times\Xi$.

\makeatletter
\def\tabcaption{Definition of $\tau$-functions for a single-degree-of-freedom system subjected to free vibration with $k\sim\mathrm{Uniform}$ in $[k_a,k_b]$}
\def\tablabel{apptab1SDOF1000}
\@ifundefined{puformat}%
{
\begin{table}
\centering\small
\caption{Non-orthogonalized version of the random basis used for the $d$-dimensional stochastic problem (Section \ref{sec2StoPro5Dim})}
\label{apptab2StoPro5Dim1000}
\begin{tabular}{@{}p{\textwidth}@{}}
\toprule\\[-4.5ex]
\begin{equation*}
\Phi_{0.i}(\xi):=1
\end{equation*}\\[-3ex]
\midrule\\[-4.5ex]
\begin{equation*}
\Phi_{1.i}(\xi):=\hat{\varphi}^1(M)(0,\hat{s}(t_i,\xi))= u_{.i}(t_i,\xi)=u_{.i-1}(t_i,\xi)
\end{equation*}\\[-3ex]
\midrule\\[-4.5ex]
\begin{equation*}
\Phi_{2.i}(\xi):=\hat{\varphi}^2(M)(0,\hat{s}(t_i,\xi))= \dot{u}_{.i}(t_i,\xi)=\dot{u}_{.i-1}(t_i,\xi)
\end{equation*}\\[-3ex]
\midrule\\[-4.5ex]
\begin{equation*}
\Phi_{3.i}(\xi):=\hat{\varphi}^3(M)(0,\hat{s}(t_i,\xi))=\partial_t^2u_{.i}(t_i,\xi)=f(\xi)-k(t_i,\xi)\,u_{.i}(t_i,\xi)
\end{equation*}\\[-3ex]
\midrule\\[-4.5ex]
\begin{equation*}
\Phi_{4.i}(\xi):=\hat{\varphi}^4(M)(0,\hat{s}(t_i,\xi))=\partial_t^3u_{.i}(t_i,\xi)=\partial_t f(t_i,\xi)-\partial_t k(t_i,\xi)\,u_{.i}(t_i,\xi)-k(t_i,\xi)\,\dot{u}_{.i}(t_i,\xi)
\end{equation*}\\[-3ex]
\midrule\\[-7.5ex]
\begin{multline*}
\Phi_{5.i}(\xi):=\hat{\varphi}^5(M)(0,\hat{s}(t_i,\xi))=\partial_t^4u_{.i}(t_i,\xi)=\partial_t^2 f(t_i,\xi)-\partial_t^2 k(t_i,\xi)\,u_{.i}(t_i,\xi)\\
-2\,\partial_t k(t_i,\xi)\,\dot{u}_{.i}(t_i,\xi)-k(t_i,\xi)\,\partial_t^2 u_{.i}(t_i,\xi)
\end{multline*}\\[-3ex]
\bottomrule
\end{tabular}
\end{table}
}%
{
\begin{table}[p!]
\centering\small
\caption{Non-orthogonalized version of the random basis used for the $d$-dimensional stochastic problem (Section \ref{sec2StoPro5Dim})}
\label{apptab2StoPro5Dim1000}
\begin{tabular}{@{}p{\textwidth}@{}}
\toprule\\[-8.5ex]
\begin{equation*}
\Phi_{0.i}(\xi):=1
\end{equation*}\\[-5ex]
\midrule\\[-8.5ex]
\begin{equation*}
\Phi_{1.i}(\xi):=\hat{\varphi}^1(M)(0,\hat{s}(t_i,\xi))= u_{.i}(t_i,\xi)=u_{.i-1}(t_i,\xi)
\end{equation*}\\[-5ex]
\midrule\\[-8.5ex]
\begin{equation*}
\Phi_{2.i}(\xi):=\hat{\varphi}^2(M)(0,\hat{s}(t_i,\xi))= \dot{u}_{.i}(t_i,\xi)=\dot{u}_{.i-1}(t_i,\xi)
\end{equation*}\\[-5ex]
\midrule\\[-8.5ex]
\begin{equation*}
\Phi_{3.i}(\xi):=\hat{\varphi}^3(M)(0,\hat{s}(t_i,\xi))=\partial_t^2u_{.i}(t_i,\xi)=f(\xi)-k(t_i,\xi)\,u_{.i}(t_i,\xi)
\end{equation*}\\[-5ex]
\midrule\\[-10ex]
\begin{equation*}
\Phi_{4.i}(\xi):=\hat{\varphi}^4(M)(0,\hat{s}(t_i,\xi))=\partial_t^3u_{.i}(t_i,\xi)=\partial_t f(t_i,\xi)-\partial_t k(t_i,\xi)\,u_{.i}(t_i,\xi)-k(t_i,\xi)\,\dot{u}_{.i}(t_i,\xi)
\end{equation*}\\[-5ex]
\midrule\\[-11ex]
\begin{multline*}
\Phi_{5.i}(\xi):=\hat{\varphi}^5(M)(0,\hat{s}(t_i,\xi))=\partial_t^4u_{.i}(t_i,\xi)=\partial_t^2 f(t_i,\xi)-\partial_t^2 k(t_i,\xi)\,u_{.i}(t_i,\xi)\\
-2\,\partial_t k(t_i,\xi)\,\dot{u}_{.i}(t_i,\xi)-k(t_i,\xi)\,\partial_t^2 u_{.i}(t_i,\xi)
\end{multline*}\\[-5ex]
\bottomrule
\end{tabular}
\end{table}
}
\makeatother

\section{Time-complexity analysis for Theorem \ref{thm2paper100} and the traditional Gram-Schmidt process}\label{appsec2TimComAnaThe1}

In this last section we derive the number of elementary operations\footnote{By elementary operations we mean: addition, subtraction, multiplication and division on the real numbers.} that we need to perform in order to obtain an orthogonal basis from a set of linearly independent random functions.
We do this for Theorem \ref{thm2paper100} and the traditional Gram-Schmidt process for sake of comparison.
For both analyses, we assume right away that $\Psi_0\equiv1$ and $\langle\Psi_0,\Psi_0\rangle=1$.
Therefore, for the purpose of this section, the basis to orthogonalize takes the form: $\{\Phi_j\}_{j=1}^P$, and the resulting orthogonalized basis the form: $\{\Psi_j\}_{j=0}^P$ with $\Psi_0\equiv1$.

Below, $Q$ denotes the number of quadrature points used to estimate the inner products with a quadrature rule such as the Gaussian quadrature rule.

\paragraph{Theorem \ref{thm2paper100}}

In Table \ref{apptab2TimComAnaThe11100} we have disaggregated the number of elementary operations into the different steps involved in Theorem \ref{thm2paper100}.
If we assume that both the expectation vector and the covariance matrix are known beforehand, the number of operations needed in Theorem \ref{thm2paper100} is
\begin{equation}\label{appeq2TimComAnaThe1100}
N_\mathrm{op}=\sum_{i=3}^5 N_{\mathrm{op},i}=
\begin{cases}
0 & \text{for $P=0$}\\
P(P+1)Q+\tfrac{1}{30}P^5+\tfrac{1}{6}P^4-\tfrac{1}{3}P^3+\tfrac{1}{3}P^2-\tfrac{6}{5}P+1 & \text{for $P\geq1$}.
\end{cases}
\end{equation}
Otherwise, the number of operations needed is
\begin{equation}\label{appeq2TimComAnaThe1110}
N_\mathrm{op}=\sum_{i=1}^5 N_{\mathrm{op},i}=
\begin{cases}
0 & \text{for $P=0$}\\
\tfrac{7}{2}P(P+\tfrac{11}{7})Q+\tfrac{1}{30}P^5+\tfrac{1}{6}P^4-\tfrac{1}{3}P^3-\tfrac{1}{6}P^2-\tfrac{27}{10}P+1 & \text{for $P\geq1$}.
\end{cases}
\end{equation}

\makeatletter
\def\tabcaption{Number of elementary operations involved in each step of Theorem \ref{thm2paper100} to get $\{\Phi_j\}_{j=1}^P\mapsto\{\Psi_j\}_{j=0}^P$}
\def\tabpath{fig2/tcaTheorem1.pdf}
\def\tablabel{apptab2TimComAnaThe11100}

\@ifundefined{puformat}%
{
\begin{table}
\centering
\caption{\tabcaption}
\label{\tablabel}
\includegraphics[]{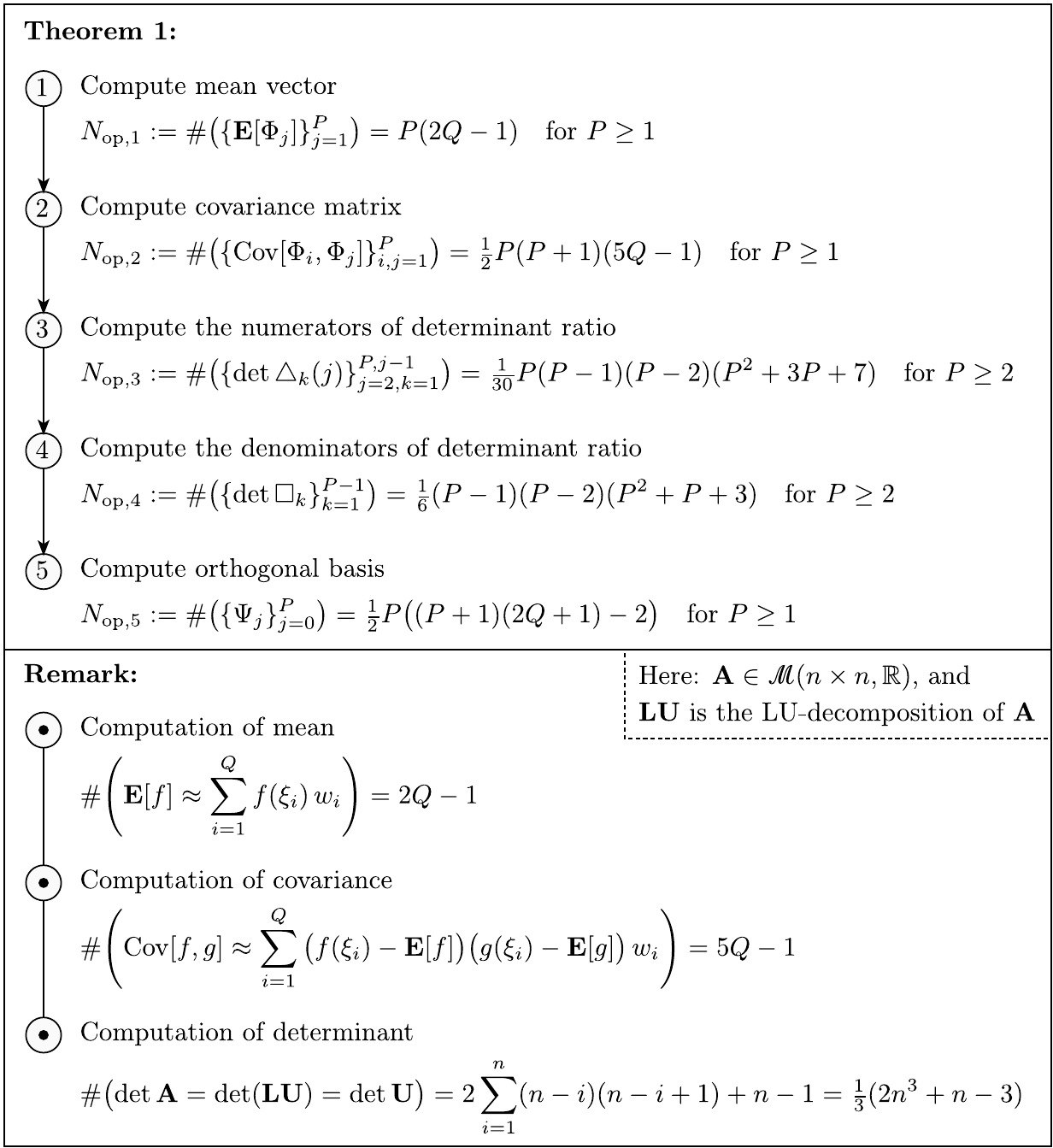}
\end{table}
}%
{
\begin{table}[p!]
\centering
\caption{\tabcaption}
\label{\tablabel}
\includegraphics[scale=1.2]{\tabpath}
\end{table}
}
\makeatother

\paragraph{Traditional Gram-Schmidt process}

As with Theorem \ref{thm2paper100}, in Table \ref{apptab2TimComAnaThe11200} we have disaggregated into steps the number of elementary operations needed in the traditional Gram-Schmidt process (from \eqref{eq2GraSch1200}), yielding
\begin{equation}\label{appeq2TimComAnaThe1200}
N'_\mathrm{op}=\sum_{i=1}^3 N'_{\mathrm{op},i}=
\begin{cases}
0 & \text{for $P=0$}\\
\tfrac{5}{2}(P^2+\tfrac{9}{5}P-\tfrac{6}{5})Q-2P+1 & \text{for $P\geq1$}.
\end{cases}
\end{equation}

\makeatletter
\def\tabcaption{Number of elementary operations involved in each step of the traditional Gram-Schmidt process to get $\{\Phi_j\}_{j=1}^P\mapsto\{\Psi_j\}_{j=0}^P$}
\def\tabpath{fig2/tcaTraditionalGramSchmidtProcess.pdf}
\def\tablabel{apptab2TimComAnaThe11200}

\@ifundefined{puformat}%
{
\begin{table}
\centering
\caption{\tabcaption}
\label{\tablabel}
\includegraphics[]{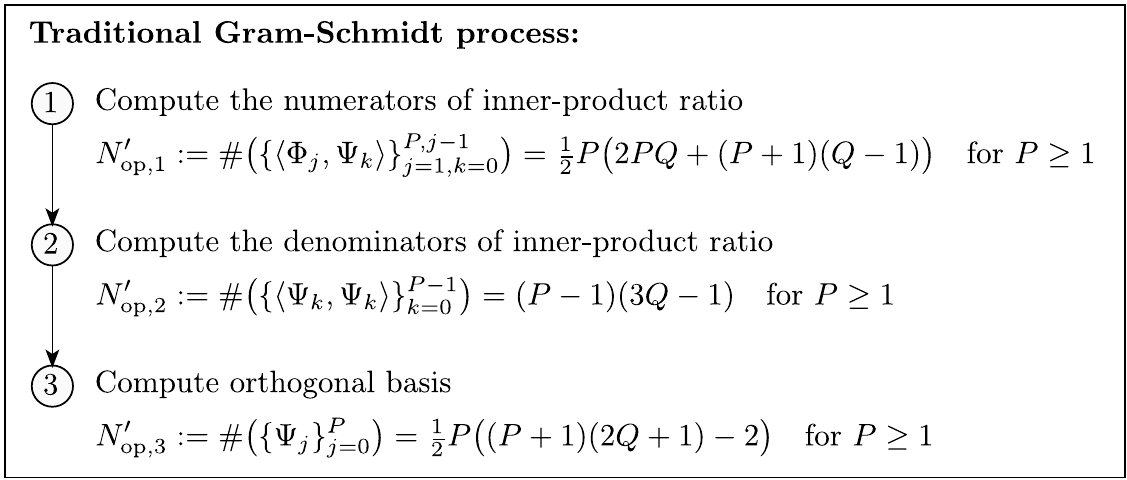}
\end{table}
}%
{
\begin{table}[p!]
\centering
\caption{\tabcaption}
\label{\tablabel}
\includegraphics[scale=1.2]{\tabpath}
\end{table}
}
\makeatother

\paragraph{Performance comparison} 

Combining the result obtained in \eqref{appeq2TimComAnaThe1100} with that in \eqref{appeq2TimComAnaThe1200} produces the plot depicted in Fig.~\ref{fig2TimeComplexityAnalysis}.
From this figure we can see that as the number of quadrature points increases, the time complexity of Theorem \ref{thm2paper100} decreases relative to that of the traditional Gram-Schmidt process.
In other words, when the number of quadrature points is sufficiently large, Theorem \ref{thm2paper100} surpasses in efficiency the traditional Gram-Schmidt process.
In particular, if $P=3,4$, Theorem \ref{thm2paper100} is up to 2.75 times faster.
However, if the expectation vector or the covariance matrix are not known beforehand, Theorem \ref{thm2paper100} is slower than the traditional Gram-Schmidt process, because the $Q$-coefficient in \eqref{appeq2TimComAnaThe1110} is always larger than that in \eqref{appeq2TimComAnaThe1200} for all values of $P$.

\makeatletter
\def\figcaption{Time-complexity analysis for Theorem \ref{thm2paper100} and the traditional Gram-Schmidt process}
\def\figpath{fig2/TimeComplexityAnalysis.pdf}
\def\figlabel{fig2TimeComplexityAnalysis}
\@ifundefined{puformat}%
{
\begin{figure}
\centering
\includegraphics[width=0.495\textwidth]{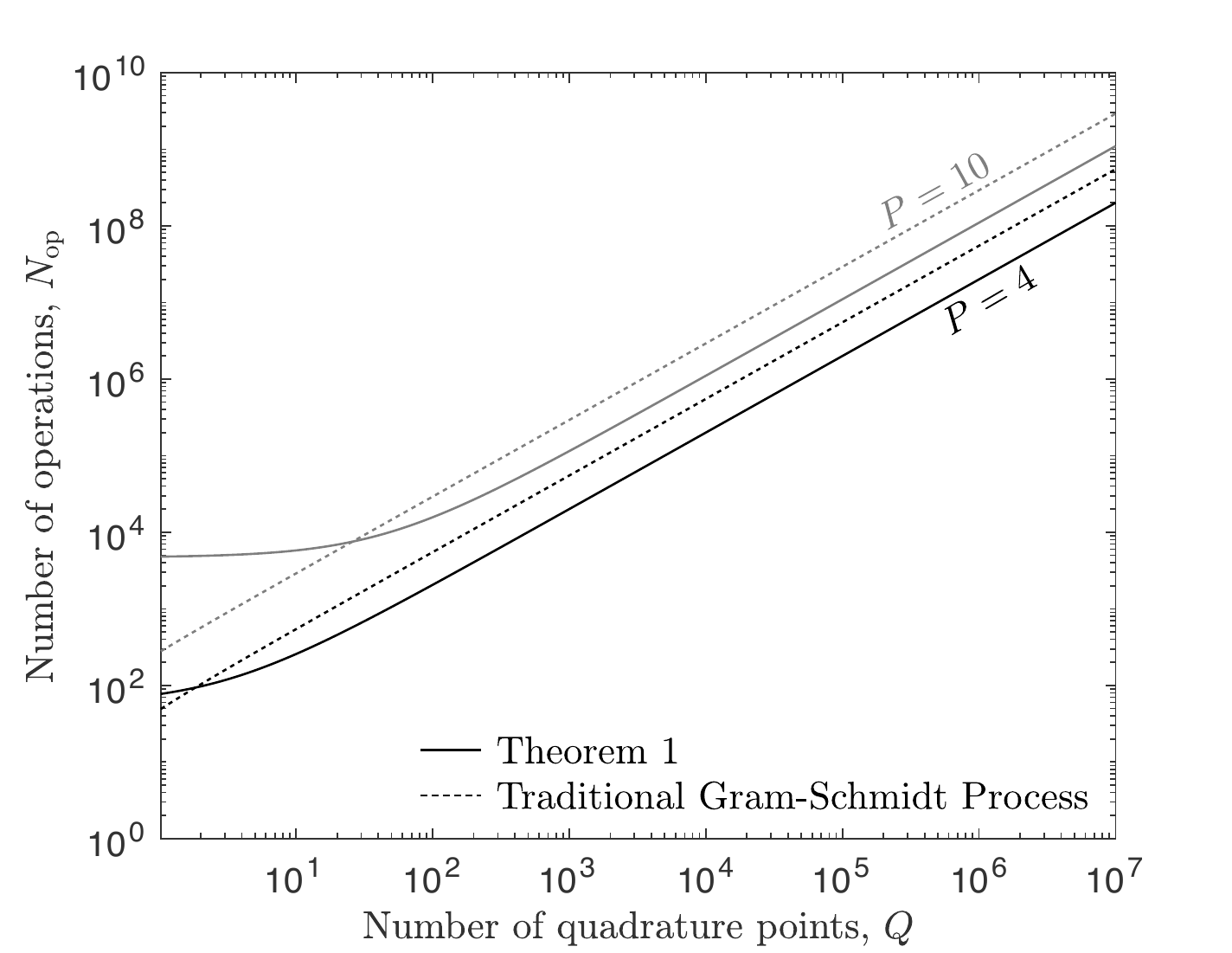}
\caption{\figcaption}
\label{\figlabel}
\end{figure}
}%
{
\begin{figure}[p!]
\centering
\includegraphics[width=0.594\textwidth]{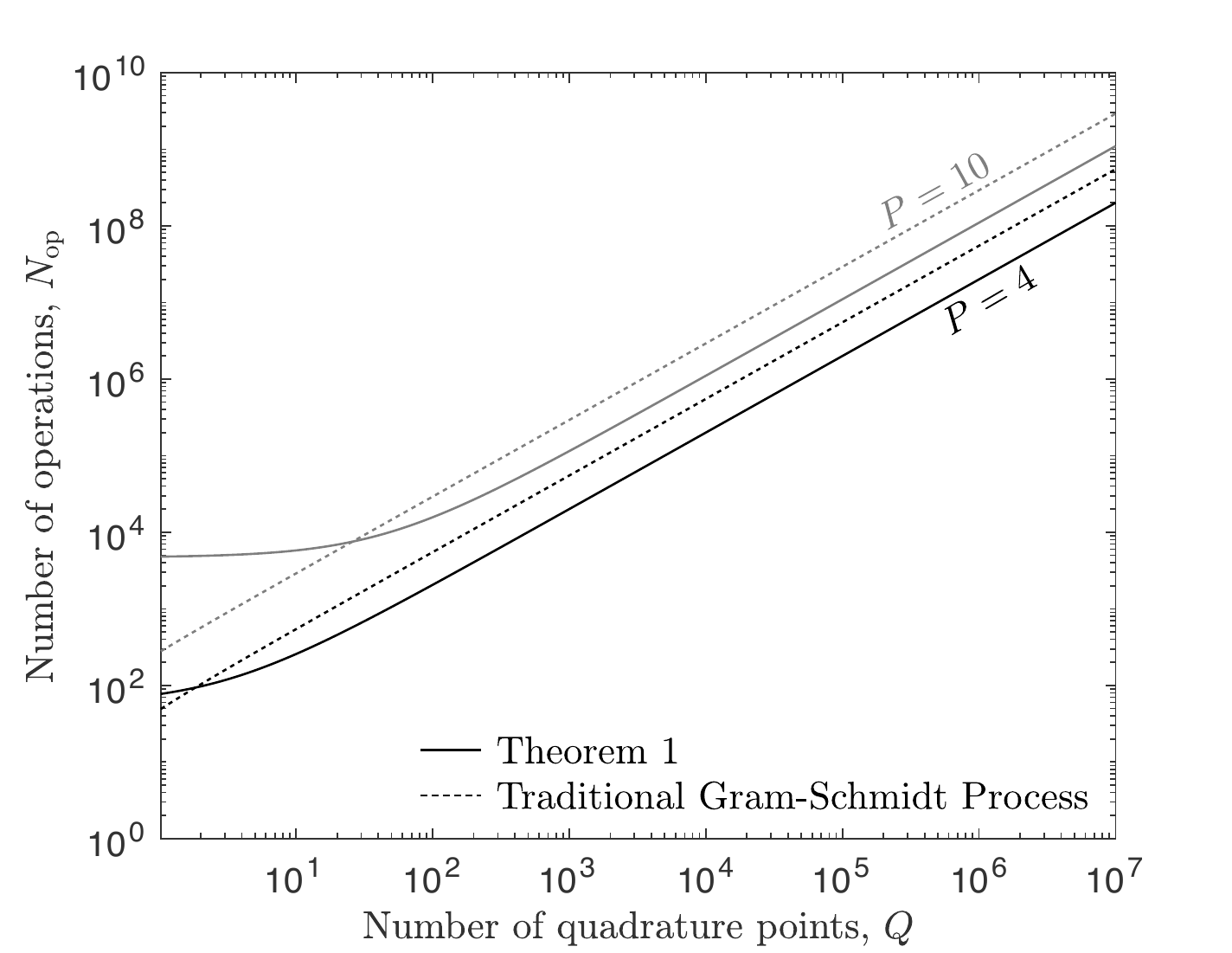}
\caption{\figcaption}
\label{\figlabel}
\end{figure}
}
\makeatother

\end{appendices}

\section*{Acknowledgements}
The first author gratefully acknowledges the scholarship granted by Colciencias and Atl\'antico Department (Colombia) under Call 673 to pursue a Ph.D.~degree in Structural Engineering at Purdue University, West Lafayette, IN.
The authors also gratefully acknowledge the support of the National Science Foundation (CNS-1136075, DMS-1555072, DMS-1736364, CMMI-1634832, and CMMI-1560834), Brookhaven National Laboratory subcontract 382247, ARO/MURI grant W911NF-15-1-0562, and U.S.~Department of Energy (DOE) Office of Science Advanced Scientific Computing Research program DE-SC0021142.

\section*{Credit author statement}
{\bf Hugo Esquivel:} Conceptualization, Data curation, Formal analysis, Investigation, Methodology, Resources, Software, Validation, Visualization, Writing--`original draft', Writing--`review and editing'. {\bf Arun Prakash:} Funding acquisition, Project administration, Resources, Supervision, Writing--`review and editing'. {\bf Guang Lin:} Funding acquisition, Project administration, Resources, Supervision, Writing--`review and editing'.

\bibliography{references2}

\clearpage

\section*{Errata}
This updated version of the article fixes a typo found on Page 5\footnote{Also on Page 5 of the journal article.}.
\begin{quote}
Throughout this paper, we assume that the components of the $d$-tuple random variable $\xi=(\xi^1,\ldots,\xi^d)$ are \xout{pairwise} mutually independent and that the random domain $\Xi$ is a hypercube of $d$ dimensions obtained by performing a $d$-fold Cartesian product of intervals $\bar{\Xi}_i:=\xi^i(\Omega)$.
\end{quote}

\end{document}